\documentclass[english]{article}
\usepackage{geometry}
\usepackage{mathtools,dsfont}

\usepackage[T1]{fontenc}
\usepackage{amsmath,amssymb,amsthm,graphicx,bm}
\usepackage{epsfig, fontawesome}
\usepackage[authoryear]{natbib} 
\usepackage{psfrag}
\usepackage{enumerate} \usepackage{setspace}
\usepackage{float} 
\usepackage{color}
\usepackage{array}
\usepackage{microtype, wasysym, comment}
\usepackage{subfigure}
\usepackage{booktabs}
\usepackage{enumitem}
\usepackage[unicode=true,
 bookmarks=false,
 breaklinks=false,pdfborder={0 0 1},colorlinks=false]
 {hyperref}
\hypersetup{
 colorlinks,citecolor=blue,filecolor=blue,linkcolor=blue,urlcolor=blue
}

\usepackage[ruled,vlined]{algorithm2e}

\definecolor{yuting}{RGB}{255,69,0}
\definecolor{yue}{RGB}{199,21,133}

\renewcommand{\a}{\alpha}
\newcommand{\p}{\phi(\alpha)}
\renewcommand{\t}{\bs{\theta}}
\newcommand{\ts}{\taustart}
\newcommand{\normal}{\mathcal{N}}
\newcommand{\lzero}[1]{\|#1\|_0}
\newcommand{\tint}{\widehat{\bs{\theta}}^{\mathsf{Int}}}
\newcommand{\epsilonstar}{\epsilon^\star}

\theoremstyle{plain}

\newtheorem{theo}{Theorem}[section]

\newtheorem{lem}{Lemma}[section]
\newtheorem{prop}{Proposition}[section]
\newtheorem{cor}{Corollary}[section]

\theoremstyle{definition} 

\newtheorem{nota}{Notation}[section]
\newtheorem{de}{Definition}[section]
\newtheorem{exa}{Example}[section]
\newtheorem{as}{Assumption}[section]
\newtheorem{alg}{Algorithm}[section]

\newcommand{\btheo}{\begin{theo}}
\newcommand{\bde}{\begin{de}}
\newcommand{\ble}{\begin{lem}}
\newcommand{\bpr}{\begin{prop}}
\newcommand{\bno}{\begin{nota}}
\newcommand{\bex}{\begin{exa}}
\newcommand{\bcor}{\begin{cor}}
\newcommand{\spro}{\begin{proof}}
\newcommand{\bas}{\begin{as}}
\newcommand{\balg}{\begin{alg}}

\newcommand{\etheo}{\end{theo}}
\newcommand{\ede}{\end{de}}
\newcommand{\ele}{\end{lem}}
\newcommand{\epr}{\end{prop}}
\newcommand{\eno}{\end{nota}}
\newcommand{\eex}{\end{exa}}
\newcommand{\ecor}{\end{cor}}
\newcommand{\fpro}{\end{proof}}
\newcommand{\eas}{\end{as}}
\newcommand{\ealg}{\end{alg}}

\theoremstyle{plain}

\newtheorem{theos}{Theorem}
\newtheorem{props}{Proposition}
\newtheorem{lems}{Lemma}
\newtheorem{cors}{Corollary}

\theoremstyle{definition}
\newtheorem{exas}{Example}
\newtheorem{algs}{Algorithm}
\newtheorem{asss}{Assumption}
\newtheorem{defns}{Definition}

\newcommand{\btheos}{\begin{theos}}
\newcommand{\etheos}{\end{theos}}
\newcommand{\bprops}{\begin{props}}
\newcommand{\eprops}{\end{props}}
\newcommand{\bdes}{\begin{defns}}
\newcommand{\edes}{\end{defns}}
\newcommand{\blems}{\begin{lems}}
\newcommand{\elems}{\end{lems}}
\newcommand{\bcors}{\begin{cors}}
\newcommand{\ecors}{\end{cors}}
\newcommand{\bexs}{\begin{exas}}
\newcommand{\eexs}{\end{exas}}
\newcommand{\balgs}{\begin{algs}}
\newcommand{\ealgs}{\end{algs}}
\newcommand{\bass}{\begin{asss}}
\newcommand{\eass}{\end{asss}}

\newcommand{\ltwo}[1]{\|#1\|_2}

\newcommand{\Risk}{\mathsf{Risk}}
\newcommand{\taustar}{{\tau^\star}}
\newcommand{\taustart}{{\tau^\star_t}}
\newcommand{\alphastar}{{\alpha^\star}}
\newcommand{\alphastart}{{\alpha^\star_t}}
\newcommand{\zetastar}{\zeta^\star}
\newcommand{\nustar}{{\nu^\star}}
\renewcommand{\b}{\sqrt{\delta}\nu}
\newcommand{\bstar}{{\sqrt{\delta}\nu^\star}}
\newcommand{\real}{\ensuremath{\mathbb{R}}}

\newcommand{\thetastar}{\ensuremath{\bm{\theta}^\star}}
\newcommand{\xnew}{\bm{x}_{\text{new}}}
\newcommand{\ynew}{y_{\text{new}}}
\newcommand{\znew}{z_{\text{new}}}

\newcommand{\thetaint}{\widehat{\bs{\theta}}^{\mathsf{Int}}}
\newcommand{\thetainti}{\widehat{\theta}^{\mathsf{Int}}_i}

\newcommand{\residualT}{l_t}

\newcommand{\inprod}[2]{\ensuremath{\langle #1 , \, #2 \rangle}}
\newcommand{\NORMAL}{\mathcal{N}}
\newcommand{\Ind}{\ensuremath{\textbf{I}}}

\newcommand{\goto}{\rightarrow}
\newcommand{\bs}{\boldsymbol}

\newcommand{\defn}{\coloneqq}
\newcommand{\argmin}{\arg\!\min}

\newcommand{\Exs}{\ensuremath{\mathbb{E}}}

\newcommand{\coef}{\bm{\theta}^\star}
\newcommand{\coefi}{\theta^\star_i}
\newcommand{\noise}{\bm{z}}

\newcommand{\sign}{\mathsf{sign}}

\long\def\comment#1{}

\newcommand{\HACKPROOF}{\begin{proof}}
\newcommand{\HACKENDPROOF}{\end{proof}}

\newlength{\widebarargwidth}
\newlength{\widebarargheight}
\newlength{\widebarargdepth}

\makeatletter
\long\def\@makecaption#1#2{
        \vskip 0.8ex
        \setbox\@tempboxa\hbox{\small {\bf #1:} #2}
        \parindent 1.5em  
        \dimen0=\hsize
        \advance\dimen0 by -3em
        \ifdim \wd\@tempboxa >\dimen0
                \hbox to \hsize{
                        \parindent 0em
                        \hfil 
                        \parbox{\dimen0}{\def\baselinestretch{0.96}\small
                                {\bf #1.} #2
                                } 
                        \hfil}
        \else \hbox to \hsize{\hfil \box\@tempboxa \hfil}
        \fi
        }
\makeatother

\theoremstyle{remark}
\newtheorem{remark}{\textbf{Remark}}

\allowdisplaybreaks

\begin{document}

\title{Minimum $\ell_{1}$-norm interpolators: \\ Precise asymptotics and multiple descent}

\author{Yue Li \thanks{Department of Statistics and Data Science, Carnegie Mellon University, Pittsburgh, PA 15213, USA.} 
	\and 
	Yuting Wei \thanks{Department of Statistics and Data Science, The Wharton School, University of Pennsylvania, Philadelphia, PA 19104, USA.}\\
}

\maketitle

\begin{abstract}
An evolving line of machine learning works observe empirical evidence that suggests interpolating estimators --- the ones that achieve zero training error --- may not necessarily be harmful. This paper pursues theoretical understanding for an important type of interpolators: the minimum $\ell_{1}$-norm interpolator, which is motivated by the observation that several learning algorithms favor low $\ell_1$-norm solutions in the over-parameterized regime. Concretely, we consider the noisy sparse regression model under Gaussian design, focusing on linear sparsity and high-dimensional asymptotics (so that both the number of features and the sparsity level scale proportionally with the sample size). 

We observe, and provide rigorous theoretical justification for, a curious \emph{multi-descent} phenomenon; that is, the generalization risk of the minimum $\ell_1$-norm interpolator undergoes multiple (and possibly more than two) phases of descent and ascent as one increases the model capacity. This phenomenon stems from the special structure of the minimum $\ell_1$-norm interpolator as well as the delicate interplay between the over-parameterized ratio and the sparsity,  thus unveiling a fundamental distinction in geometry from the minimum $\ell_2$-norm interpolator. Our finding is built upon an exact characterization of the risk behavior, which is governed by a system of two non-linear equations with two unknowns.
\end{abstract}

\medskip


{
\setcounter{tocdepth}{2}
  \tableofcontents
}

\section{Introduction}

At the core of statistical learning lies the problem of understanding the generalization performance (e.g., out-of-sample errors) of the learning algorithms in use. 
Conventional wisdom in statistics held that including too many covariates when training statistical models can hurt generalization (despite improving training accuracy), due to the undesired over-fit. 
This leads to the classical conclusion that: proper regularization --- through either adding certain penalty functions to the loss function or algorithmic self-regularization --- seems to be critical in achieving the desired accuracy (e.g., \cite{friedman2001elements,wei2019early}). 
However, an evolving line of works in machine learning observes empirical evidence that suggests, to the surprise of many statisticians,  over-parameterization is not necessarily harmful. 
Indeed, many machine learning models (such as random forests or deep neural networks) are trained until the training error vanishes to zero --- meaning that they are able to perfectly interpolate the data --- while still generalizing well  (e.g., \cite{zhang2021understanding,wyner2017explaining,belkin2019reconciling}).
As a key observation to explain this phenomenon, 
many models when trained by gradient type methods (e.g., gradient descent, stochastic gradient descent, AdaBoost) converge to certain minimum norm interpolators, which implicitly favor models with smaller model complexity.

These empirical mysteries inspire a recent flurry of activity towards understanding the generalization properties of various interpolators. 
A dominant fraction of recent efforts, however, concentrated on studying certain minimum $\ell_{2}$-norm interpolators, primarily in the context of linear and/or kernel regression (see, e.g., \cite{liang2020just,mei2019generalization,hastie2019surprises,belkin2020two,bartlett2020benign,bartlett2021deep} and the references therein). 
This was in part due to the existence of closed-form expressions for minimum $\ell_2$-norm interpolators, 
which are particularly handy when determining the statistical risk. 
In contrast, the theoretical underpinnings for minimum $\ell_{1}$-norm interpolators, 
despite growing interest (e.g., \cite{ju2020overfitting,liang2020precise,chinot2021robustness}), remain highly inadequate and considerably more challenging to establish. 
Given that multiple learning algorithms are known to favor low $\ell_1$-norm solutions in the over-parameterized regime (such as \cite{rosset2004boosting,gunasekar2018implicit}), 
understanding the statistical properties of the minimum $\ell_{1}$-norm interpolation plays a pivotal role in unveiling the trade-offs between over-parameterization and generalization, which we seek to explore in this paper.

\subsection{Motivation: a multi-descent phenomenon}

An intriguing empirical phenomenon called ``\emph{double descent}'' 
has recently emerged in the study of over-parameterized learning models   
 \citep{neyshabur2014search,nakkiran2019deep,belkin2018understand,belkin2019reconciling,belkin2021fit}. 
Consider, for example, a risk curve that depicts how the generalization error varies as more parameters are added to the model. 
 Following the classical bias-variance trade-off U-shape curve before entering the interpolation (or over-parameterized) regime \citep{friedman2001elements}, the generalization error of various models descends again as one further increases the number of parameters beyond the interpolation limit. In addition, this double-descent phenomenon is also closely related to a curious observation --- the non-monotonicy of risk as the model capacity grows --- that has attracted much recent attention \citep{viering2019open}.

Aimed at distilling insights that help explain this phenomenon, 
a recent body of works studied the behavior of the minimum $\ell_2$-norm interpolator in the presence of a linear model, 
which solidified the double-descent phenomenon for this interpolator (see, e.g., \cite{mei2019generalization,hastie2019surprises,bartlett2020benign,belkin2020two} and the references therein). 
Moving beyond minimum $\ell_2$-norm interpolators, 
empirical observations have been discussed regarding the minimum $\ell_{1}$-norm interpolator as well; for instance, 
similar double descent was numerically observed in \cite{muthukumar2020harmless}, with heuristic justification provided in \cite{mitra2019understanding} based on statistical physics intuitions. 
Our own numerical experiments uncover even more intriguing risk behavior of the minimum $\ell_{1}$-norm interpolator. 
As illustrated in Figure~\ref{fig:sparse-s} and Figure~\ref{fig:sparse}, we observe ``\emph{multiple descent}'' in certain parameter regimes; that is,    
as the model complexity continues to grow, the out-of-sample risk of the minimum $\ell_1$-norm interpolator undergoes multiple phases of increase and decrease, and ultimately becomes non-increasing even as the over-parameterized ratio tends to infinity. 
There is, however, lack of theoretical support that elucidates this empirical observation. It remains unclear how to interpret the striking distinction in the risk behavior between the minimum $\ell_1$-norm and the minimum $\ell_2$-norm interpolators. 

\begin{figure}[t]
	\centering
	\includegraphics[width=0.7\linewidth]{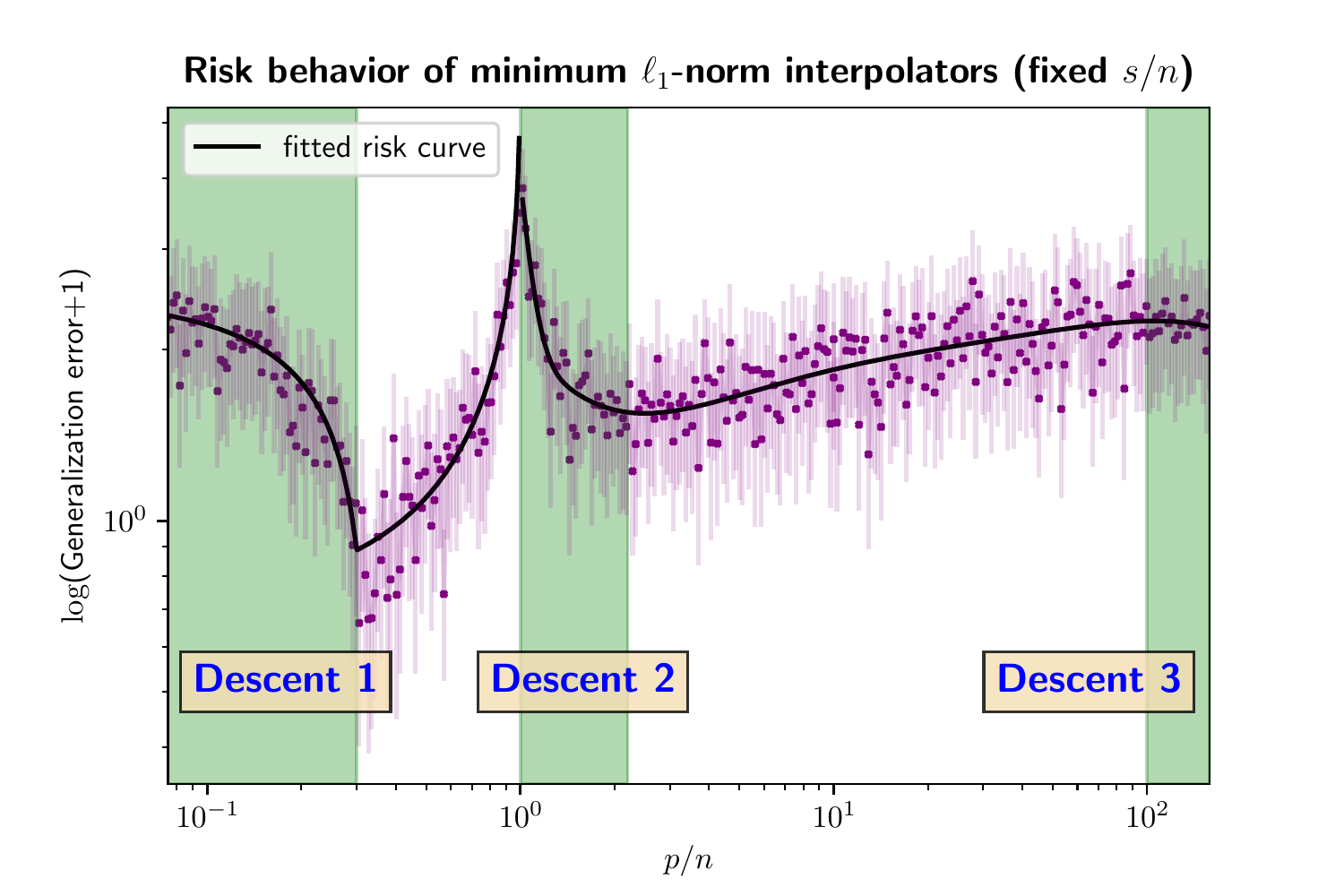}
	\caption{Triple descent in sparse linear regression (see model~\eqref{eq:model}), 
	when the ratio of the sparsity $s$ of the true signal and the sample size $n$ stays fixed.
	More specifically, we fix 
	$s/n = 0.3$ and $s/n \cdot M^2 = 10$ (where $M$ is the magnitude of non-zero entries). When $s \leq p$, the true signal $\thetastar$ is set as an $s$-sparse vector. When $p < s$, we still set the true signal $\thetastar$ as an $s$-dimensional vector, while assuming we only have access to a subset of $p$ features. 
	We set the sample size as $n=100$, and choose $500$ values of $p/n$ such that the $\log (p/n)$'s are uniformly spaced over $[-2, 2.2]$. 
    In each run and for each $p/n$ ratio, we generate a random instance and compute the minimum $\ell_1$-norm interpolator and its risk.
    We report the average risk and error bar over these 30 independent runs for each $p/n$ ratio. 
	The solid line represents the fitted risk curve: when $p/n < 1$, we use the theoretical risk of least-square estimators in the current setting;  when $p/n > 1$, we employ cubic spline smoothing to fit an empirical risk curve. }
	\label{fig:sparse-s}
\end{figure}

\begin{figure}[t]
	\centering
	\includegraphics[width=0.7\linewidth]{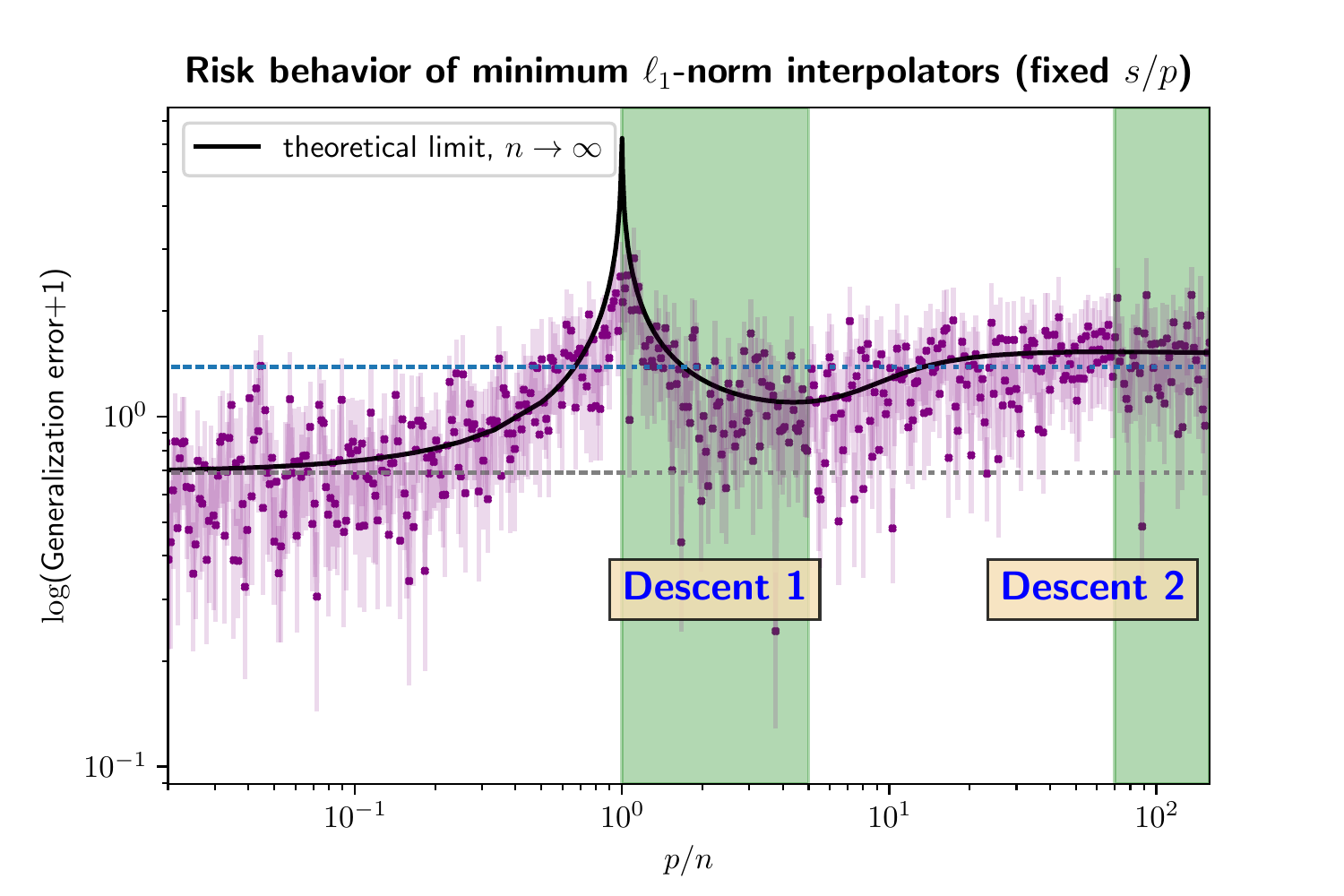}
	\caption{
	Multiple-descent phenomenon observed in numerical experiments. 
 	We generate data from a linear model~\eqref{eq:model} with i.i.d.~Gaussian design, where parameters are set as $\sigma = 1$, SNR$\defn\epsilon M^2= 2$, and sparsity level $\epsilon = 0.01$. The sample size is fixed at $n=100$, and choice of $p$'s and the calculation of error bars are the same as Figure~\ref{fig:sparse-s}.
	 The theoretical curve, predicted by results in the present paper, is shown in solid line, and the $p/n\goto 0$ and $p/n\goto\infty$ limits are shown in dotted lines. 
	 When $p \geq n$, two descending phases are observed here, where  
	 the first descending regime happens at the interpolation point where $n=p$, 
	 which is common for various types of models including the minimum $\ell_2$-norm interpolation. 
	 The second descent appears when $p/n$ is large enough, presenting a unique behavior for $\ell_1$-norm minimization problem.
	}
	\label{fig:sparse}
\end{figure}

\subsection{Main results and insights}

In this paper, we concentrate on linear models, 
and investigate the generalization error (in terms of the out-of-sample squared error) of 
the minimum $\ell_{1}$-norm interpolator --- or equivalently, 
the Lasso estimator with regularization parameter approaching zero. 
We pursue a comprehensive understanding of such estimators in the \emph{proportional growth} and \emph{over-parameterized} regime, where the number of parameters $p$ scales linearly with, but larger than, the number of samples $n$. 
Recognize that Figure~\ref{fig:sparse-s} and Figure~\ref{fig:sparse} (whose difference only lies in how the sparsity levels are chosen) exhibit very similar behavior in the over-parameterized regime. 
To streamline presentation and avoid repetition, we shall restrict our attention to the geometric properties of the risk curve in the setting of Figure~\ref{fig:sparse}.   
In what follows, we formulate the problem precisely, followed by a summary of our main results.

\paragraph{Models.} 

Setting the stage, imagine that we have gathered $n$ i.i.d.~noisy training data drawn from a linear model 
\begin{align}
	y_i = \langle \bm{x}_i, \thetastar \rangle + z_i, \qquad 1\leq i\leq n,
	\label{eq:model}
\end{align}
where $\thetastar\in \mathbb{R}^p$ is a vector composed of $p$ unknowns, 
$\bm{x}_i\in \mathbb{R}^p$ stands for a (random) design vector known {\em a priori}, 
and the $z_i$'s denote i.i.d.~Gaussian noise. 
In addition, we consider the linear sparsity regime, where (i) the unknown signal $\thetastar$ is $(\epsilon \cdot p)$-sparse for some fixed constant $\epsilon>0$, 
and (ii) all the $\epsilon \cdot p$ non-zero entries have magnitudes proportional to some given quantity $M$ (to be made precise momentarily in Section~\ref{sec:model-assumptions}).

Under the well-specified linear model, the generalization error (or out-of-sample risk) of any estimator $\widehat{\bm{\theta}}$ is defined as the expected prediction risk over a new sample data $(\xnew, \ynew)$:\footnote{The expectation is not taken w.r.t.~$\thetastar$ here, which is however not important as the risk will converge almost surely to the expected risk in all cases considered in this paper.} 
\begin{align}
\label{eqn:risk-form}
	\Risk(\widehat{\bs{\theta}}) 
	\defn \Exs\big[(\xnew^\top \widehat{\bs{\theta}} - \ynew)^2 \big] 
\end{align}
where the new sample data follows the same distributions as the training data and is independent of the estimator $\widehat{\bs{\theta}}$.   
Our focal point is the high-dimensional asymptotics (or the large system limit), 
that is, we study the case when $n,p\rightarrow$ with their ratio $n/p$ held fixed.  
For notational convenience, we shall often abbreviate the limiting risk as follows as long as the limit exists almost surely:
\begin{align}
	\label{eqn:risk-form-delta}
	\Risk\big( \widehat{\bs{\theta}}; \delta \big) \coloneqq \lim_{\substack{n/p 
	= \delta \\ n,\, p\to\infty}} \Risk\big( \widehat{\bs{\theta}} \big) .
\end{align}

\paragraph{Main findings: the risk curve of the minimum $\ell_{1}$-norm interpolator.}

When it comes to the over-parametrized regime where $p > n$, the system of equations $y_i = \langle \bm{x}_i, \bm \theta\rangle$, $1\leq i\leq n$ is under-determined, thus implying the existence of multiple regression parameters $\bm \theta$ 
that interpolate the training data perfectly. 
Among all possible interpolators, the focal point of this paper is the {\em minimum $\ell_{1}$-norm intepolator}, 
which enjoys the smallest $\ell_1$-norm as defined below
\begin{align}
	\label{eqn:sparse-intepolation}
	\widehat{\bs{\theta}}^{\mathsf{Int}} \defn \argmin_{\bs{\theta}\in\mathbb{R}^p}\|\bs{\theta}\|_1 \qquad 
	\text{subject to} \quad y_i = \inprod{\bs{x}_i}{\bs{\theta}},~~ 1\leq i\leq n.
\end{align}
In an attempt to understand its generalization behavior, we seek to pin down the exact asymptotics of the above risk metric.  
Encouragingly, the large system limit of $\Risk(\widehat{\bs{\theta}}^{\mathsf{Int}})$ can be accurately pinpointed  
 by solving a system of two nonlinear equations with two unknowns (to be formalized in Theorem~\ref{thm:interpolation-risk}).  
In turn,  such risk characterizations provide a rigorous footing for the multi-descent behavior numerically observed in Figure~\ref{fig:sparse}, as asserted by the following theorem. 

\begin{theos}[Shape of the risk curve]
\label{thm:interpolation}
	Suppose that $0<\delta  <1$, and fix $n=\delta p$. 
	Assume i.i.d.~Gaussian design, i.i.d.~Gaussian noise, and linear sparsity 
	(to be made precise in Section~\ref{sec:model-assumptions}).  
	Then the generalization error (cf.~\eqref{eqn:risk-form-delta}) of the minimum $\ell_1$-norm interpolator \eqref{eqn:sparse-intepolation} satisfies the following properties: 
	\begin{enumerate}

		\item[(a)] There exist two constants $1<\eta_1< \eta_2 <\infty$ such that $\Risk(\widehat{\bs{\theta}}^{\mathsf{Int}}; \delta)$ decreases with $p/n$ within the range $p/n \in (1, \eta_1) \cup (\eta_2,\infty)$. 

		\item[(b)] $\Risk(\widehat{\bs{\theta}}^{\mathsf{Int}};\delta)$ approaches the risk of the zero estimator (i.e., $\Risk(\bm{0})$) 
			as $ p/n$ tends to infinity.

		\item[(c)] For any fixed signal-to-noise ratio (to be defined precisely in \eqref{eq:defn-SNR}), there exists a constant $\epsilon^*>0$ such that 
  if the sparsity ratio $\epsilon$ obeys $\epsilon < \epsilon^*$, then one can find a region within the range $p/n \in (\eta_1,\eta_2)$ such that $\Risk(\widehat{\bs{\theta}}^{\mathsf{Int}}; \delta)$ increases with $p/n$. 

  \item [(d)] In addition, for every given $\delta$, there exists a threshold
	  $\tilde{\epsilon}(\delta)$ such that 
  $\Risk(\widehat{\bs{\theta}}^{\mathsf{Int}};\delta)$ decreases with $p/n$ at this particular point $\delta$ as long as the sparsity ratio $\epsilon$ satisfies
   $\epsilon \leq \tilde{\epsilon}(\delta).$

	\end{enumerate}
\end{theos}

\paragraph{Geometric implications and insights.}  

Theorem~\ref{thm:interpolation} reveals certain geometric properties of the risk curve of the minimum $\ell_1$-norm estimator in the over-parameterized regime  (i.e., $p>n$). 
Let us take a moment to discuss the implications regarding how the risk changes in the over-parameterized ratio $p/n$. 
\begin{itemize}

	\item Theorem~\ref{thm:interpolation}(a) identifies two {\em non-overlapping} regions within the over-parameterized regime that exhibit risk descent. Consequently, the total number of descent depends largely on the risk behavior in between these two regions.

	\item Theorem~\ref{thm:interpolation}(c) indicates that the risk in between the above-mentioned two regions exhibits contrastingly different behavior depending on the sparsity ratio. 

	\begin{itemize}	

	\item 
When the sparsity ratio is relatively small, the generalization error exhibits an intriguing ``\emph{decreasing -- increasing -- decreasing again}'' pattern in the over-parameterized regime. 
This taken together with what happens in the under-parameterized regime unveils a \emph{``triple-descent''} behavior, 
which matches the numerical findings in Figure~\ref{fig:lasso}. 
Interestingly, the minimum $\ell_{2}$-norm interpolator for such a model often enjoys a double-descent behavior rather than triple descent, thus uncovering a fundamental difference between the minimum $\ell_1$-norm and the minimum $\ell_2$-norm interpolation. 

	\item

	In stark contrast, 
	if the sparsity ratio is relatively large (in the sense that $\epsilon > \epsilon^\star$), the generalization error might actually decrease monotonically with $p/n$ in the entire over-parameterized regime. 
If this were true, then taking it together with the classical conclusion in the under-parameterized regime 
would justify the double-descent behavior that has also been empirically observed in \cite{muthukumar2020harmless,mitra2019understanding}.

\end{itemize}

\item In view of Theorem~\ref{thm:interpolation}(b), the minimum $\ell_1$-norm interpolator is essentially no better than a trivial estimator (i.e., the zero estimator) when the over-parameterized ratio $p/n$ is overly large.

\item Finally, Theorem~\ref{thm:interpolation}(d) reveals that at any over-parameterized ratio, the generalization risk can be decreasing with $p/n$ as long as the sparsity ratio is small enough. 

\end{itemize}

\begin{figure}
	\centering
	\includegraphics[width=0.7\linewidth]{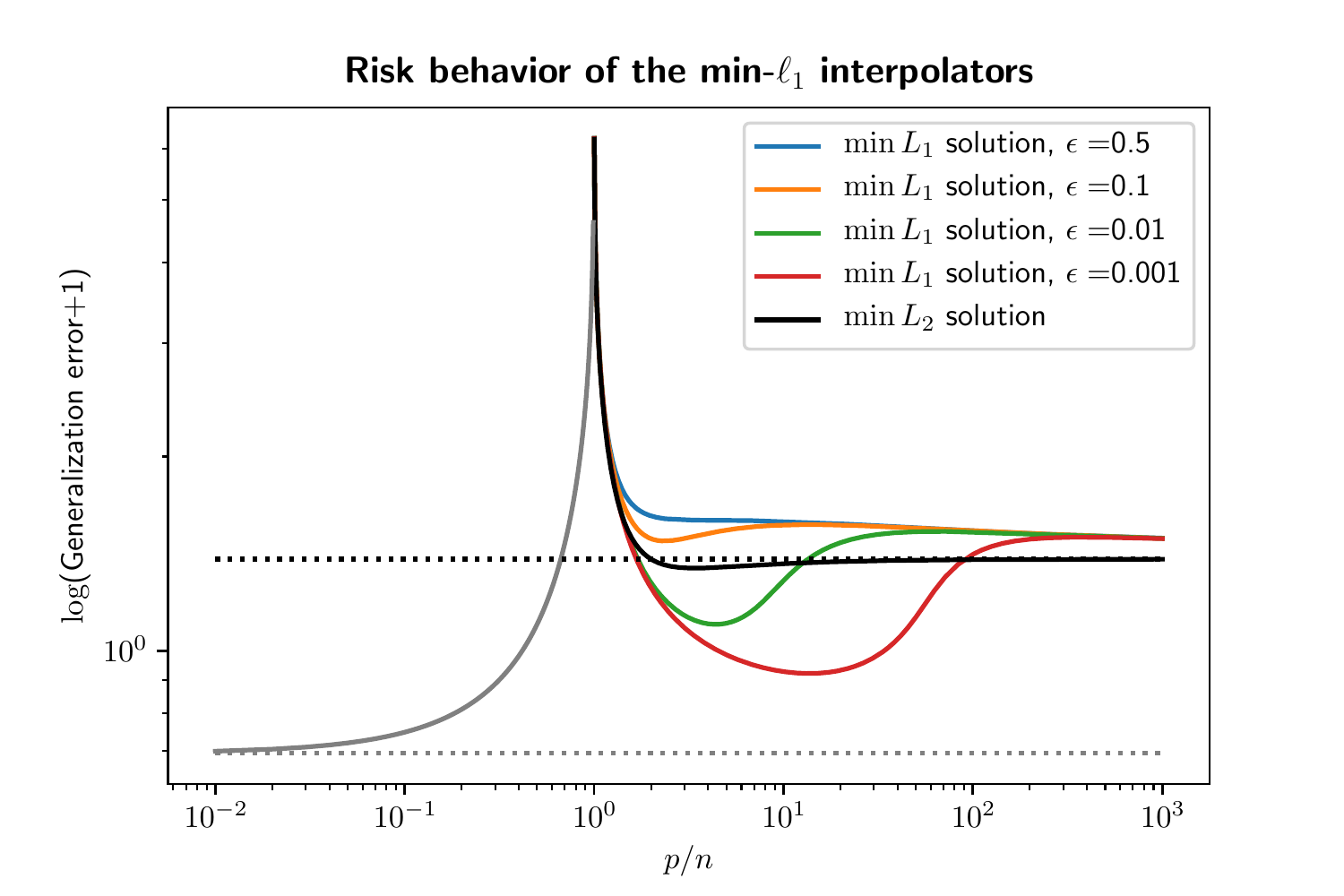}
	\caption{
	Theoretical risk curves for the minimum $\ell_1$-norm interpolation (obtained by solving the system of equations~\eqref{eqn:fix-eqn}).  	Here, we set SNR $ = 2$, and consider different values of the sparsity ratio $\epsilon$. 
	When $p/n < 1$,  the risk curves share similar behavior as the ordinary least square estimator.
	When $p/n>1$ and when $\epsilon$ drops below a certain level, the risk curves present one more descent phase. 
	In contrast, the minimum $\ell_2$-norm interpolation curve exhibits just one descent phase in the $p>n$ regime in all cases, as plotted in black. 
	}
	\label{fig:lasso}
\end{figure}

\subsection{A glimpse of our technical approach and novelty} 

Demonstrating the multi-descent phenomenon requires understanding the asymptotic risk of the interpolator of interest, 
which can be achieved by analyzing an iterative algorithm called  \emph{Approximate Message Passing} (AMP), 
originally proposed by \cite{donoho2009message} in the context of compressed sensing. 
Most relevant to our paper is the series of papers by \cite{bayati2011dynamics,bayati2011lasso} 
that determined  the asymptotic Lasso risk with a fixed and \emph{strictly} positive regularization. 
In order to analyze the minimum $\ell_1$-norm interpolator, 
the present work extends the AMP machinery to accommodate Lasso with the regularization parameter approaching zero, 
which can be accomplished by running a sequence of AMP that changes the algorithm parameters in an epoch-based manner.  
Noteworthily, 
previous analyses relied on the observation that having positive regularization encourages sparse solutions and, in turn, induces certainty restricted strong convexity around the solution. 
This, however, fails to capture our AMP dynamics due to the absence of positive regularization. 
To remedy this issue, we develop a new type of structural properties that allows one to analyze AMP iterates with changing parameters.  
As it turns out, the minimum $\ell_{1}$-norm solution coincides with the fixed-point of the new AMP updates, whose risk behavior can be characterized by a new system of two nonlinear equations with two unknowns. 
Obtaining the exact characterization of the minimum $\ell_{1}$-norm is beyond what prior AMP theory has to offer.

With the risk characterization in place, everything boils down to analyzing the above-mentioned nonlinear systems of equations --- in particular, how its solutions vary with the aspect ratio $\delta$. 
This, however, is challenging to cope with, as there is no closed-form expression of the solution points. 
While the prior work \cite{miolane2018distribution} studied the existence and uniqueness of the state evolution solutions,  
it is unclear how the solution varies with $\delta$, particularly in the absence of strictly positive regularization.  
All this is addressed in the present paper via careful analysis of the first- and second-order properties of the system of equations, 
which constitutes much of our analysis. We expect our analysis idea to be useful to analyze other estimators such as more general M-estimators \citep{donoho2016high} or the SLOPE estimator \citep{su2016slope}.


\subsection{Other related works}

\paragraph{Multiple descent.}
While the emergence of the multi-descent phenomenon in our setting is caused by the special structure of minimum $\ell_{1}$-norm interpolators as well as the interplay between the over-parameterized ratio and the sparsity level, 
this phenomenon has also been observed in other settings for $\ell_{2}$-norm minimization --- albeit of different nature compared to ours. 
It is noteworthy that the presence of multiple descent can be caused by various other structures of the design matrix. 
As concrete examples, this might arise in non-isotropic linear regression where the covariance of the design matrix possesses two eigenspaces of different variance \citep{nakkiran2020optimal}; another possibility that leads to this phenomenon is to tweak the change points of the risk curve of the minimum $\ell_{2}$-norm solution by carefully adding new columns (features) to the design matrix (with either standard Gaussian or Gaussian mixtures distributions) \citep{chen2020multiple}. 
Additionally, this phenomenon might also stem from the regression kernel in use. For instance, \citet{adlam2020neural} derived the high-dimensional asymptotics for the risk curve when using neural tangent kernels in a two-layer neural network; 
\cite{liang2020multiple} studied the convergence properties of the minimum kernel-Hilbert norm interpolators under various scaling of $p = n^{\alpha}$, $\alpha \in (0,1)$, and suggested possible change points (from ascent to descent) at $\alpha = \frac{1}{l + 1/2}$ for every integer $l.$
In addition, \cite{d2020triple} empirically observed the multi-descent phenomenon under the random Fourier feature model.

\paragraph{Minimum $\ell_1$-norm solutions.} 

In the over-parametrized regime ($p>n$), 
the minimum $\ell_1$-norm interpolator considered herein is closely related to the problem of Basis Pursuit (BP) in the compressed sensing literature (e.g., \cite{chen2001atomic,wojtaszczyk2010stability,candes2006near,donoho2006compressed,donoho2005stable}).   
In particular, the algorithm \eqref{eqn:sparse-intepolation} has been well-established paradigm for finding a sparse solution to a noiseless linear system. 
When it comes to the proportional growth and linear sparsity regime in the noiseless case,    
\cite{donoho2009counting,amelunxen2014living} characterized the exact phase transition boundary regarding the sample size in achieving perfect recovery. Moving to the noisy scenario, \cite{ju2020overfitting} considered the same estimator and exhibited a double-descent phenomenon when $p$ is exponentially larger than $n$.
\cite{chinot2021robustness} studied the setting with $p$ exceeding the order of $n \log^{1-\beta}(n)$ for some constant $\beta\in (0,1)$, which did not focus on determining exact pre-constants and the double- or multi-descent phenomenon.  
Another recent work \cite{liang2020precise}
studied a drastically different problem --- binary classification, and pinned down exact asymptotics of the minimum $\ell_1$-norm solution when the data are separable, which has intimate connection to AdaBoost.

\paragraph{Exact high-dimensional asymptotics.} 

The exact asymptotic framework adopted in this work is closely related to the risk characterization of the Lasso estimator (for positive $\lambda$) that has been obtained in prior literature. 
In the proportional growth regime (so that $p$ and $n$ are comparable),  
the Lasso risk under i.i.d.~Gaussian designs has been determined by \cite{bayati2011lasso,stojnic2013framework,oymak2013squared}. 
In particular, the AMP machinery is a powerful tool for determining exact asymptotics in this regime, and we postpone further discussions to Section~\ref{sec:main-proof-AMP}. 
The distributional characterization of the Lasso has been recently established by \cite{miolane2018distribution} under the i.i.d.~Gaussian designs, and by \cite{celentano2020lasso,bellec2019second} under general correlated Gaussian designs, where the first two works were built upon the convex Gaussian min-max theorem.  
Going beyond the $\ell_1$-penalty, the estimation risk of the robust regression estimators was pioneered by \citet{karoui2013asymptotic,el2018impact,donoho2016high} and extensively studied by, e.g., \citet{dobriban2018high,thrampoulidis2018precise,hastie2019surprises,patil2021uniform}.

\paragraph{Approximate message passing.} 
  
Inspired by statistical physics and information theory literature, AMP was first proposed as an efficient scheme to solve compressed sensing problems \citep{donoho2009message}.  
\cite{bayati2011dynamics,javanmard2013state} then rigorously proved that the dynamics of AMP can be accurately tracked by a simple small-dimensional recursive formula called the \emph{state evolution}. 
This state-evolution characterization made AMP amenable as a analysis device to describe the statistical behaviors for various problems, despite that AMP is an effective algorithm on its own.  
The AMP algorithm and machinery has been successfully applied to a variety of problems beyond compressed sensing,
including but not limited to robust M-estimators \citep{donoho2016high}, SLOPE \citep{bu2020algorithmic},  
low-rank matrix estimation and PCA \citep{rangan2012iterative,montanari2021estimation,fan2020approximate,zhong2021approximate}, stochastic block models \citep{deshpande2015asymptotic}, phase retrieval \citep{ma2018optimization}, phase synchronization \citep{celentano2021local}, and generalized linear models \citep{rangan2011generalized,sur2019likelihood,sur2019modern,barbier2019optimal}. 
See \cite{feng2021unifying} for an accessible introduction of this machinery and its applications. 
Moreover, a dominant fraction of the AMP works focused on high-dimensional asymptotics (so that the problem dimension tends to infinity first before the number of iterations), except for \cite{rush2018finite} that derived finite-sample guarantees allowing the number of iterations to grow up to $O(\log n / \log\log n)$.


\subsection{Notation}

Here, we provide a summary of notation to be used throughout the present paper. 
In general, scalars are denoted by lowercase letters, vectors are represented by boldface lowercase letters, while
 matrices are denoted by boldface uppercase letters. 
For every $q \in [1,\infty]$ and any vector $\bs{x} \in \real^{p}$, we use $\|\bs{x}\|_{q} \defn (\sum_{i=1}^{p} |x_{i}|^q)^{1/q}$ to represent the $\ell_{q}$-norm of $\bm{x}$, and  let $\lzero{\bs{x}}$ indicate the number of non-zero coordinates in $\bs{x}$. 
We denote by
\begin{align*}
	\left\langle \bs{x}\right\rangle \defn \frac{1}{p}\sum_{i=1}^p x_i,
	\quad \text{the average of the entries of the vector }
	\bs{x}\in \real^p.
\end{align*}
Additionally, let $\sigma_{\min}(\bs{M})$ and $\sigma_{\max}(\bs{M})$ denote respectively the  minimum and the maximum singular values of a matrix $\bs{M}$. 
Define $[n]\defn \{1,\cdots,n\}$ for an integer $n$.

For two functions $f(\cdot)$ and $g(\cdot)$, we often employ the convenient notation  $f(\delta) \lesssim g(\delta)$ (resp.~$f(\delta) \gtrsim g(\delta)$) to indicate that 
$$
	\lim_{\delta \goto \delta_0} f(\delta)/g(\delta) \leq 1 \qquad  \text{(resp. }\lim_{\delta \goto \delta_0} f(\delta)/g(\delta) \geq 1),
$$
where $\delta_{0}$ is a certain limiting point that will be clear from the context. 
We also write $f(\delta) \sim g(\delta)$ when both $f(\delta) \lesssim g(\delta)$
and $f(\delta) \gtrsim g(\delta)$ hold true. 
In addition, the soft-thresholding function is defined as 
\begin{align}
\label{eqn:soft-thresholding}
	\eta(x;\zeta) \defn (|x| - \zeta)_{+} \, \sign(x)
\end{align}
for any $x\in \mathbb{R}$ and a given threshold $\zeta\in \mathbb{R}^+$, 
where $z_+ \defn \max\{ z, 0\}$.  
Further, we let $\eta'(\cdot; \cdot)$ denote differentiation with respect to the first
variable. 
When a function is applied to a vector, it should be understood as being applied in a component-wise manner. 
Following conventional notation, we denote by $\partial f$ the sub-differential of a function $f$. 
When it comes to the $\ell_1$-norm $\|\cdot\|_1$, its sub-gradient at the point $\bs{x}\in \mathbb{R}^p$ can be any vector $\bs{v}=[v_i]_{1\leq i\leq p}$ satisfying
\begin{align*}
\left\lbrace\begin{matrix}
	v_i  = \sign(x_i), & \text{if }x_i \neq 0;\\[0.2mm]
	v_i \in [-1, 1], & \text{if } x_i = 0.
\end{matrix} \right.
\end{align*}
Moreover, a function $\psi: \mathbb{R}^2 \goto \mathbb{R}$ is said to be pseudo-Lipschitz if there exists a constant $L>0$ such that  
\begin{align}
\label{eqn:pseudo-lipschitz}
|\psi(x)-\psi(y)|\leq L(1+\|x\|_2+\|y\|_2)\|x-y\|_2 
\end{align}
holds for all $x, y\in \mathbb{R}^2$. Additionally, we shall often suppress a.s.~in the notation $\overset{\mathrm{a.s.}}{=}$ for almost sure convergence if it is clear from the context.


\section{Risk characterization for the minimum $\ell_1$-norm interpolator}

\subsection{Modelling assumptions}
\label{sec:model-assumptions}

For notational simplicity, we shall often adopt the vector and matrix notation as follows
\begin{subequations}
\begin{align}
	\bm{z} \defn [z_i]_{1\leq i\leq n} &\in \mathbb{R}^n, 
	\qquad \bm{X} \defn [\bm{x}_1,\cdots,\bm{x}_n]^{\top} \in \mathbb{R}^{n\times p} , \\
	&\bm{y} \defn [y_i]_{1\leq i\leq n} = \bm{X} \bm{\theta}^{\star} + \bm{z}  \in \mathbb{R}^n. 
\end{align}
\end{subequations}
To formalize the problem setting, we first impose the following assumptions on the sampling process throughout the paper.

\begin{itemize}
	\item {\em Gaussian design.} We study i.i.d.~Gaussian design, where each design vector is independently drawn:  
\begin{equation}
	\bm{x}_i \overset{\text{i.i.d.}}{\sim} \normal\bigg(0,\frac{1}{n} \Ind_p \bigg),\qquad 
	1\leq i\leq n.
	\label{eq:Gaussian-design}
\end{equation}
Here, the scaling factor $1/n$ is introduced merely for normalization purpose.  
This tractable model is widely adopted when studying the high-dimensional asymptotics of Lasso (e.g.~\cite{bayati2011lasso,miolane2018distribution,su2017false}) and has been extended to other statistical learning problems 
(e.g., \cite{donoho2016high,karoui2013asymptotic,sur2019likelihood,thrampoulidis2018precise}).  
While Gaussian design is typically not satisfied in practice, 
it allows for useful mathematical insights that might shed light on practical contexts.

	\item {\em Gaussian noise.} It is assumed that the noise components are independent and obey
\begin{align}
	\bm{z} \sim \mathcal{N}(0,\sigma^2 \Ind_n). 
	\label{eq:Gaussian-noise}
\end{align}

\end{itemize}

\noindent
Under the above Gaussian design and Gaussian noise model, the generalization error \eqref{eqn:risk-form} of an estimator $\widehat{\bm{\theta}}$ should be defined when $(\xnew, \ynew)$ is drawn from the same assumption, i.e., $\ynew = \langle \xnew, \bm{\theta}^{\star} \rangle + \znew$ with $\xnew \sim \normal\big(0,\frac{1}{n} \Ind_p \big)$ and $\znew \sim \mathcal{N}(0,\sigma^2)$. This leads to
\begin{align}
\label{def:risk}
	\Risk(\widehat{\bs{\theta}}) 
	\defn \Exs\big[(\xnew^\top \widehat{\bs{\theta}} - \ynew)^2 \big] 
	 = \Exs\big[\big(\xnew^\top(\widehat{\bs{\theta}} - \thetastar)\big)^2\big] + \sigma^{2}
	 = \frac{1}{n} \big\| \widehat{\bs{\theta}} - \thetastar\big\|_2^2 + \sigma^{2}. 
\end{align}

In addition, we shall make assumptions regarding how the ground truth is generated, as formalized below.

\begin{itemize}

	\item {\em Linear sparsity.} Suppose that each coordinate of $\thetastar=[\coefi]_{1\leq i\leq p}$ is identically and independently drawn as follows 
\begin{equation}\label{eq:theta-distribution}
	\coefi \overset{\mathrm{i.i.d.}}{\sim} \epsilon \mathcal{P}_{M\sqrt{\delta}} + (1-\epsilon)\mathcal{P}_0,
\end{equation}
where $\mathcal{P}_{c}$ denotes the Dirac measure at point $c\in \mathbb{R}$, 
and $M>0$ is some given quantity that determines the magnitude of a non-zero entry. 
In words, each coordinate is non-zero (and with magnitude $M\sqrt{\delta}$) with probability $\epsilon$.  
Here, the scaling factor $\sqrt{\delta}$ is introduced solely for notational convenience, 
which ensures that the \emph{signal-to-noise-ratio} (SNR) obeys
\begin{align}
	\mathrm{SNR} \defn \frac{\mathbb{E}\big[ (\bs{x}^\top \thetastar)^{2} \big]}{\sigma^2} 
	= \frac{ \epsilon M^2 }{\sigma^2}.
	\label{eq:defn-SNR}
\end{align}
When $\epsilon$ is a fixed constant, 
the number of non-zero coordinates concentrates around $\epsilon \cdot p$, 
meaning that $\epsilon$ determines the sparsity level of $\thetastar$. 
Noteworthily, a model with linear sparsity lends itself well to high-dimensional applications with only moderate degrees of sparsity (for instance, in various problems in genomics, the relevant signals are observed to be spread out across a good fraction of the genome \citep{boyle2017expanded,tam2019benefits}).

\end{itemize}

\begin{remark}
It is worth noting that the linear sparsity regime often precludes consistency results in both estimation and support recovery, 
which is in stark contrast to the regime where the sparsity level is vanishingly small compared to the sample size \citep{bickel2009simultaneous,wainwright2009sharp,buhlmann2011statistics}. 
In fact, results featuring this regime often require an additional  adjustment due to the effect of undersampling, as discussed in \cite{el2013robust}. 
\end{remark}

\subsection{Risk characterization}

In order to depict the shape of the risk curve, we first characterize 
the precise asymptotics of $\Risk(\widehat{\bs{\theta}}^{\mathsf{Int}}; \delta)$ with a fixed aspect ratio $0<\delta<1$. 
As alluded to previously, this is accomplished by considering sequences of instances of increasing sizes, along which the minimum $\ell_{1}$-norm interpolator (cf.~\eqref{eqn:sparse-intepolation}) has a non-trivial limiting risk behavior.

Towards this end, we consider a more general distribution on $\bm{\theta}^{\star}$ by assuming that
\begin{align}
	\text{The empirical distribution of }\bm{\theta}^{\star}\text{ converges weakly to a probability measure }P_{\Theta}. 
	\label{eq:empirical-Theta-assumptions}
\end{align}
The following theorem determines the exact asymptotics of $\Risk(\widehat{\bs{\theta}}^{\mathsf{Int}}; \delta)$ for any given $0<\delta <1$.

\begin{theos}[Risk of min $\ell_1$-norm interpolation]
\label{thm:interpolation-risk}
Consider the linear model \eqref{eq:model}, and suppose that the assumptions \eqref{eq:Gaussian-design}, \eqref{eq:Gaussian-noise} and 
\eqref{eq:empirical-Theta-assumptions} hold. Consider any given $0 < \delta < 1$. 
If $\mathbb{E}[\Theta^2] < \infty$ and $\mathbb{P}(\Theta \neq 0)>0$,
then the prediction risk of the minimum $\ell_{1}$-norm interpolator obeys
\begin{align}
	\lim_{\substack{n/p 
	= \delta \\ n,\, p\to\infty}} \Risk\big( \widehat{\bs{\theta}}^{\mathsf{Int}} \big)
	~\overset{\mathrm{a.s.}}{=}~ \taustar^2 .
\end{align}
Here, $(\taustar, \alphastar)$ stands for the unique solution to the following system of equations 
\begin{subequations}
	\label{eqn:fix-eqn}
		\begin{align}
		\label{eq:fix-1} \tau^2 & = \sigma^2 + \frac{1}{\delta}\mathbb{E}\left[ \big( \eta(\Theta+\tau Z; \alpha\tau)-\Theta\big) ^2\right] , \\
		\label{eq:fix-2} \delta & =\mathbb{P}\big( |\Theta+\tau Z| > \alpha\tau \big),
		\end{align}
\end{subequations}
where $\Theta \sim P_{\Theta}$, and $Z \sim \normal(0,1)$ and is independent of $\Theta$. 

\end{theos}
\begin{remark}
	It is noteworthy that Theorem~\ref{thm:interpolation-risk} is completely general regarding the distribution of $\coef$ as long as its empirical distribution converges to a fixed measure; in particular, it does not require $\Theta$ to follow the sparse distribution specified in the expression~\eqref{eq:theta-distribution}. 
\end{remark}

First, there exists a unique solution pair to the set of equations \eqref{eqn:fix-eqn} as asserted by Proposition~\ref{lem:uniqueness} (see Section~\ref{sec:proof-state-evolution} for more details). 
Experienced readers who are familiar with literature on Lasso shall immediately recognize the similarity between these equations and the ones used to determine the Lasso risk in the proportional regime \citep{bayati2011lasso}. We shall elaborate a bit more on their connections and differences in Section~\ref{Sec:connect-lasso}.

We now pause to interpret the above result. 
The risk of the minimum $\ell_1$-norm interpolator --- when the ratio $n/p$ is held fixed --- converges to a quantity $\taustar^2$, which is a function of $(\sigma, \delta, P_{\Theta})$ and can be determined by solving a system of two nonlinear equations with two unknowns. 
At a high level, the  equation~\eqref{eq:fix-1} indicates that 
$\taustar > \sigma$, which can be viewed as variance inflation as a result of undersampling. 
In addition, $\taustar$ taken together with the other parameter $\alphastar$ controls the sparsity level of $\widehat{\bs{\theta}}^{\mathsf{Int}}$. In fact, as can be seen from the equation~\eqref{eq:fix-2} and our analysis,  we have
\begin{align*}
\lim_{\substack{n/p 
	= \delta \\ n,\, p\to\infty}} \frac{1}{p} \lzero{\widehat{\bs{\theta}}^{\mathsf{Int}}} 
	~\overset{\mathrm{a.s.}}{=}~ \mathbb{P}\big( |\Theta+\taustar Z| > \alphastar \taustar\big) = \delta,
\end{align*}
which is essentially saying that the support size of $\widehat{\bs{\theta}}^{\mathsf{Int}}$ converges to $n$ in the limit. 
The proof of Theorem~\ref{thm:interpolation-risk} is established via analyzing 
a sequence of \emph{Approximate Message Passing} (AMP) updates with careful choices of parameters, such that the minimum $\ell_{1}$-norm solution is the fixed point of these updates. 
The state evolution formula that characterizes the large $n$ limit for each iterate is derived, and its large $t$ limit corresponds to the risk of 
the minimum $\ell_{1}$-norm solution. 
The readers are referred to Section~\ref{sec:main-proof-AMP} for details.

\paragraph{Multi-descent phenomenon.}

Having obtained an exact characterization of $\Risk(\widehat{\bs{\theta}}^{\mathsf{Int}})$ in a general manner,
we are ready to specialize Theorem~\ref{thm:interpolation-risk} to the ground-truth distribution \eqref{eq:theta-distribution}
and examine how $\taustar$ changes as a function of the aspect ratio $\delta$. 
Specifically, if denote $\nu \defn M/\tau$, then  the equations~\eqref{eqn:fix-eqn} simplify to
\begin{subequations}
\label{eqn:fix-eqn-sparse} 
\begin{align}
1 &= \frac{\nu^2}{M^2}\sigma^2 + \frac{\epsilon}{\delta} \mathbb{E}\left[\big( \eta(\sqrt{\delta}\nu+Z; \alpha)-\sqrt{\delta}\nu \big)  ^2\right] + \frac{1-\epsilon}{\delta} \mathbb{E}\left[\eta^2(Z; \alpha)\right] \\
\delta &= \epsilon\mathbb{P} \big( |\nu \sqrt{\delta}+ Z|> \alpha \big) + (1-\epsilon) \mathbb{P} ( |Z |> \alpha )
\end{align}
\end{subequations}
in the presence of the distribution \eqref{eq:theta-distribution}. 
From equation set~\eqref{eqn:fix-eqn-sparse}, we can readily examine how  $\taustar$ varies with $\delta$, 
which is the content of Section~\ref{sec:main-proof-double-descent} (along with the corresponding appendix).  
In particular, 
we can use \eqref{eqn:fix-eqn-sparse} to demonstrate that: the risk curve undergoes a phase transition in terms of the sparsity level $\epsilon$ --- as summarized in Theorem~\ref{thm:interpolation} --- such that the curve transitions from a single descent to multiple descent in the over-parameterized regime. 
To the best of our knowledge, this provides the first theoretical justification for the multiple-descent phenomenon associated with the minimum $\ell_{1}$-norm interpolator, and might shed light on understanding the behavior of other interpolators such as the M-estimators with a general family of objective functions.

\paragraph{Comparisons with ridgeless regression.}
\cite{hastie2019surprises} investigated the risk behavior of the ridge estimator when the penalized parameter $\lambda$ tends to zero --- which corresponds to the minimum $\ell_2$-norm interpolator in the over-parameterized regime --- and solidified a double-descent phenomenon as one increases the over-parameterized ratio $p/n$.  
To facilitate comparisons to their results, we first translate the results in \cite{hastie2019surprises} using our notation. 
Specifically, the generalization error of the minimum $\ell_2$ interpolator  --- denoted by $\widehat{\bs{\theta}}^{\mathsf{Int}, \ell_2}$ ---  obeys 
\begin{equation}
\label{eqn:risk-l2}
\lim_{\substack{n/p = \delta \\ n,\, p\to\infty}}
\Risk(\widehat{\bs{\theta}}^{\mathsf{Int}, \ell_2}) 
	~ \overset{\mathrm{a.s.}}{=} ~ 
\begin{cases}
\frac{\delta}{\delta-1}\sigma^2, & \text{if }\delta > 1 \\[0.2cm]
	\epsilon M^2(1-\delta) + \frac{1}{1-\delta}\sigma^2, & \text{if }\delta < 1
\end{cases}
\end{equation}
under the model~\eqref{eq:theta-distribution}. 
By calculating the derivative of the right-hand side of \eqref{eqn:risk-l2} w.r.t.~$\delta$, 
one can easily demonstrate that the exact asymptotics of $\Risk(\widehat{\bs{\theta}}^{\mathsf{Int}, \ell_2})$ decays with\footnote{Following the convention, we study the relation regarding $1/\delta = p/n$ instead of $\delta = n/p$.} $1/\delta$  when $\epsilon \leq \sigma^2/M^2$; 
otherwise, if $\epsilon > \sigma^2/M^2$,  then the risk curve undergoes a decreasing phase before hitting 
the point associated with $\delta = 1 - \frac{\sigma}{\sqrt{\epsilon} M}$, and starts to increase with $1/\delta$ afterward. 
Next, we single out a few key differences between their results and ours in Theorem~\ref{thm:interpolation}. 
\begin{itemize}
\item The current paper considers the case where the sparsity ratio of $\thetastar$ is held fixed across different random instances of increasing dimension, with the SNR frozen to be  $\epsilon M^{2}/\sigma^2$. 
The role of over-parametrization is studied when the minimum $\ell_{1}$-norm estimator (which naturally promotes sparse solutions) is fitted with full model dimension  $p$. 

In contrast, \cite{hastie2019surprises} studied the case where the underlying signal $\thetastar$ has a bounded $\ell_{2}$-norm and potentially dense.

\item Interestingly, Theorem~\ref{thm:interpolation} suggests that the minimum $\ell_{1}$-norm interpolator often exhibits more than two descent, thus revealing a fundamental difference between these two types of interpolation.

\item There exists a convenient closed-form expression for the minimum $\ell_2$-norm interpolator, 
	which assists in characterizing the precise asymptotics (i.e., one can decompose the risk formula into bias and variance terms, and pin down each term with the aid of random matrix theory).   
		Unfortunately, the minimum $\ell_{1}$-norm interpolator does not admit a concise closed-form expression, thus making it  considerably more challenging to analyze. In light of this, Section~\ref{sec:main-proof-double-descent} is devoted to the analysis of the above-mentioned nonlinear system of equations, with the aim of determining (local) monotoncity of the corresponding quantities of interest.

\end{itemize}

\subsection{Connections to the Lasso estimator}
\label{Sec:connect-lasso}

Apparently, the minimum $\ell_{1}$-norm interpolator~\eqref{eqn:sparse-intepolation} is closely related to the classical Lasso estimator studied extensively in high-dimensional statistics \citep{tibshirani1996regression}. 
Given a positive regularization parameter $\lambda > 0$, 
the Lasso estimates the regression coefficients by solving the following optimization problem 
\begin{align}
\label{eqn:lasso-definition}
\widehat{\bs{\theta}}_\lambda \defn 
\argmin_{\bs{\theta}\in\mathbb{R}^p}\left\lbrace \frac{1}{2}\left\| \bm y - \bs{X} \bs{\theta} \right\|_2^2 + \lambda \left\| \bs{\theta} \right\|_1 \right\rbrace.
\end{align}
As a consequence, the minimum $\ell_{1}$-norm interpolator corresponds to the limit of $\widehat{\bs{\theta}}_\lambda$ when taking $\lambda$ to  zero. 

Several prior works have attempted to characterize the exact asymptotics of the Lasso risk $\mathsf{Risk}(\widehat{\bs{\theta}}_\lambda)$ in the proportional regime. Specifically,  it has been proven that for any given $\lambda>0$, $\mathsf{Risk}(\widehat{\bs{\theta}}_\lambda)$ converges to a non-trivial limit $\taustar(\lambda)$. 
Here,  
$(\taustar(\lambda), \alpha^*(\lambda))$ represents the solution pair to the following set of nonlinear equations
\begin{subequations}
\label{eq:nonzero-lambda}
	\begin{align}
	\label{eq:nonzero-lambda-1} \tau^2 & = \sigma^2 + \frac{1}{\delta}\mathbb{E}\left[ \big( \eta(\Theta+\tau Z, \alpha\tau)-\Theta\big) ^2\right] ,\\
	\label{eq:nonzero-lambda-2} \lambda & = \alpha\tau\left( 1-\frac{1}{\delta}\mathbb{P}\big( |\Theta+\tau Z| > \alpha\tau \big) \right),
	\end{align}
\end{subequations}
where $\Theta \sim P_{\Theta}$ and $Z\sim \mathcal{N}(0,1)$ are independent random variables.  
The interested reader can consult \citet{bayati2011lasso,miolane2018distribution,celentano2020lasso},
which determined the Lasso risk using either the AMP machinery or the convex Gaussian min-max theorem.

As can be easily seen, the system of equations~\eqref{eq:nonzero-lambda} bears much resemblance to \eqref{eqn:fix-eqn}. 
More precisely, by directly setting $\lambda$ to $0$, the equation set~\eqref{eq:nonzero-lambda} reduces to the one in~\eqref{eqn:fix-eqn}. In other words, one has 
\begin{align}
\label{eqn:double-lim-lasso}
	\lim_{\lambda\to 0} \lim_{\substack{n/p = \delta \\ n,p\to\infty}} \mathsf{Risk}(\widehat{\bs{\theta}}_\lambda) 
	~\overset{\mathrm{a.s.}}{=}~ 
	\lim_{\lambda \to 0}[\taustar(\lambda)]^2
	  = \taustar^2. 
\end{align}
where $\taustar$ denotes the quantity in Theorem~\ref{thm:interpolation-risk}, and 
the last identity holds under certin continuity assumptions w.r.t.~the equations \eqref{eq:nonzero-lambda}.

Intuitively, Theorem~\ref{thm:interpolation-risk} can be directly established if it is legitimate to switch the order of limits between $\lambda$ and $p$ on the left hand side of expression~\eqref{eqn:double-lim-lasso},  given that the minimum $\ell_{1}$-norm interpolator is the limit of the Lasso by taking $\lambda$ to zero. 
However, formally establishing the validity of exchanging limits is quite challenging, since doing so normally requires the loss function being strongly convex (at least locally strongly convex around the solution point). Such a strong convexity property, however, is lacking in our problem structure when $\lambda$ is taken to zero. 
In fact, this presents a major roadblock to directly applying the established AMP theory for the Lasso estimator.

Fortunately, we can directly argue that exchanging the two limits leads to the same result, as formalized in the proposition below.  The proof of this result is postponed to Section~\ref{sec:lasso-limit-proof}.

\begin{props}[The Lasso limit when $\lambda\goto 0$]\label{prop:lasso-limit}
In the setting of Theorem~\ref{thm:interpolation-risk}, 
the Lasso risk obeys the following asymptotically exact characterization:
	\begin{enumerate}
	
		\item When $\delta < 1$, the asymptotic Lasso risk converges to the risk of min $\ell_1$-norm interpolator \eqref{eqn:sparse-intepolation}:
		\begin{align*}
		\lim_{\lambda\goto 0}\lim_{\substack{n/p = \delta \\ n,p\to\infty}}
		\mathsf{Risk}(\widehat{\bs{\theta}}_\lambda)
		~\overset{\mathrm{a.s.}}{=}~ \taustar^2,
		\end{align*}
		with $\taustar$ being the solution to the system of equations \eqref{eqn:fix-eqn}.

		\item When $\delta > 1$, the asymptotic Lasso risk converges to the risk of the ordinary least-square solution, namely, 	
		\begin{align*}
		\lim_{\lambda\goto 0}\lim_{\substack{n/p = \delta \\ n,p\to\infty}}
		\mathsf{Risk}(\widehat{\bs{\theta}}_\lambda)
			~\overset{\mathrm{a.s.}}{=}~
		\frac{\delta}{\delta-1} \sigma^2.
		\end{align*}

	\end{enumerate}
\end{props}
In words, the above result reveals that: while the connection between the set of equations~\eqref{eq:nonzero-lambda} and the Lasso risk was previously only shown for a positive $\lambda$,  such exact asymptotics continue be valid even in the limit when $\lambda$ approaches zero.


\section{Key analysis}

This section presents the key ideas for proving our main results.
We start by presenting the proof strategy for Theorem~\ref{thm:interpolation-risk}, which is built upon the recently developed approximate message passing machinery. It is then followed by the proof of Theorem~\ref{thm:interpolation} that characterizes the geometric properties of the risk curve.

\subsection{Key analysis tool: approximate message passing}
\label{sec:main-proof-AMP}

The major technical enabler for proving Theorem~\ref{thm:interpolation-risk} lies in the recent development of an iterative algorithm called the \emph{Approximate Message Passing} (AMP) algorithms. As mentioned previously, \cite{bayati2011lasso} employed AMP to pin down the risk of the Lasso estimator with positive regularization. Motivated by this line of works,  this paper resorts to the AMP technique as a proof device towards understanding the risk behavior of the minimum $\ell_1$-norm interpolators.

For our purpose, we need to generalize the original AMP updates (\cite{bayati2011lasso}) --- which were designed to solve a single Lasso problem in the large-system limit --- to approximate a sequence of Lasso problems with changing (and converging) regularization parameters. 
To better illustrate this idea,  we shall first provide a brief review of how AMP is invoked to solve a single Lasso problem, followed by a generalization of this framework to accommodate the minimum $\ell_1$-norm interpolator.

\subsubsection{AMP for the Lasso estimator}

\paragraph{AMP updates for Lasso.} 
Recall that the soft-thresholding function is defined in expression~\eqref{eqn:soft-thresholding} and $\left\langle \cdot \right\rangle $ denotes the average of the coordinates for the target vector.
When initialized at $\t^0 = \bs{0}$ and $\bs{z}^{-1}=\bs{0}$, 
the AMP algorithm proceeds recursively in the following fashion 
\begin{subequations}
	\label{eq:amp-formula-linear}
	\begin{align}
	\label{eq:amp-formula-linear-1} 
	\t^{t+1} &= \eta(\bs{X}^\top \bs{z}^t + \t^t; \zeta_t); \\
	\label{eq:amp-formula-linear-2}
	\bs{z}^t &= \bs{y} - \bs{X}\t^t + \frac{1}{\delta} \bs{z}^{t-1}\left\langle \eta'(\bs{X}^\top \bs{z}^{t-1} + \t^{t-1} ; \zeta_{t-1})\right\rangle.
	\end{align}
\end{subequations} 
Here, $\{\zeta_{t}\}_{t=0}^{\infty}$ is an appropriate sequence of scalars to be selected. 
To approximate the Lasso solution with positive $\lambda > 0$ (defined in \eqref{eqn:lasso-definition}), \cite{bayati2011lasso} showed that it suffices to set 
\begin{align}
\label{eq:amp-formula-lasso}
	\zeta_t = \alphastar(\lambda)\cdot\tau_t(\lambda), \quad \text{for all } t\geq 0,
\end{align}
where $\alphastar(\lambda)$ is taken as the corresponding solution to the fixed-point equation~\eqref{eq:nonzero-lambda}
and $\tau_t(\lambda)$ shall be specified momentarily. 
Given this choice of parameters, \citet[Theorem 1.8]{bayati2011lasso} proved that the corresponding AMP update $\t^t$ converges to the Lasso solution in the following sense: as long as $\mathbb{E}[\Theta^2]<\infty$ and $\mathbb{P}(\Theta\neq 0)>0$, it holds that
\begin{align}
\label{eq:lasso-amp-gap}
	\lim_{t\goto\infty}\lim_{\substack{n/p = \delta \\ n,p\to\infty}}\frac{1}{p}\|\t^t - \widehat{\t}_{\lambda}\|_2^2 ~\overset{\mathrm{a.s.}}{=}~ 0.
\end{align}
We emphasize that this convergence result requires taking the limit of the model dimensions before taking the limit of the iteration steps; hence, it should be understood as high-dimensional asymptotics. 
Equipped with this result, one is able to study the limiting performance of Lasso via the AMP iterations at each fixed step $t$; the latter is made possible by the state evolution characterization to be introduced below.

\paragraph{State evolution.} 
Consider the AMP procedure~\eqref{eq:amp-formula-linear} with an arbitrary sequence of thresholds $\{\zeta_t\} > 0$. 
The state evolution sequence $\{\tau_t^2\}_{t= 0}^\infty$ is a one-dimensional iteration sequence, recursively defined for all $t\geq 0$ as follows
\begin{subequations}
\label{eqn:state-evolution}
\begin{align}
\label{eqn:state-evolution-1} 
\tau_{t+1}^2 &= \mathsf{F}(\tau_t^2, \zeta_t)\\
\label{eqn:state-evolution-2}
	\text{where }~\mathsf{F}(\tau^2, \zeta) & \defn \sigma^2 + \frac{1}{\delta}\mathbb{E}\left[ \big[ \eta(\Theta + \tau Z; \zeta) - \Theta\big]^2\right]
\end{align}
\end{subequations}
with initialization $\tau_0^2 = \sigma^2 + \mathbb{E}[\Theta^2]/\delta$. 
Here, $\Theta$ is the distribution of the true signal, and $Z\sim \NORMAL(0, 1)$ is independent of $\Theta$. 
The above sequence is known to characterize the limiting variance of the AMP recursion, as formalized by the following result. 
\begin{props}[Theorem 1.1, \cite{bayati2011lasso}]
	\label{prop:amp-fixed-t}
	Consider the linear model~\eqref{eq:model} and i.i.d.~Gaussian design. 
	If $\mathbb{E}[\Theta^2] < \infty$ and $\mathbb{P}(\Theta \neq 0)>0$,
	then for any positive sequence $\{\zeta_t\}$ and any pseudo-Lipschitz function $\psi$, it holds that 
	\begin{align*}
		\lim_{p\goto\infty}\frac{1}{p}\sum_{i=1}^p \psi(\theta_i^{t+1}, \theta^\star_i) ~\overset{\mathrm{a.s.}}{=}~ 
		\mathbb{E}\big[ \psi\big(\eta(\Theta  + \tau_t Z; \zeta_t), \Theta \big)\big],
	\end{align*}
	where $Z\sim\NORMAL(0, 1)$ is independent of $\Theta$.
\end{props}

In words, this proposition asserts that the coordinates of $\bs{X}^\top\bs{z}^t + \t^t$ have roughly the same distribution as $\Theta + \tau_t Z$. 
Taking $\psi(x,y) = (x-y)^{2}$ and combining this with the expression~\eqref{eqn:state-evolution-2} indicate that: the asymptotic risk of the AMP in the $t$-th iteration is characterized by $\tau_t^{2}$. 
Indeed, the state evolution $\tau_t^{2}$ quantifies how this asymptotic risk
evolves with the iteration count.  
If we let the iteration number $t$ tend to infinity, then $\tau_t^{2}$ converges to a nonzero limit --- i.e., the solution to the system of equations~\eqref{eq:nonzero-lambda} --- which is precisely the limiting risk for the Lasso estimator by virtue of the property~\eqref{eq:lasso-amp-gap}.

\subsubsection{AMP for the minimum $\ell_1$-norm interpolator}

As discussed above, when AMP adopts the choice of $\zeta_t = \alphastar(\lambda)\cdot\tau_t(\lambda)$, then in  each iteration, it makes progress towards the Lasso solution in the presence of a positive regularization parameter $\lambda$. 
Intuitively, one can run AMP in an epoch-based manner, and gradually reduce the value of $\lambda$ by taking a vanishing sequence of $\left\lbrace \lambda_t\right\rbrace$ and set $\zeta_t = \alphastar(\lambda_t)\cdot\tau_t(\lambda_t)$. Heuristically, each epoch solves a Lasso problem approximately with parameter $\lambda_{t}$ and, in the end, one can recover the minimum $\ell_1$-norm interpolator in the limit. Similar heuristics have been pointed out by \cite{donoho2010message} without a rigorous argument.

It turns out this intuition can be solidified 
as long as one selects the sequence $\left\lbrace \lambda_t\right\rbrace$ appropriately. 
Let us now describe our choice of the $\left\lbrace \lambda_t\right\rbrace$ sequence, and use them to construct $\left\lbrace \zeta_t\right\rbrace $ in the AMP updates. 

\paragraph{Choice of $\{\zeta_t\}$ in our setting.}
Our first step is to construct a positive sequence of $\left\lbrace \lambda_t\right\rbrace $ satisfying the following assumption:
\begin{asss}
\label{assump:lambda}
For every $t=1,2,\ldots$, define $\Lambda_t \defn \sum_{s=1}^t \lambda_s$. 
We assume that $\left\lbrace \lambda_t\right\rbrace_{t=1}^{\infty} $ satisfies the following conditions:
	\begin{itemize}
		\item $\lim_{t\goto\infty}\lambda_t = 0$ and $\lim_{t\goto\infty}\lambda_t/\lambda_{t+1}= 1$;
		\item $\sum_{j=t/2}^t \lambda_j \geq c \log t$ for every constant $c$ and sufficiently large $t$;
		\item The following two sequences are summable for every constant $c$,
		\begin{align}
		\label{eqn:assmption-lambda-summable2}
		&\sum_{t=1}^{+\infty}  \exp\left\lbrace - c\Lambda_t\right\rbrace < \infty,
		\qquad \text{ and }~~~\sum_{t=1}^{+\infty}  \sqrt{l_t} < \infty,\\
		&\notag\text{where }~~
		l_t \defn \sum_{s=1}^t \big|\lambda_s-\lambda_{s+1}\big| 
			\exp\big( - c [\Lambda_t - \Lambda_s]\big). 
		\end{align}
	\end{itemize}
\end{asss}
In words, Assumption~\ref{assump:lambda} requires that $\left\lbrace \lambda_t\right\rbrace $ converges to $0$, but the convergence rate should be slow enough. 
We shall provide an example of $\{\lambda_t\}$ satisfying this assumption in Section~\ref{sec:verify-assumption}.
As will be made clear from the proof, 
having $\lambda_t \goto 0$ guarantees that the solution to the system of equations converges to the minimum $\ell_1$-norm solution, 
whereas other conditions ensure that the AMP iterates do not experience drastic changes in adjacent iterations, such that the state evolution formula still accurately describe the behavior of their limit. Specifically, under these other conditions, the difference between the support sets of consecutive iterates can be properly controlled.

With this choice of $\left\lbrace \lambda_t\right\rbrace$ sequence, we  define a series of nonlinear systems of equations with two unknowns, indexed by $t$ as follows:
\begin{gather}
\label{eq:fixed-point}
\begin{aligned}
\tau^2 & = \mathsf{F}(\tau^2, \alpha\tau),\\
\lambda_t & = \alpha\tau \left( 1 - \frac{1}{\delta}\mathbb{E}\big[ \eta'(\Theta + \tau Z; \alpha\tau)\big] \right),
\end{aligned}
\end{gather}
where the function $\mathsf{F}(\cdot,\cdot)$ is specified in expression~\eqref{eqn:state-evolution-2}. As usual, $Z$ is a standard Gaussian random variable that is independent of $\Theta$. 
Recognizing the existence and uniqueness property shown in Section~\ref{sec:proof-state-evolution}, we can guarantee that the equation set~\eqref{eq:fixed-point} yields a unique solution pair, which shall be denoted by $(\alphastart, \taustart)$. 
Further, we define the threshold $\zeta_{t}$ for our AMP updates~\eqref{eq:amp-formula-linear} as follows
\begin{align}
\label{eqn:choice-zeta-t}
	\zeta_t \defn \alphastart \cdot \tau_t
	\qquad \text{for all } t >0,
\end{align}
where $\zeta_0 \defn 1$ and $\tau_{t}$ corresponds to the state evolution formula provided in the expression~\eqref{eqn:state-evolution-1}. 
In view of the correspondence between the AMP updates and the Lasso estimator, iteration $t$ of our AMP updates takes a step towards approximating the Lasso estimator with parameter $\lambda_{t}$. 
As $\lambda_{t}$ converges to zero, the iteration procedure has the minimum $\ell_1$-norm interpolator as a limiting point. 
The informal intuition is illustrated in Figure~\ref{fig:amp}. 
\begin{figure}[t]
	\centering
	\includegraphics[width=0.7\textwidth]{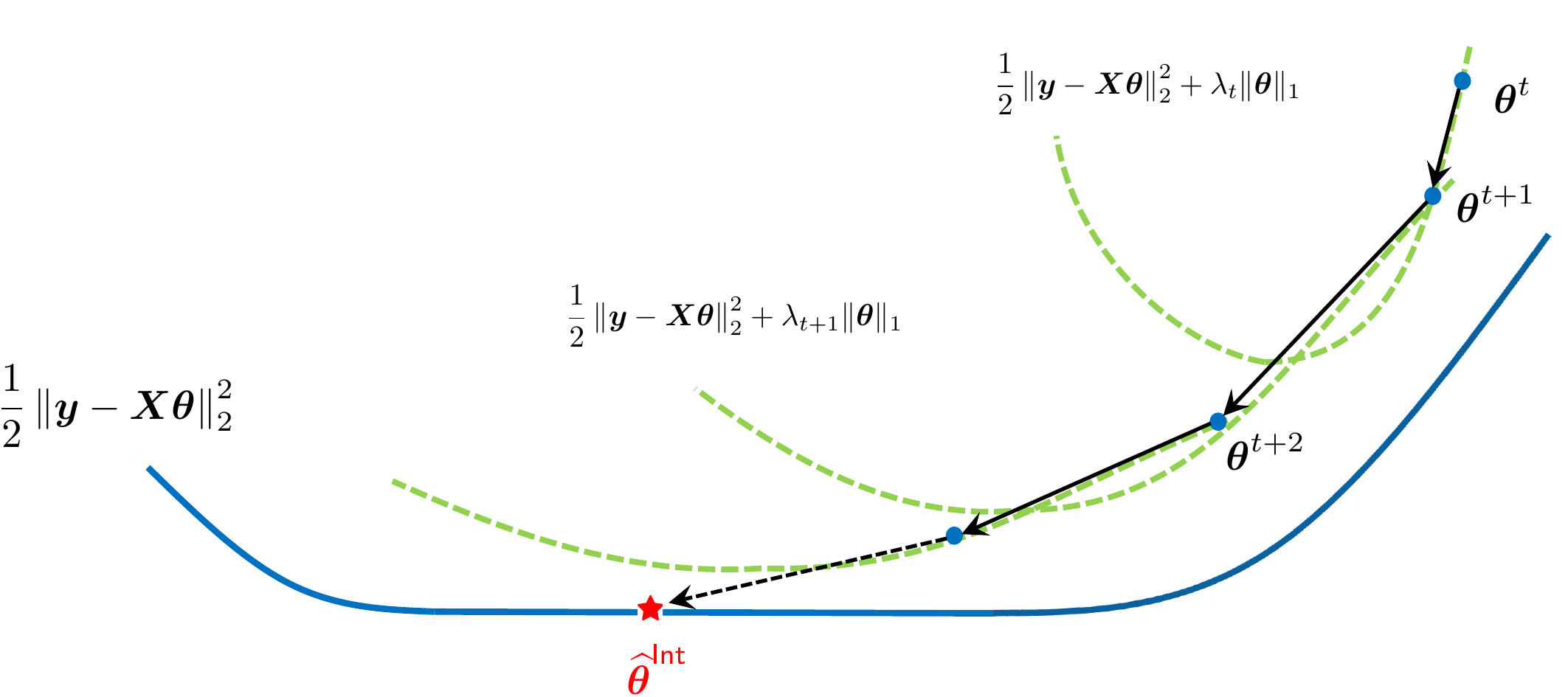}
	\caption{Illustration of the AMP updates for the minimum $\ell_1$-norm interpolator. At each step, $\bs{\theta}^{t+1}$ is computed as in the expression~\eqref{eq:amp-formula-linear-1}, with $\zeta_t$ chosen according to \eqref{eqn:choice-zeta-t}.
	The plateau of the blue curve stands for all the interpolators that satisfy $\bm y = \bm X \bm \theta$, among which $\thetaint$ has the smallest $\ell_1$-norm. 
	The curves in green stand for the Lasso loss functions (with parameter $\lambda_t$ changing with $t$) where AMP aims to move towards its minimizer in each $t$.}
	\label{fig:amp}
\end{figure}

We are now ready to state our main result on the risk of the min $\ell_1$-norm solution. 
\begin{theos}\label{thm:interpolation-risk-amp}
Consider the linear model~\eqref{eq:model} and i.i.d.~Gaussian design. 
	If $\mathbb{E}[\Theta^2] < \infty$ and $\mathbb{P}(\Theta\neq 0) > 0$,
then for any pseudo-Lipschitz function $\psi$, one has
	\begin{equation}
	\lim_{p\goto\infty} \frac{1}{p}\sum_{i=1}^p 
		\psi(\thetainti, \theta_i^\star) ~\overset{\mathrm{a.s.}}{=}~ \mathbb{E}\big[ \psi(\eta(\Theta + \taustar Z; \alphastar \taustar), \Theta)\big] ,
	\end{equation}
	where $Z\sim \NORMAL(0, 1)$ is independent of $\Theta$. 
	Here, $(\alphastar, \taustar)$ is the solution to the following equations 
	\begin{equation}\label{eq:fixed-point-zero}
	\tau^2 = \mathsf{F}(\tau^2, \alpha\tau);\quad 
	\frac{1}{\delta}\mathbb{E}\big[ \eta'(\Theta + \tau Z; \alpha\tau)\big] = 1.
	\end{equation}
\end{theos}

We now point out an immediate consequence of Theorem~\ref{thm:interpolation-risk-amp}. In view of the pseudo-Lipschitz property of the function $\psi(a, b)= (a-b)^2$, we can obtain Theorem~\ref{thm:interpolation-risk} as a corollary, namely, 
the limiting risk of the minimum $\ell_1$-norm solution obeys 
	\begin{align}
		\lim_{p\goto\infty} \frac{1}{p} \big\|\thetaint - \coef \big\|_2^2 ~\overset{\mathrm{a.s.}}{=}~ {\taustar}^2.
	\end{align}
This enables us to study the risk curve by examining the equations \eqref{eq:fixed-point-zero}.

\paragraph*{Proof ideas.} Before proceeding, let us highlight several key challenges and differences in this part of the proof in comparison to \cite{bayati2011lasso}. The complete details are deferred to Section~\ref{sec:proof-AMP}. 
As already mentioned, we look at a sequence of AMP updates, each targeting at solving a Lasso problem with a different regularization parameter $\lambda_{t}$ that obeys $\lim_{t\to\infty} \lambda_{t} = 0$. 
In the fixed $\lambda$ scenario, it is known that even if $p>n$,  the loss function around the Lasso estimate enjoys certain restricted strongly convexity. This is, however, not the case for the minimum $\ell_{1}$-norm interpolator, whose support size equals $n$; this implies that the condition number of $\bs{X}^{\top} \bs{X}$ (restricted to the support) might be very large. 
Consequently, it calls for the development of a new structural property tailored to the minimum $\ell_{1}$-norm solution, as we shall detail in Lemma~\ref{lem:lasso-structure}. 

Moreover, for each $\lambda_{t}$, the AMP iterations are contractive towards different fixed points (i.e., minimizers of different Lasso problems). One thus needs to investigate how the pesudo-state evolution point $\tau^{\star}_{t}$ varies with the iteration number $t$. 
In addition, to demonstrate that AMP converges to the new system of equations as specified in \eqref{eqn:fix-eqn}, 
at a high level, we construct some distance measure between $\bs{\theta}^{t+1}$ and $\bs{\theta}^{t}$ so as to guarantee that 
\begin{align}
\label{eqn:intuition}
	\mathsf{dist}(\bs{\theta}^{t+1}, \bs{\theta}^{t})
	\leq 
	\exp(-\lambda_t) \cdot \mathsf{dist}(\bs{\theta}^{t}, \bs{\theta}^{t-1})
	+
	c |\lambda_t - \lambda_{t+1}|.
\end{align}
In the case of a fixed $\lambda$, the above relation simplifies to 
$\mathsf{dist}(\bs{\theta}^{t+1}, \bs{\theta}^{t})
	\leq 
	\exp(-\lambda) \cdot \mathsf{dist}(\bs{\theta}^{t}, \bs{\theta}^{t-1}), $ 
which means that $\mathsf{dist}(\bs{\theta}^{t+1}, \bs{\theta}^{t})$ converges linearly and $\bs{\theta}^{t}$ converges to the corresponding limit. 
In contrast, the second term on the right-hand side of \eqref{eqn:intuition} reflects the price one needs to pay when $\lambda_{t}$ varies across iterations. 
Assumption~\ref{assump:lambda} is imposed to help ensure that these 
errors --- albeit accumulated over time --- stay bounded.


\subsection{Analysis ingredients for the risk curve}
\label{sec:main-proof-double-descent}

Thus far, we have demonstrated that the risk curve of the minimum $\ell_{1}$-norm interpolator can be characterized by the solutions to the system of equations~\eqref{eqn:fix-eqn}.
In order to analyze the shape of the risk curve and establish Theorem~\ref{thm:interpolation}, 
this section takes a close look at the geometric properties of these solutions.  
For ease of exposition, let us assume without loss of generality that $\sigma^2=1$ throughout the proof; clearly, having a different value of $\sigma^{2}$ does not change the shape of the curve as long as the SNR remains unchanged.

\paragraph*{Roadmap of the proof.}
We start by providing a roadmap of our proof. 
To begin with, it is challenging to track the behavior of the solution $\taustar$ directly in the original form of the equations \eqref{eqn:fix-eqn}; 
in fact, $\taustar$ blows up as $\delta \goto 1$. 
Hence, we find it more convenient to work with a new parameter $\nu \defn M/\tau$ (resp.~$\nustar \defn M/\taustar$) that leads to alternative versions of Theorem~\ref{thm:interpolation} and \eqref{eqn:fix-eqn}.  
With this change of variables in place, it suffices to study how $\nustar$ behaves as one varies $\delta$. 
Towards this end,  we first eliminate the parameter $\a$ and express $\nu$ purely as a function of $\delta$. We then proceed to analyze the derivative of $\nustar(\delta)$ using the implicit function theorem, with a special focus on the sign of $\nustar'(\delta)$ when $\delta$ is close to $0$ or $1$, as well as when $\epsilon \goto 0$. As we shall argue momentarily, these steps suffice in establishing Theorem~\ref{thm:interpolation}.

\paragraph{An equivalent formulation.}
As mentioned above, let us denote $\nu \defn M/\tau$, and define two functions of $(\nu, \delta, \alpha)$ as follows
\begin{subequations}
	\begin{align}
	\label{eq:fix-5} 
	F_1(\nu, \delta, \alpha) & \defn \epsilon\mathbb{P}\left( \big|\nu \sqrt{\delta}+ Z\big|> \alpha \right) + (1-\epsilon) \mathbb{P}\left( \left| Z\right|> \alpha \right)-\delta, \\
	\label{eq:fix-6}
	F_2(\nu, \delta,\alpha) & \defn  \frac{\nu^2}{M^2}-1  + \frac{\epsilon}{\delta} \mathbb{E}\left[\big( \eta(\sqrt{\delta}\nu+Z;\alpha)-\sqrt{\delta}\nu \big) ^2\right] +\frac{1-\epsilon}{\delta} \mathbb{E}\left[\eta^2(Z; \alpha)\right] ,
	\end{align}
\end{subequations}
where $Z\sim \NORMAL(0, 1)$ and is independent of $\Theta \sim P_{\Theta}$.  
Under our assumptions on $\Theta$ (cf.~\eqref{eq:theta-distribution}), 
	solving the equations \eqref{eq:fix-1} and \eqref{eq:fix-2} can be accomplished by first finding the solutions $\big(\nustar(\delta), \alphastar(\delta)\big)$ to
\begin{align}\label{eq:fix-7}
	\begin{cases} F_1(\nu, \delta, \alpha) =0, \\
		 F_2(\nu, \delta, \alpha) =0,
	\end{cases}
\end{align}
and then mapping $\nustar$ back to $\taustar$. 

\subsubsection{Step 1: existence of the mapping $\nustar(\delta)$} 

We now attempt to eliminate the variable $\a$ in \eqref{eq:fix-7}, and expressing $\nustar$ as a function of $\delta$. For any $\delta\in (0, 1)$ and $\nu> 0$, direct computation of the derivative of the function $F_1(\nu, \delta, \alpha)$ yields 
\begin{align*}
\nabla_{\alpha} F_1(\nu, \delta, \alpha)  = - \epsilon \left[ \phi(\alpha-\sqrt{\delta}\nu) + \phi(\alpha+\sqrt{\delta}\nu)\right] - 2(1-\epsilon)\phi(\alpha) < 0.
\end{align*}
It is also straightforward to calculate the limiting values
\begin{align*}
\lim_{\alpha \goto 0^+} F_1(\nu, \delta, \alpha) = 1-\delta > 0; \quad \lim_{\alpha \goto +\infty} F_1(\nu, \delta, \alpha) = -\delta < 0.	
\end{align*}
Based on the above observations, given any $\delta$ and $\nu$, the function $F_1(\nu, \delta, \alpha)$ is monotonically non-increasing in $\alpha$, and can take both positive and negative values within the interval $(0,\infty)$. 
As a result, there exists a mapping from $(\nu, \delta) \to \alpha$ that satisfies $F_1(\nu, \delta, \alpha)=0$; with an abuse of notation, we often denote this function as $\alpha(\nu, \delta)$. 
Substitution into the function $F_2$ allows us to define 
\begin{equation}
\label{eq:f3}
F_3(\nu, \delta) \defn F_2(\nu, \delta, \alpha(\nu, \delta)).
\end{equation}
Here, the function $F_{3}$ depends solely on the two parameters $(\nu, \delta).$
Armed with the derivations above, solving \eqref{eq:fix-7} comes down to finding a solution to $F_3(\nu, \delta) = 0$. 

By construction of the function $F_3$, we know that the solutions to $F_3(\nu, \delta) = 0$ correspond to the solutions to the system of equations \eqref{eqn:fix-eqn}. 
As we shall demonstrate in Proposition~\ref{lem:uniqueness}, for every $0<\delta < 1$, there exists a unique pair of $(\tau, \a)$ satisfying \eqref{eqn:fix-eqn}; therefore, $F_3(\nu, \delta) = 0$ yields a unique solution for every $0 < \delta < 1$ --- which shall be denoted by ${\nustar}(\delta)$ in the sequel.
Correspondingly, the solution pair for \eqref{eqn:fix-eqn} shall be written as 
$(\taustar(\delta), \alphastar(\delta))$ where $ \alphastar(\delta) \defn \alpha(\taustar(\delta), \delta)$.
We also note that since both $F_1$ and $F_2$ are smooth functions with bounded derivatives w.r.t.~all parameters, ${\nustar}'(\delta)$ exists and is continuous.

\subsubsection{Step 2: derivative of $\nustar(\delta)$} 

With the mapping $\nustar(\delta)$ in place, we can translate Theorem~\ref{thm:interpolation} into statements about ${\nustar}'(\delta)$. 
Before doing so, recall that the risk incurred by using $\bs{\theta}=\bs{0}$ as the estimator satisfies 
\begin{align}
\label{eq:zero-limit}
	 \Risk(\bm{0}) = 
	\mathbb{E}\left[ (\inprod{\bs{x}_i}{\coef} + z_i)^2\right] = 
	1 + \frac{1}{n}  \|\coef\|_2^2  ~\overset{\mathrm{a.s.}}{\longrightarrow}~ 1 + \epsilon M^2 \eqqcolon \tau_0^2.
\end{align}
Let us define the corresponding value of $\nu$ as $\nu_0 \defn M/\tau_0$. 
Formally, to prove the first two claims in Theorem~\ref{thm:interpolation}, it suffices to establish the following proposition.

\begin{props}
\label{thm:interpolation-version}
	In the setting of Theorem~\ref{thm:interpolation}, for $\delta \in (0, 1)$, $\nustar(\delta)$ satisfies the following properties: 
	\begin{enumerate}
		\item $\lim_{\delta \goto 0^+} \nustar(\delta) = \nu_0$;
		\item There exist two constants $0 < \delta_1, \delta_2 < 1$ such that when $0<\delta <\delta_1$ and $\delta_2 < \delta < 1$, ${\nustar}'(\delta) < 0$.
	\end{enumerate}
\end{props}

Clearly, if Proposition~\ref{thm:interpolation-version} were valid, 
then the first two claims in Theorem~\ref{thm:interpolation} would follow immediately by invoking the change of variables $\tau =  M/\nu$.
Now we discuss how to establish this proposition. 
For notational simplicity, we use $\nustar$ and $\alphastar$ to denote the unique solution to \eqref{eq:fix-1} and \eqref{eq:fix-2} for any $\delta\in (0, 1)$, which should be understood as $\nustar(\delta) $ and $\alphastar(\delta)$.
To begin with, recognizing the fact that $F_1(\nustar,\delta,\alphastar)=0$, 
the implicit function theorem implies that 
\begin{align}
\label{eqn:piazzolla}
\nabla_{\nu}\alpha(\nustar,\delta) = \left. - \frac{\nabla_{\nu}F_1(\nu,\delta,\alpha)}{\nabla_{\alpha}F_1(\nu,\delta,\alpha)}\right|_{(\nustar, \delta, \alphastar)} ; \quad \nabla_{\delta}\alpha(\nustar,\delta) = \left. - \frac{\nabla_{\delta}F_1(\nu,\delta,\alpha)}{\nabla_{\alpha}F_1(\nu,\delta,\alpha)}\right|_{(\nustar, \delta, \alphastar)}.	
\end{align}
We are now ready to derive an explicit expression of $\nustar'(\delta)$. 
A little algebra leads to 
\begin{align}
\label{eq:derivative-nu}
\notag \nustar'(\delta)  = \left. - \frac{\nabla_{\delta}F_3(\nu,\delta)}{\nabla_{\nu}F_3(\nu,\delta)} \right|_{(\nustar, \delta)}
&= \left. - \frac{\nabla_{\delta}F_2(\nu,\delta, \alpha) + \nabla_{\delta}\alpha(\nu,\delta)\nabla_{\alpha}F_2(\nu,\delta, \alpha)}{\nabla_{\nu}F_2(\nu,\delta, \alpha) + \nabla_{\nu}\alpha(\nu,\delta)\nabla_{\alpha}F_2(\nu,\delta, \alpha)}\right|_{(\nustar, \delta, \alphastar)} \\
& = \left. - \frac{\nabla_{\delta}F_2\nabla_{\alpha}F_1 - \nabla_{\alpha}F_2\nabla_{\delta}F_1}{\nabla_{\nu}F_2\nabla_{\alpha}F_1 - \nabla_{\alpha}F_2\nabla_{\nu}F_1}\right|_{(\nustar, \delta, \alphastar)},
\end{align}
where the second equality invokes the relation~\eqref{eqn:piazzolla}. 
This expression plays a crucial role in our subsequent analysis in understanding how $\taustar$ changes with $\delta$.

In order to establish Proposition~\ref{thm:interpolation-version}, we gather in the following two lemmas some key facts on the limiting behavior of $\nustar'(\delta)$ and $\nustar(\delta)$, when $\delta \goto 0^+$ and $\delta \goto 1^-$, respectively.  
All of these are stated with the assumptions of Theorem~\ref{thm:interpolation} imposed, with the proofs  
deferred to Section~\ref{sec:proof-derivative-lemmas}. 
\begin{lems}
\label{lem:nu-prime}
When $\delta \goto 0^+$, the numerator and denominator in \eqref{eq:derivative-nu} satisfy respectively the following properties
\begin{subequations}
\label{eqn:num-de}
	\begin{gather}
	\label{eq:numerator}
	\nabla_{\delta}F_2\nabla_{\alpha}F_1 - \nabla_{\alpha}F_2\nabla_{\delta}F_1 \sim - 2{\alphastar}^{-3}, \\
	\label{eq:denominator}
	\nabla_{\nu}F_2\nabla_{\alpha}F_1 - \nabla_{\alpha}F_2\nabla_{\nu}F_1
	\sim -4\nu_0^{-1}\phi(\alphastar),
	\end{gather}
\end{subequations}
where all the partial derivatives of $F_1$ and $F_2$ are evaluated at the point $(\nustar, \delta, \alphastar)$. 
Additionally, it holds that 
\begin{align}\label{eq:nu-zero-limit}
\lim_{\delta \goto 0^+} \nustar = \nu_0. 
\end{align}

\end{lems}

\begin{lems}
\label{lem:nu-one-limit}
	When $\delta \goto 1^-$, it satisfies that 
	\begin{equation}\label{eq:nu-one-limit}
	\lim_{\delta \goto 1^-}{\nustar}'(\delta) = - \infty. 
	\end{equation}
\end{lems}

Note that the property~\eqref{eq:nu-zero-limit} in Lemma~\ref{lem:nu-prime} validates the first claim in Proposition~\ref{thm:interpolation-version} directly. 
Therefore, to prove Proposition~\ref{thm:interpolation-version}, 
we are only left with verifying the second claim in Proposition~\ref{thm:interpolation-version}. 
In view of Lemma~\ref{lem:nu-prime}, it is guaranteed that as $\delta\goto 0^+$, both the denominator and the numerator in \eqref{eq:derivative-nu} are negative. 
By continuity of $\nustar'(\delta)$, there exists some  $\delta_1>0$ such that when $0 < \delta < \delta_1$, one has $\nustar'(\delta) < 0$. 
Finally, Lemma~\ref{lem:nu-one-limit} immediately suggests that one can find $\delta_2>0$ such that: when $\delta_2 < \delta < 1$, one has $\nustar'(\delta) < 0$.
Taking these properties collectively concludes the proof of Proposition~\ref{thm:interpolation-version}.

\subsubsection{Step 3: limit behavior for the case with $\epsilon \goto 0$ }

Finally, let us move on to establishing the third and fourth claims of Theorem~\ref{thm:interpolation}. Specifically, fixing some $\delta > 0$, we shall study how the risk limit $\taustar$ behaves as $\epsilon$ varies, particularly as it tends to zero. 
Thus far, we have focused on the case when both $\epsilon$ and $M$ are regarded as fixed constants while the value of $\delta$ varies;  
in this case, the analyses were primarily performed w.r.t.~$\nustar'(\delta)$, 
since studying $\nustar(\delta)$ and $\nustar(\delta)/M$ are equivalent when $M$ is taken to be a fixed constant. 
However, in the case when we fix SNR (namely, $\epsilon M^2$) as opposed to $M$, one has $M\goto \infty$ as $\epsilon \goto 0$, 
and hence studying $\nustar(\delta)$ and studying $\nustar(\delta)/M$ are no longer equivalent. 
As a result, we need to analyze $\nustar(\delta)/M$ directly, that is, to examine the behavior of $\nustar(\delta)/M$ and $\nustar'(\delta)/M$ in a fixed-SNR regime as $\epsilon$ approaches zero.

The readers shall also bear in mind that we always focus on the derivative of the risk of the model~\eqref{eq:theta-distribution} with given $(M, \epsilon)$ but varying $\delta$. 
In other words, the function $\nustar(\delta): (0, 1)\mapsto \mathbb{R}_+$ is defined for any given $(M, \epsilon)$, and we do not associate the change of $(M, \epsilon)$ and the change of $\delta$ together. 
To emphasize that we now work with a \emph{fixed} ratio, we shall use $\delta_0$ in place of $\delta$. 
At this given ratio $\delta_0$, the quantities $\alphastar$ and $\nustar$ are treated as functions of $\epsilon$.

\paragraph*{Proof for part (d) of Theorem~\ref{thm:interpolation}.}
When $\epsilon \goto 0$, we first make a key observation on 
the behavior of  $\frac{\nustar'(\delta_0)}{M}$, as summarized in the lemma below. 
\begin{lems}\label{lem:epsilon}
	In the setting of Theorem~\ref{thm:interpolation}, given any fixed SNR$= \epsilon M^{2}$ and $\delta_0 \in (0, 1)$, the derivative $\nustar'$ (with respect to $\delta$) obeys 
	\begin{align}
	\label{eqn:brahms}
	 	\lim_{\epsilon\goto 0} \frac{\nustar'(\delta_0)}{M} > 0. 
	 \end{align}
\end{lems}

The proof of Lemma~\ref{lem:epsilon} contains two main parts, whose 
details are deferred to Section~\ref{sec:pf-lemma-epsilon}. 
First, letting $\a_0\defn -\Phi^{-1}(\delta_0/2)$ for this given $\delta_{0}$, we establish the following relation
\begin{align}
\label{eq:epsilon-limit-2}
\alphastar \goto \a_0; \quad \text{ and }~
\frac{\nustar}{M} \goto \sqrt{1 - 2\delta_0^{-1}[-\a_0\phi(\a_0) + (\a_0^2+1)\Phi(-\a_0)]},
\end{align}
as one takes $\epsilon \goto 0$; the details can be found in Section~\ref{sec:pf-lemma-epsilon}.
It is worth noting that both $\alphastar$ and $\frac{\nustar}{M}$ converge to fixed quantities that are determined only by $\delta_{0}$ in this limit. 

Equipped with these two limiting values, we proceed to consider the numerator and denominator of $\frac{\nustar'}{M}$, with the assistance of 
the expression~\eqref{eq:derivative-nu}. 
In fact, one can pin down the limiting orders of these two parts as follows
\begin{subequations}
\begin{align}
\label{eqn:nu-ratio-numerator}
\nabla_{\delta}F_2\nabla_{\alpha}F_1 - \nabla_{\alpha}F_2\nabla_{\delta}F_1
&  ~\goto~ 2\delta_0^{-2} \phi(\alpha_0)\left[ -2\a_0\phi(\a_0) + (\a_0^2+1)\delta_0\right] - 2 \delta_0^{-1} \left[ 2 \phi(\a_0) - \a_0\delta_0\right],
\end{align}
and 
\begin{align}
\label{eqn:nu-ratio-denominator}
	M(\nabla_{\nu}F_2\nabla_{\alpha}F_1 - \nabla_{\alpha}F_2\nabla_{\nu}F_1) &
	~\goto~ -4\phi(\a_0) \sqrt{1 - 2\delta_0^{-1}[-\a_0\phi(\a_0) + (\a_0^2+1)\Phi(-\a_0)]}.
\end{align}
\end{subequations}
The details can be found in Step 2 in Section~\ref{sec:pf-lemma-epsilon}.

Putting these together, we can conclude that
\begin{align*}
	\lim_{\epsilon \to 0}\frac{\nustar'(\delta_0)}{M} 
	&=
	\frac{2\delta_0^{-2}\left[ -2\a_0\phi^2(\a_0) + (\a_0^2+1)\delta_0\phi(\a_0) - 2 \delta_0\phi(\a_0) + \a_0\delta_0^2\right]}{4\phi(\a_0)\sqrt{1 - 2\delta_0^{-1}[-\a_0\phi(\a_0) + (\a_0^2+1)\Phi(-\a_0)]}}
	< 0,
\end{align*}
where the last inequality follows due to the fact that 
$\Phi(-\a_0) = \delta_0/2$ and the basic relation 
\begin{align*}
	\Phi(-\a_0) \in 
	\Big[\phi(\a_0)\bigg(\frac{1}{\a_0}-\frac{1}{\a_0^3}\bigg),
	~ \phi(\a_0)\bigg(\frac{1}{\a_0}-\frac{1}{\a_0^3} + \frac{1}{\a_0^5}\bigg)\Big].
\end{align*}

In summary, in view of Lemma~\ref{lem:epsilon}, 
we can conclude that there exists $\epsilonstar>0$, depending only on SNR and $\delta_0$, such that: when $\epsilon < \epsilonstar$, one has $\nustar'(\delta_0)/M < 0$. 
Translating this back to $\taustar = M/\nustar$ ensures the existence of an $\epsilonstar$ such that: when $\epsilon < \epsilonstar$, one has $\taustar'< 0$.
We have thus completed the proof of Part (d) of Theorem~\ref{thm:interpolation}.

\paragraph*{Proof of part (c) of Theorem~\ref{thm:interpolation}.}
The idea for proving this result is to find $\delta \in (0,1)$ such that the value of ${\taustar}^2(\delta)$ is strictly below $\Risk(\bm{0})$. 
Recognizing that ${\taustar}^2(\delta)$ decays to $\Risk(\bm{0})$ as $p/n$ approaches infinity, there must exist an ascending regime for $\taustar$ as a function of $p/n.$ 

More concretely, let us view $\nustar/M$ as a function of $\delta$ within the interval $\delta \in (0, 1)$.
Rewriting the relation~\eqref{eq:epsilon-limit-2} ensures that as $\epsilon \to 0$,
one has 
	\begin{align*}
	\frac{1}{\taustar(\delta)} = \frac{\nustar(\delta)}{M} \goto \sqrt{ \frac{\a\phi(\a) - \a^2\Phi(-\a)}{\Phi(-\a)}} \eqqcolon H(\delta)
	\qquad \text{for }~\a\defn -\Phi^{-1}(\delta/2).
	\end{align*}
It can be easily verified that the function $H(\cdot)$ is a continuous and  decreasing function of $\delta$ on $(0, 1)$.  
Additionally, direct calculations yield 
\begin{align}
	\label{eqn:lim-h-function}
	\lim_{\delta\goto 1^-} H(\delta)= 0; 
	\qquad \lim_{\delta\goto 0^+} H(\delta)= 1.
\end{align}
As a result, the continuity of $H(\cdot)$ guarantees that there exists $\delta_{\mathrm{SNR}}>0$ such that 
\begin{align*} 
H(\delta) \defn \lim_{\epsilon \goto 0} \frac{\nustar(\delta)}{M} 
> \frac{1}{\sqrt{1+\mathrm{SNR}}} = \frac{1}{\sqrt{\Risk(\bm{0})}}  \in (0,1), \qquad \text{for }~\delta < \delta_{\mathrm{SNR}},
\end{align*}
where we recall $\mathrm{SNR} \defn \epsilon M^{2}$.
In other words, recognizing the relation $\taustar(\delta) \defn M/\nustar(\delta)$, we can show the existence of a regime for $\delta \in (0,1)$ where the $\epsilon$-limit of $\taustar(\delta)$ lies below $\sqrt{\Risk(\bm{0})}$. 

In addition, for any given $(\epsilon, M)$, recall that the limiting value (as $\delta \goto 0^+$) obeys 
\begin{align}
\label{eqn:lim-delta-limit-h-function}
	\lim_{\delta\goto 0^+} \frac{\nustar(\delta)}{M} 
	=
	\frac{1}{\sqrt{1+\mathrm{SNR}}} .
\end{align} 
It further implies that for any fixed $\delta_0 <\delta_{\mathrm{SNR}}$, one can find a corresponding $\epsilon_{0}$ depending on $\delta_{0}$ such that 
\begin{align*}
\taustar(\delta_0) < \lim_{\delta \to 0^+} \taustar(\delta)
\end{align*} 
holds for every $\epsilon \leq \epsilon_{0}$. 
Consequently, $\taustar(\delta)$ has an ascending phase w.r.t.~$p/n.$
Putting the above pieces together establishes the claimed result.


\section{Numerical simulations and discussion}
\label{Sec:numerics}

This section conducts numerical experiments to confirm the applicability of our results in finite samples and non-Gaussian designs. 
Along the way, we shall also point out several directions worthy of future investigation.

\paragraph*{Finite-sample behavior.}

Although the theorems obtained in the paper are asymptotic in nature, our numerical experiments suggest that they are accurate descriptions of the risk behavior even when $p$ and $n$ are on the order of 10s or 100s. 
As an illustration, we plot in Figure~\ref{fig:sample_size} two cases when $n = 100$ and $n=1000$, respectively,  with $p/n$ varying between $[10^{-2},~ 10^2]$. In these plots, the multi-descent phenomenon already manifests itself in the case when $n=100$. 
\begin{figure}[t]
	\includegraphics[width=0.49\textwidth]{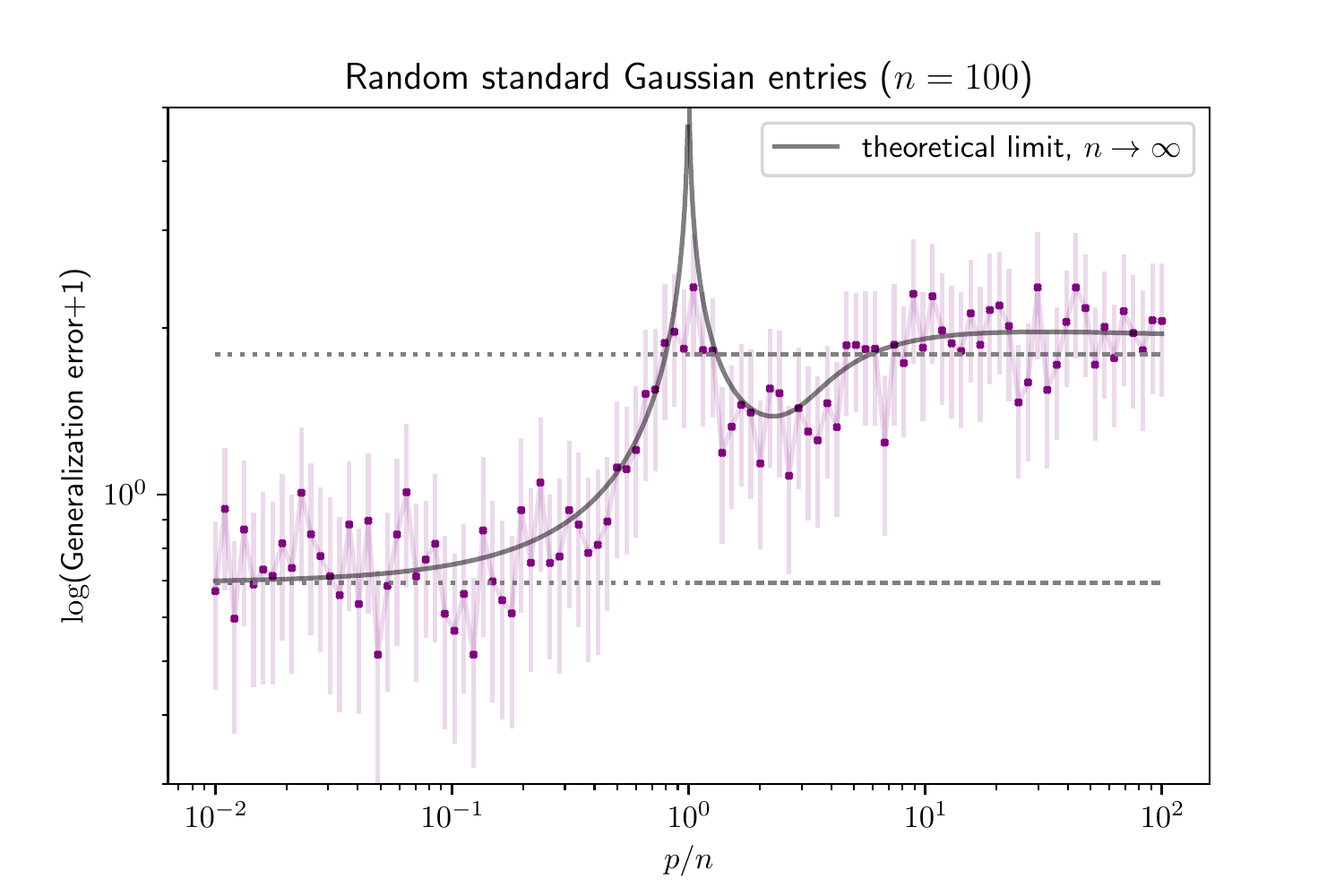}
	\includegraphics[width=0.49\textwidth]{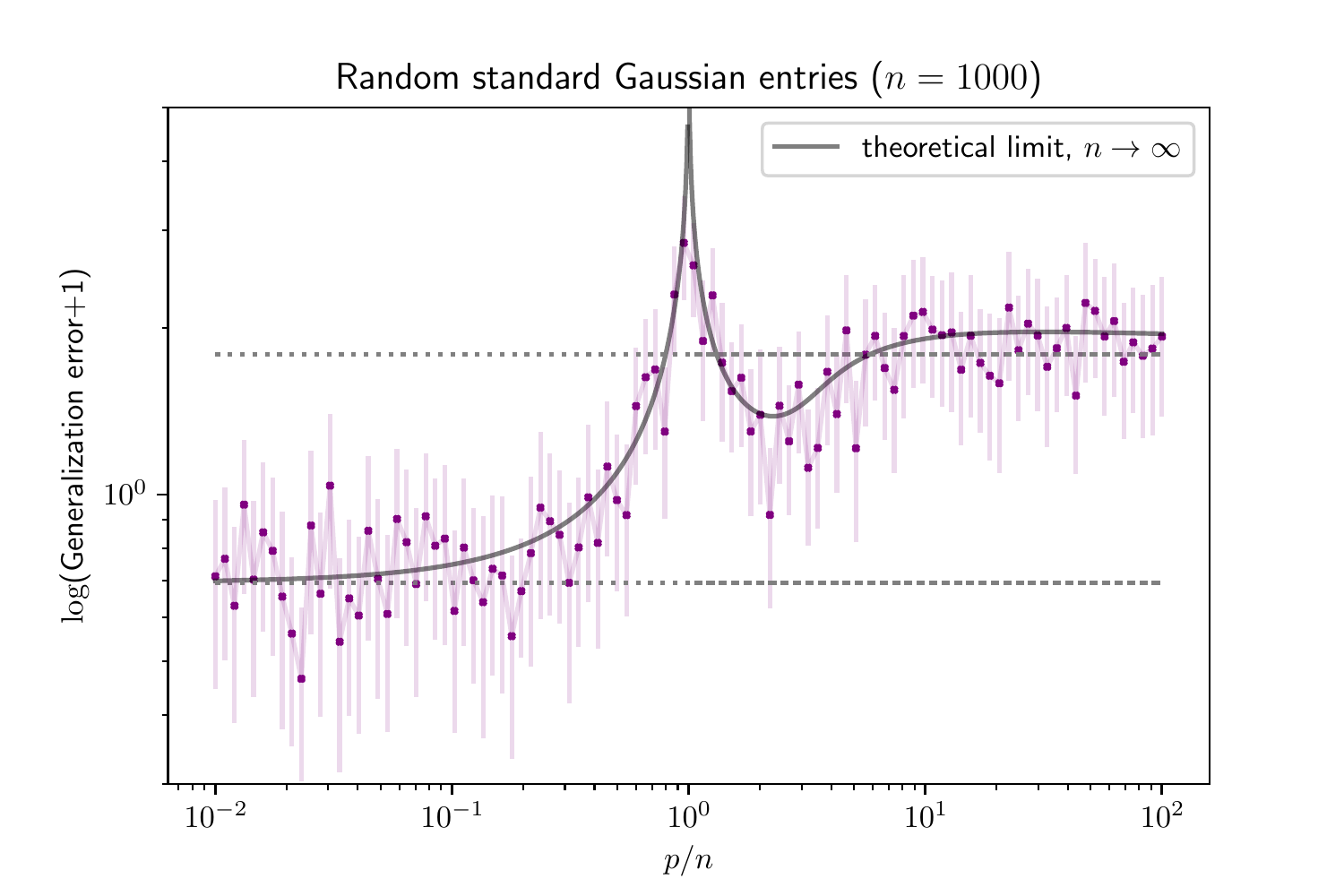}
	\caption{
	Finite-sample behavior. 
	The data are generated from a linear model~\eqref{eq:model} under i.i.d.~Gaussian design, where SNR$= 4$ and the sparsity level is $\epsilon = 0.05$. The sample size is set as $n=100$ in the left figure, and $n=1000$ in the right figure. 
	The theoretical curve, computed by solving the equations~\eqref{eqn:fix-eqn}, is displayed in solid line, where the limits with $p/n\goto 0$ and $p/n\goto\infty$ are plotted in dotted lines. 
	Here, both the $x$-axis and the $y$-axis are plotted in logarithmic scale.  
	We choose $100$ different values of $p/n$ in a way that the $\log(p/n)$'s are uniformly spaced over $[-2, 2]$.
	For each $p/n$, we generate a random instance, compute the minimum $\ell_1$-norm interpolator and its risk, and repeat this procedure for $30$ times. We report the average risk and error bar over 30 independent runs. 
	}
	\label{fig:sample_size}
\end{figure}

\paragraph*{Beyond Gaussian design.}
Thus far, our risk characterization focuses on the idealistic case with i.i.d.~Gaussian design matrices. 
There is no shortage of practical scenarios where such distributional assumptions are violated.  
To see whether our prediction continues to be valid beyond Gaussian design, 
we carry out several empirical experiments concerning design matrices that are composed of i.i.d.~non-Gaussian entries. 
Figure~\ref{fig:t-dist} illustrates two cases where the entries are generated from the Bernoulli distributions and the $t$-distribution with parameter $3$, respectively. 
Our theoretical risk characterization remains fairly accurate in these numerical experiments. 
This is perhaps not unexpected, due to a \emph{universality} phemoneon that has been justified in multiple other problems with i.i.d.~random design (see, e.g., \cite{bayati2015universality,oymak2018universality,montanari2017universality,chen2021universality}).  
These predictions might, however, be completely off when the covariates are correlated, 
meaning that the covariance structure of the design matrix plays a pivotal role in determining the shape of the risk curves.
Leveraging the current effort towards understanding Lasso under correlated designs \citep{celentano2020lasso}, we conjecture that the risk of the interpolator is dictated by a more complicated nonlinear system of equations that reflects the covariance structure.   
Given that the main message of this paper is to verify the existence of a multiple-descent phenomenon, 
we leave these more general cases to future investigation.

\begin{figure}[t]
	\includegraphics[width=0.49\textwidth]{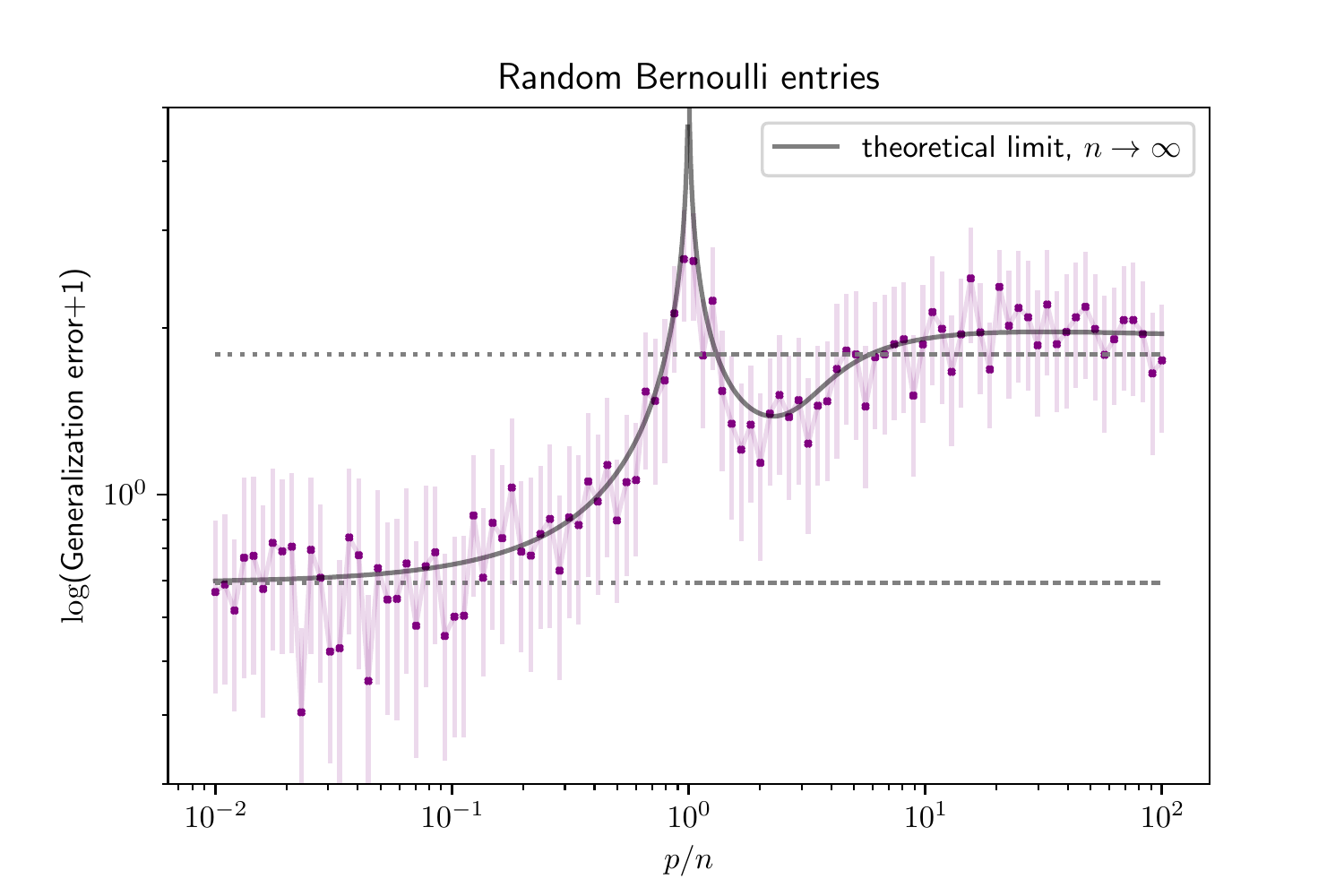}
	\includegraphics[width=0.49\textwidth]{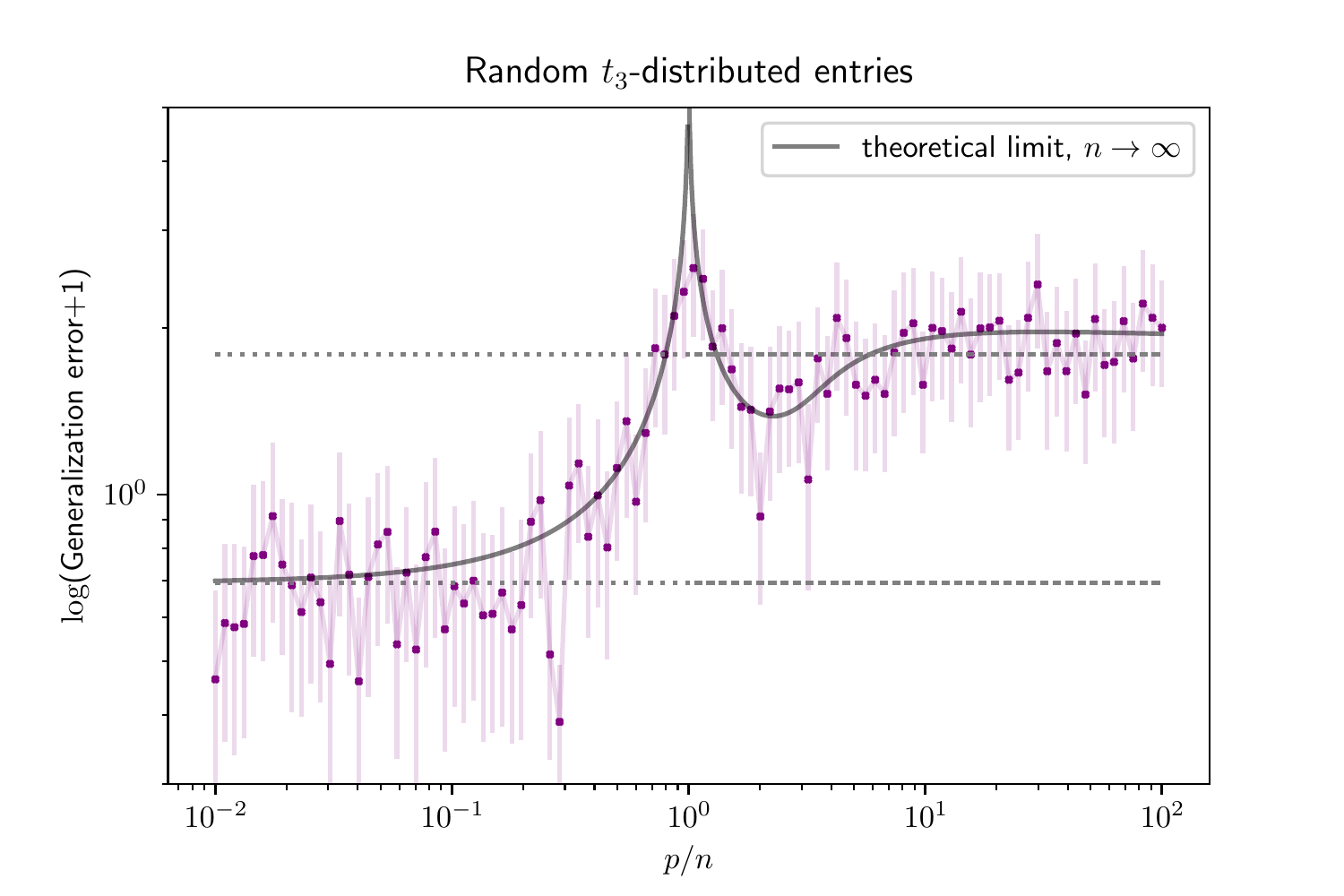}
	\caption{
	Experiments for non-Gaussian designs. 
	In these plots, the sample size is fixed as $n=100$, and the data is drawn from a linear model~\eqref{eq:model} with SNR$= 4$ and sparsity level $\epsilon = 0.05$. 
	The entries of the design matrix $\sqrt{n}\bs{X}$ are i.i.d. sampled from the $\mathsf{Bernoulli}(0.5)$ distribution for the left plot, and from $t(3)/\sqrt{3}$ distribution for the right plot (where the $1/\sqrt{3}$ is introduced to make the variance equals to 1). 
	The other experiment settings are the same with Figure~\ref{fig:sample_size}. 
	}
	\label{fig:t-dist}
\end{figure}

\paragraph*{Distributional characterization.}
 
We perform another series of numerical experiments about the minimum $\ell_1$-norm interpolators  $\widehat{\bs{\theta}}^{\mathsf{Int}}$ under i.i.d.~Gaussian design, and report in Figure~\ref{fig:coordinate} \emph{(i)} the empirical distribution of its $p$ coordinates over $30$ independent runs, and \emph{(ii)} the empirical distribution of the corresponding $\widehat{\bs{\theta}}^{\mathsf{Int}}$ coordinates when the underlying $\theta^{\star}_{i}$ is zero (resp.~non-zero). As can be seen from the plots, the estimates are close to being unbiased, with 
the estimates for non-zero entries exhibiting a higher level of uncertainty than the zero entries. However, how to develop a distributional theory remains unclear. A recent line of works \citep{bellec2019second,miolane2018distribution,celentano2020lasso} established distributional guarantees for a debiased Lasso estimator with positive regularization (so that the estimates after de-biasing exhibit Gaussian distributions). 
We conjecture that the analysis framework (via the convex Gaussian min-max theorem) developed in \cite{miolane2018distribution,celentano2020lasso} might be useful in establishing a fine-grained finite-sample distributional characterization for the interpolators of interest.

\begin{figure}[H]
	\includegraphics[width=0.33\textwidth]{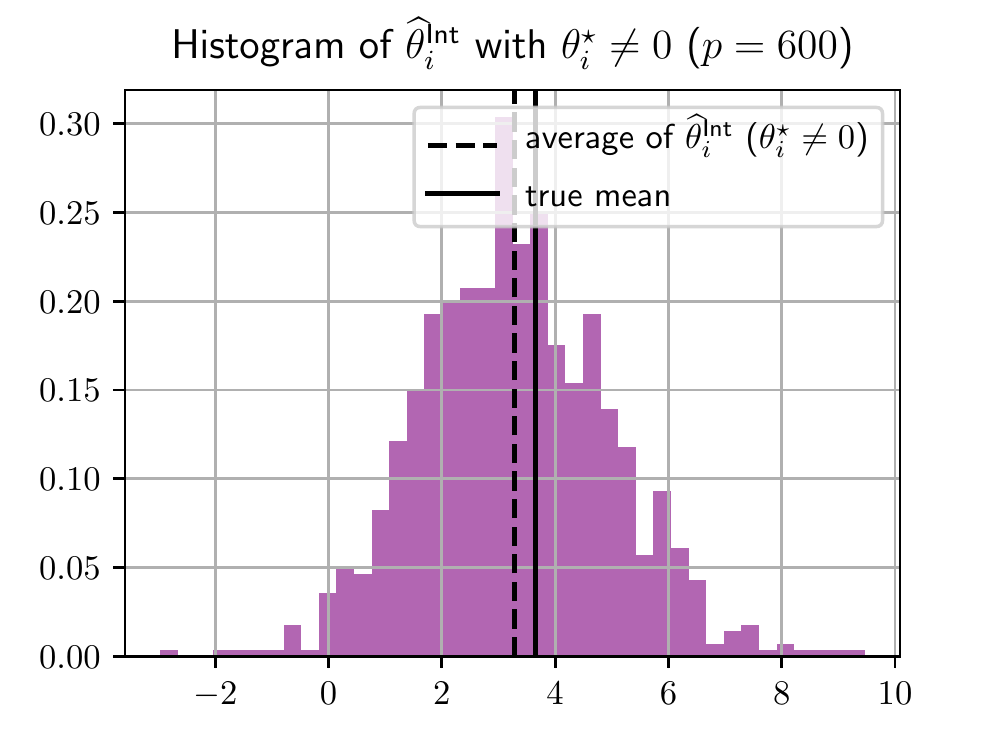}
	\includegraphics[width=0.33\textwidth]{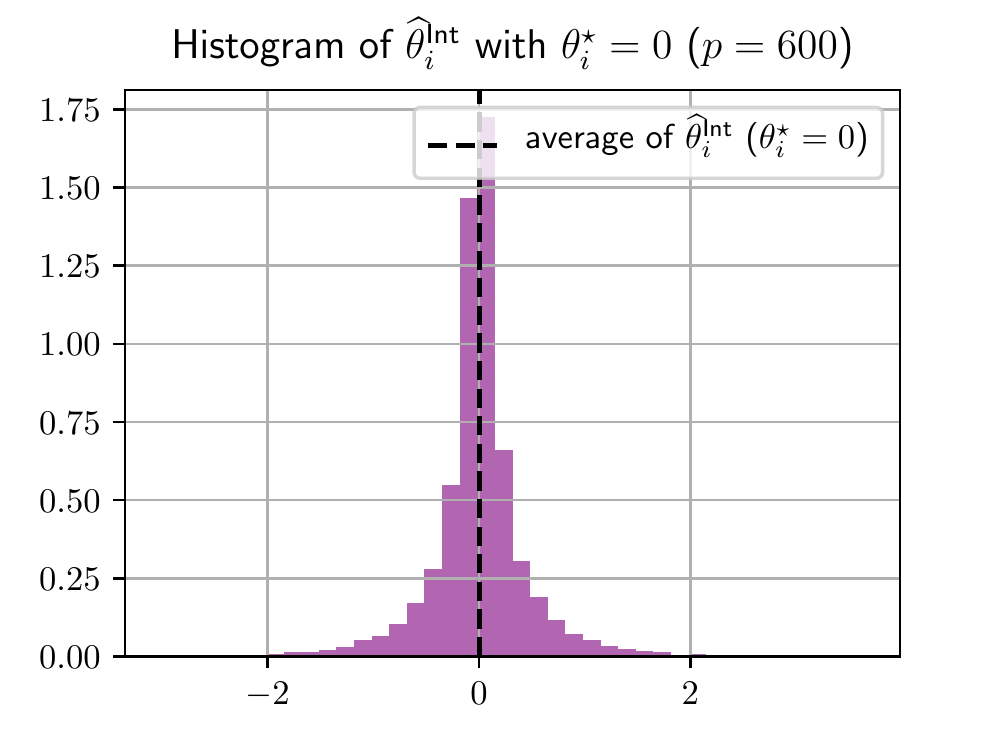}
	\includegraphics[width=0.33\textwidth]{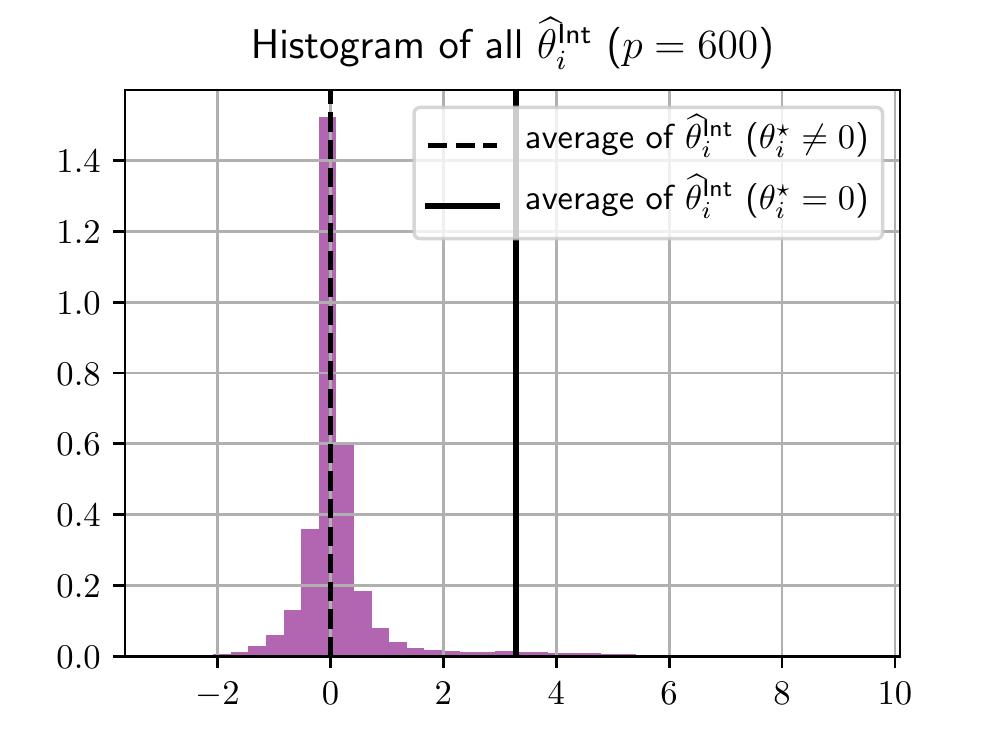}
	\includegraphics[width=0.33\textwidth]{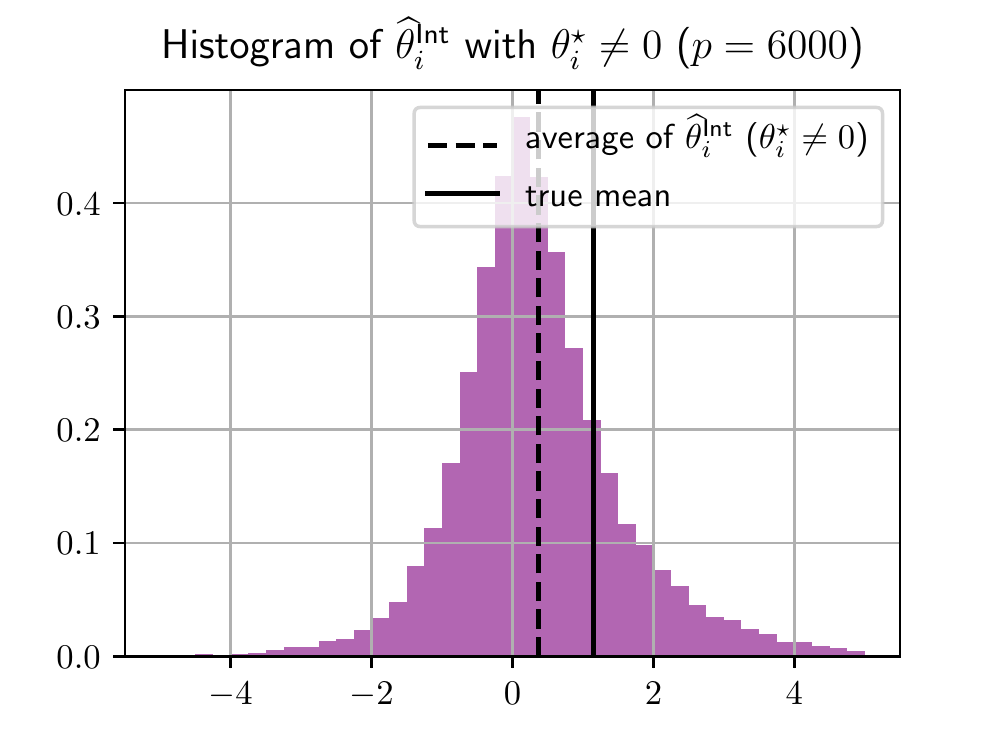}
	\includegraphics[width=0.33\textwidth]{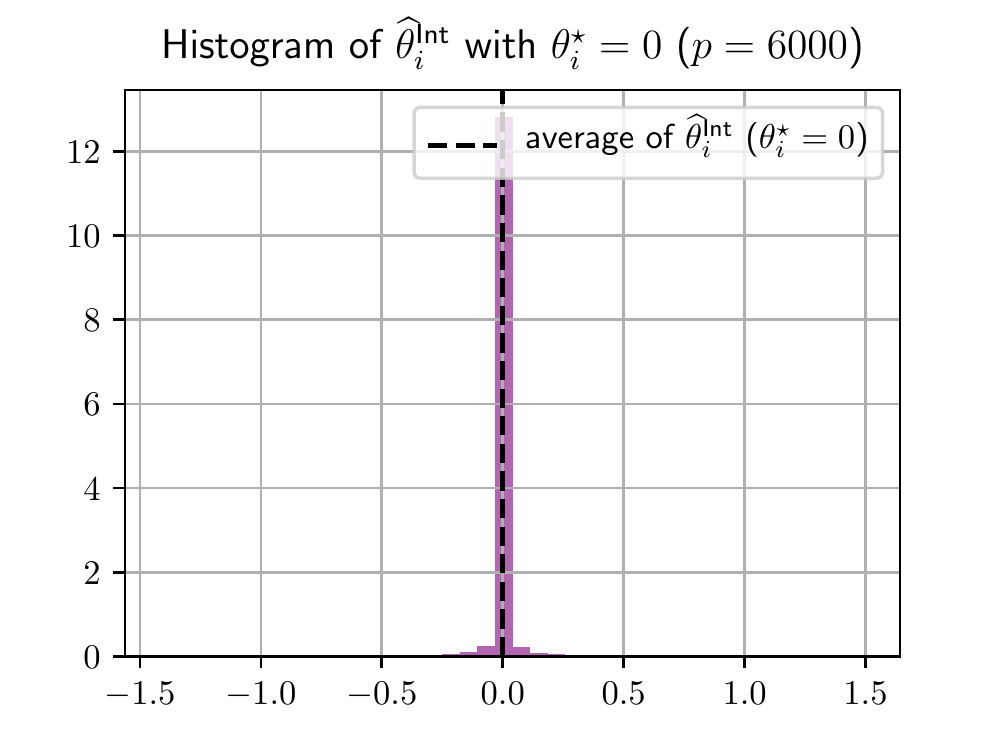}
	\includegraphics[width=0.33\textwidth]{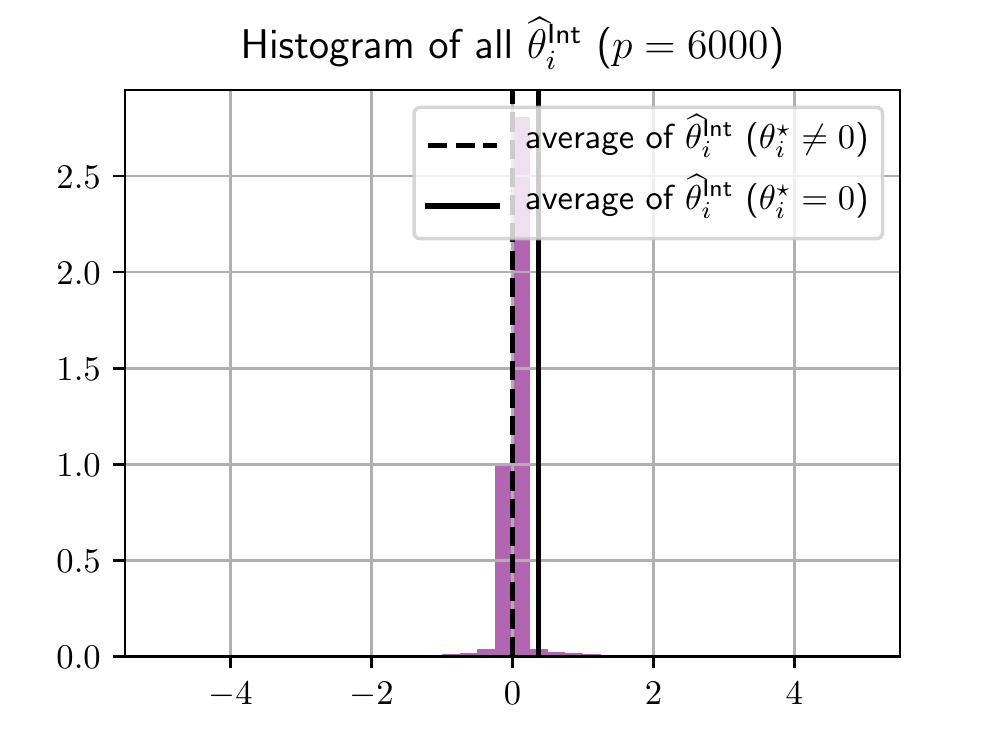}
	\caption{
	Empirical distribution for coordinates of $\thetainti$. 
	Here, we fix the sample size $n=100$, and generate data from the linear model~\eqref{eq:model} with i.i.d. Gaussian design, where SNR$ = 4$ and sparsity level $\epsilon = 0.05$.
	The other experiment settings are the same with Figure~\ref{fig:sample_size}. 
	We collect the empirical distribution of $\thetainti$'s coordinates (corresponding to those $i$ such that $\theta^\star_i\neq 0$ / $\theta^\star_i= 0$ / for every $i$, respectively) across all repeats, and generate their histograms. 
	In the top row of the plots, set $p=600$, and in the bottom row, set $p=6000$. 
	All the results reported are based on $30$ random trials. 
	The empirical averages and the ground truth $\theta^\star_{i}$ values are marked in the dotted vertical line and the solid vertical line respectively. 
	}
	\label{fig:coordinate}
\end{figure}



\section*{Acknowledgment}

The authors would like to thank Linjun Zhang for discussing this open problem with Y.~Wei when she was visiting the statistics department at Rutgers University in 2020.
This work was partially supported by the NSF grants DMS 2147546/2015447 and CCF 2106778. 
Part of this work was done while Y.~Li and Y.~Wei were visiting the Simons Institute for the Theory of Computing.

\bibliographystyle{apalike}
\bibliography{lasso,lasso2}


\begin{center}
    {\large APPENDIX}
\end{center}
\appendix


\section{Proof of Proposition~\ref{prop:lasso-limit}}
\label{sec:lasso-limit-proof}

In what follows, we intend to establish the two claims separately.

\paragraph{Case I: $n/p>1$.}
In this part, we aim to prove that, for any $\delta > 1$,  the Lasso risk converges to the risk of the least-square estimator --- denoted by $\widehat{\bs{\theta}}^{\mathsf{LS}}$ --- as $\lambda \goto 0$. 
To begin with, the risk of $\widehat{\bs{\theta}}^{\mathsf{LS}}$ can be characterized using standard random matrix theory results; see, for example, \citet[Theorem 1]{hastie2019surprises}. 
As $p\goto \infty$, one has 
\begin{align}\label{eq:ols-risk}
	\mathsf{Risk}(\widehat{\bs{\theta}}^{\mathsf{LS}}) = \sigma^2 + \frac{1}{n}\mathbb{E}\big[\|\widehat{\bs{\theta}}^{\mathsf{LS}}  -\coef\|_2^2 \big]  ~{\goto}~ \sigma^2\frac{\delta}{\delta-1}.
\end{align}

In view of the KKT condition for the corresponding loss functions, we can see that the Lasso and the least-square estimator obey 
\begin{align*}
	\widehat{\bs{\theta}}_{\lambda} -\coef = (\bs{X}^\top\bs{X})^{-1}(\bs{X}^\top\noise - \lambda \bs{s})
	\qquad \text{and} \qquad
	\widehat{\bs{\theta}}^{\mathsf{LS}} -\coef = (\bs{X}^\top\bs{X})^{-1}\bs{X}^\top\noise. 
\end{align*}
	Here, $\bs{s}=[s_j]_{1\leq j\leq p}$ denotes the sub-gradient of the $\ell_1$ norm at point $\widehat{\bs{\theta}}_{\lambda}$, which obeys $s_j \in \left[ -1, 1\right] $. Thus, the risk of the Lasso satisfies 
\begin{align*}
	\mathbb{E}\big[\|\widehat{\bs{\theta}}_{\lambda}-\coef\|_2^2\big] 
	= \mathbb{E}\big[\|\widehat{\bs{\theta}}^{\mathsf{LS}}  -\coef\|_2^2\big] - 2\lambda \mathbb{E}\big[\inprod{(\bs{X}^\top\bs{X})^{-2}\bs{X}^\top\noise}{\bs{s}}\big] + \lambda^2 \mathbb{E}\big[\|(\bs{X}^\top\bs{X})^{-1}\bs{s}\|_2^2\big], 
\end{align*}
which  combined with the Cauchy-Schwarz inequality further leads to 
\begin{align}
& \  \left| \frac{1}{n}\mathbb{E}[\|\widehat{\bs{\theta}}_{\lambda} -\coef\|_2^2] - \frac{1}{n}\mathbb{E}[\|\widehat{\bs{\theta}}^{\mathsf{LS}} -\coef\|_2^2] \right| \notag\\
%
\leq & \ 2\lambda \sqrt{\frac{1}{n}\mathbb{E}\left[\left\|  (\bs{X}^\top\bs{X})^{-1}\bs{X}^\top\noise\right\|_2^2\right] }\sqrt{\frac{1}{n}\mathbb{E}\big[\|(\bs{X}^\top\bs{X})^{-1}\|_2^2 \|\bs{s}\|_2^2 \big] } + \lambda^2 \frac{1}{n}\mathbb{E}\left[\|(\bs{X}^\top\bs{X})^{-1}\|_2^2 \|\bs{s}\|_2^2 \right] \notag\\
	\leq & \ 2\lambda \sqrt{\mathsf{Risk}(\widehat{\bs{\theta}}^{\mathsf{LS}})-\sigma^2 }\sqrt{\mathbb{E}\left[\frac{1}{\delta}\sigma_{\min}^{-4}(\bs{X}) \right] } + \frac{\lambda^2}{\delta} \mathbb{E}\left[\sigma_{\min}^{-4}(\bs{X}) \right] .
	\label{eq:least-square}
\end{align}
Now it is sufficient to control the two terms on the right-hand side above, and show that both terms converge to $0$ when $p\goto \infty. $ 
First, it has been shown in the proof of \citet[Lemma 4.1]{chen2005condition} that
\[
\mathbb{P}\left(\sigma_{\min}(\bm{X})\leq\frac{\sqrt{n}}{x^{2}}\right)<\frac{n^{n-p+1}}{(n-p+1)!}\frac{1}{x^{n-p+1}}\leq\left(\frac{e}{x}\right)^{n-p+1}
\]
for any $x>0$, 
where the last inequality comes from the well-known Stirling inequality $\sqrt{2\pi}m^{m+\frac{1}{2}}e^{-m}\leq m!$
\citep{robbins1955remark}. 
Consequently, 
\begin{align*}
\mathbb{E}\left[\sigma_{\min}^{-4}(\bm{X})\right] & \leq\bigg(\frac{2}{1-1/\sqrt{\delta}}\bigg)^{4}\mathbb{P}\left\{ \sigma_{\min}^{-1}(\bm{X})\leq\frac{2}{1-1/\sqrt{\delta}}\right\} +{\displaystyle \int}_{\big(\frac{2}{1-1/\sqrt{\delta}}\big)^{4}}^{\infty}\mathbb{P}\left\{ \sigma_{\min}^{-4}(\bm{X})>z\right\} \mathrm{d}z\\
 & \leq2\bigg(\frac{2}{1-1/\sqrt{\delta}}\bigg)^{4}+{\displaystyle \int}_{\big(\frac{2}{1-1/\sqrt{\delta}}\big)^{4}}^{\infty}\mathbb{P}\left\{ \sigma_{\min}(\bm{X})<\frac{1}{z^{1/4}}\right\} \mathrm{d}z\\
 & \leq2\bigg(\frac{2}{1-1/\sqrt{\delta}}\bigg)^{4}+{\displaystyle \int}_{\big(\frac{2}{1-1/\sqrt{\delta}}\big)^{4}}^{\infty}\left(\frac{e}{z^{1/8}n^{1/4}}\right)^{n-p+1}\mathrm{d}z\\
 & \leq4\bigg(\frac{2}{1-1/\sqrt{\delta}}\bigg)^{4}
\end{align*}
for sufficiently large $n$. 
In addition, by virtue of \eqref{eq:ols-risk}, it is guaranteed that 
\begin{align*}
	 \sqrt{\mathsf{Risk}(\widehat{\bs{\theta}}^{\mathsf{LS}})-\sigma^2 } \goto \sqrt{\frac{\sigma^2}{\delta-1}}.
\end{align*}
Substitution into \eqref{eq:least-square} yields 
\begin{align*}
	\lim_{\lambda\goto 0} \lim_{n\goto\infty} \frac{1}{n}\mathbb{E}[\|\widehat{\bs{\theta}}_{\lambda} -\coef\|_2^2] = \lim_{n\goto\infty} \frac{1}{n}\mathbb{E}\big[\|\widehat{\bs{\theta}}^{\mathsf{LS}} -\coef\|_2^2 \big]
\end{align*}
for any $0<\delta <1$ that is strictly bounded away from $1$.

\paragraph{Case II: $n/p < 1$.}

When $\delta < 1$, our goal is to demonstrate that $$\lim_{\lambda \goto 0} \lim_{p \goto \infty} \mathsf{Risk}(\widehat{\t}_{\lambda}) = \lim_{p \goto \infty} \mathsf{Risk}(\tint).$$ 
To this end, let us first state one known result about $\mathsf{Risk}(\widehat{\t}_{\lambda})$. 
Specifically, the lemma below associates the Lasso risk with the solution to the system of equations \eqref{eq:nonzero-lambda-1} and \eqref{eq:nonzero-lambda-2}.
\begin{lem}[Corollary 1.6 in \cite{bayati2011lasso}]\label{lem:lasso-risk}
	The system of equations \eqref{eq:nonzero-lambda} admits one unique solution pair $(\taustar(\lambda), \alphastar(\lambda))$. 
	With the Lasso problem formulated in \eqref{eqn:lasso-definition}, 
	it holds that 
	\begin{align*}
		\lim_{p \goto \infty} \mathsf{Risk}(\widehat{\t}_{\lambda}) = \big( {\taustar(\lambda)} \big)^2 ,
	\end{align*}
	as long as $\mathbb{P}(\coef \neq \bs{0}) > 0$. 
\end{lem}
In view of Theorem~\ref{thm:interpolation-risk}, it is guaranteed that 
\begin{align*}
\lim_{p \goto \infty} \mathsf{Risk}(\tint) = {\taustar}^2.
\end{align*}
Therefore, to obtain the desired conclusion, it suffices to show that $\lim_{\lambda \goto 0}\taustar(\lambda) = \taustar$, where $\taustar(\lambda),$ and $\taustar$ correspond to the solution to a different set of equations respectively. 
Equivalently, for any converging sequence $\left\lbrace \lambda_t\right\rbrace_{t=1}^{+\infty} $ with $\lambda_t > 0$ and $\lambda_t \goto 0$, denote the corresponding $(\tau^{\star}(\lambda_t), \alpha^{\star}(\lambda_t))$ sequence as $\left\lbrace (\tau^{\star}_t, \alpha^{\star}_t)\right\rbrace $. We now aim to show that the $\lim_{t\goto \infty} \tau^{\star}_t = \tau^{\star}$. 

In order to achieve this goal, we make two useful observations. 
First, as will be demonstrated in Lemma~\ref{lem:property-of-solution-bm11}(1), we know that as $\lambda_t \goto 0$, $\left\lbrace \alpha^{\star}_t\right\rbrace $ is a non-increasing sequence and is lower bounded by $\a_{\min}(\delta)$. Therefore, $\left\lbrace\alpha^{\star}_t \right\rbrace$ has a finite and positive limit; we shall denote this limit by $\alpha^{\star}_{\infty}$. 
In addition, applying Lemma~\ref{lem:property-of-solution-bm11}(2) ensures that $\left\lbrace \tau^{\star}_t = \tau_*(\alpha^{\star}_t) \right\rbrace $ converges; we shall denote the limiting value by $\tau^{\star}_{\infty}$. 
Consequently, taking $t\goto \infty$ on both sides of
\begin{align*}
	\lambda_t= \alpha^{\star}_t \tau^{\star}_t \left( 1 - \frac{1}{\delta}\mathbb{E}\left[ \eta'(\Theta + \tau^{\star}_t Z; \alpha^{\star}_t\tau^{\star}_t)\right] \right),
\end{align*}
(note that the right-hand side is continuously differentiable with respect to both parameters) leads to the observation that $(\alpha^{\star}_{\infty}, \tau^{\star}_{\infty})$ yields $\delta =\mathbb{P}\left( |\Theta+\tau^{\star}_{\infty} Z| > \alpha^{\star}_{\infty} \tau^{\star}_{\infty}\right)$, thus solving the equation~\eqref{eq:fix-2}. 
Similarly, one can show that $(\alpha^{\star}_{\infty}, \tau^{\star}_{\infty})$ solves the equation~\eqref{eq:fix-1}. 

Putting these pieces together and using the uniqueness of the solution to \eqref{eq:fix-1} and \eqref{eq:fix-2}, we arrive that $\lim_{t\goto \infty} \tau^{\star}_t = \tau^{\star}_{\infty} = \tau^{\star}$. 
We have thus established the advertised property.


\section{Properties of the state evolution parameters}
\label{sec:proof-state-evolution}

In this section, we collect some results about the state evolution parameters $\left\lbrace \alphastart, \taustart, \tau_t \right\rbrace_{t=1}^{\infty}$. We remind the readers that $\tau_t$ is the state evolution in the $t$-th iteration,
while $(\alphastart, \taustart)=(\alphastar(\lambda_t), \taustar(\lambda_t))$ represents the fixed point of the state evolution recursion with $\lambda_t$.


\subsection{Main results }
\label{sec:solution-property-lambda}

We first make note of several useful results about the solutions to the equations~\eqref{eq:nonzero-lambda-1} and \eqref{eq:nonzero-lambda-2}, which have been proved in \cite{bayati2011lasso}. 
Before proceeding, first recall the following mapping $\mathsf{F}(\tau^2, \zeta)$ previously introduced in \eqref{eq:amp-formula-linear-2}: 
\begin{align*}
	\mathsf{F}(\tau^2, \zeta) =\sigma^2 + \frac{1}{\delta}\mathbb{E}\left\lbrace \big[ \eta(\Theta + \tau Z; \zeta) - \Theta\big]^2\right\rbrace, 
\end{align*}
where $Z\sim\NORMAL(0, 1)$ is independent of $\Theta$.
We also recall that $\a_{\min} = \a_{\min}(\delta)$ corresponds to the non-negative solution of the equation
\begin{align*}
(1+\a^2)\Phi(-\a) - \a\phi(\a) = \frac{\delta}{2}.
\end{align*}
We now record the following properties. 
\begin{lem}
	[Proposition 1.3, Proposition 1.4, Corollary 1.7  in \cite{bayati2011lasso}]
	\label{lem:property-of-solution-bm11}
	The solution to the equations \eqref{eq:nonzero-lambda-1} and \eqref{eq:nonzero-lambda-2} obeys 
	\begin{enumerate}
		\item For any $\a > \a_{\min}(\delta)$, the first equation $\tau^2 = \mathsf{F}(\tau^2, \a\tau)$ admits a unique solution; denote this solution as $\taustar = \taustar(\a)$. Additionally we know that $\a \mapsto \taustar(\a)$ is continuously differentiable on $(\a_{\min}(\delta), +\infty)$.
		\item For any $\lambda > 0$,  there exists one unique $\a>0$ satisfying \eqref{eq:nonzero-lambda-1} and \eqref{eq:nonzero-lambda-2} with $\a > \a_{\min}(\delta)$, and the mapping from $\lambda > 0$ to $\a$ is continuous, differentiable and non-decreasing. Its inverse mapping
		\begin{align}
		\label{eq:defn-lambda-alpha}
		\lambda(\a) \defn \a\tau^\star(\a)\left[1 - \frac{1}{\delta}\mathbb{E}\left\lbrace \eta'(\Theta + \tau^\star(\a)Z; \a \tau^\star(\a))
		\right\rbrace \right]
		\end{align}
		is continuous and differentiable on $(\a_{\min}(\delta), +\infty)$, with $\lambda(\a_{\min}(\delta)+) = -\infty$ and $\lim_{\a\goto\infty} \lambda(\a)= +\infty$.
		\item Combining 1 and 2, we can see that for any $\lambda > 0$, there is a unique pair  $(\alphastar(\lambda), \taustar(\lambda))$ that solves \eqref{eq:nonzero-lambda-1} and \eqref{eq:nonzero-lambda-2}. 
	\end{enumerate}
\end{lem}

Through careful investigations about the constructed mappings $\taustar(\a)$ and $\alphastar(\lambda)$ for every $\lambda>0$, we establish the existence and uniqueness of the solution to the equations~\eqref{eq:fix-1} and \eqref{eq:fix-2}. The proof of this result is provided in Section~\ref{sec:uniqueness-proof}.
\begin{prop}[Uniqueness of the solution]
\label{lem:uniqueness}
	When $\mathbb{P}(\Theta \neq 0) > 0$, there exists one unique pair $(\alphastar, \taustar)$ with $ \alphastar > \a_{\min}(\delta)$ that satisfies \eqref{eq:fix-1} and \eqref{eq:fix-2}.
\end{prop}
It turns out that both $\taustar(\a)$ and $\alphastar(\lambda)$ are Lipschitz functions of $\lambda$, which can be rigorized in the following proposition.  
The proof of this result can be found in Section~\ref{sec:uniqueness-proof}. 

\begin{prop}[Convergence of $\alphastart$ and $\taustart$]
\label{prop:amp-fixd-point-convergence}
The following properties hold for the solutions to the equation set~\eqref{eq:fixed-point} as $t\goto\infty$.
	\begin{enumerate}
		\item $\alphastart \goto \alphastar$, $\taustart \goto \taustar$; the sequence $\left\lbrace \alphastart\right\rbrace $  increases  monotonically with $t$ when $t$ is large enough.

		\item The function $\a^\star(\lambda)$ is continuously differentiable w.r.t.~$\lambda$, and there exist some constants $c$ and $C$ determined by $\Theta$, $\sigma$ and $\delta$, such that $c\leq {\a^\star}'(0)\leq C$; as such, we can find some $L_\a$ determined by $\Theta$, $\sigma$ and $\delta$, such that
		\begin{align*}
			|\alphastart-\alpha^\star_{t+1}|\leq L_\a |\lambda_t-\lambda_{t+1}|.
		\end{align*}

		\item 
		The function $\taustar(\lambda)$ is Lipschitz w.r.t.~$\lambda$ for some $L_\tau$ determined by $\Theta$, $\sigma$ and $\delta$; in particular,
		\begin{align*}
			|\taustart-\tau^\star_{t+1}|\leq L_\tau |\lambda_t - \lambda_{t+1}|.
		\end{align*} 
	\end{enumerate}
\end{prop}

As a direct consequence of the third claim above, for every $t\geq 1$,
one has 
\begin{align*}
	\taustart \leq  \taustar + L_\tau \max_{i} \lambda_i \eqqcolon \tau_{\max}^\star; \quad \alphastart \leq \alphastar + L_\a \max_i \lambda_i =: \a_{\max}^\star.
\end{align*}
Finally, we can demonstrate that: the state evolution sequence $\tau_t$ (defined in~\eqref{eqn:state-evolution}) approaches the solution of the equations~\eqref{eq:fix-1} and \eqref{eq:fix-2} as $t\goto\infty$. The proof is deferred to Section~\ref{sec:convergence-tau}.

\begin{prop}[Convergence of the state evolution]
\label{prop:tau-t-limit}
	The state evolution sequence obeys $\tau_t \goto \taustar$ as $t\goto \infty$.
\end{prop}

\subsection{Proof of Proposition~\ref{lem:uniqueness} and Proposition~\ref{prop:amp-fixd-point-convergence}}
\label{sec:uniqueness-proof}

To begin with, we introduce the following result on the derivatives of $\lambda'(\a)$ and ${\tau^\star}'(\a)$, whose proof is provided in Section~\ref{sec:proof-derivative-lemma}.

\begin{lem}\label{lem:derivative-lemma}
	For any $\a_0 > \a_{\min}(\delta)$ with $\lambda(\a_0)=0$ (cf.~\eqref{eq:defn-lambda-alpha}), we have 
	\begin{gather}
	\label{eq:lambda-bound} 0 < C_1 <  \lambda'(\a_0) < C_2; \\
	\label{eq:tau-alpha-bound} \taustar'(\a_0) \leq C_3.
	\end{gather}
	Here, $C_1, C_2, C_3$ are constants that depend only on $\a_0, \delta$ and $\Theta$.
\end{lem}
We are ready to prove Proposition~\ref{lem:uniqueness} and Proposition~\ref{prop:amp-fixd-point-convergence} with the assistance of Lemma~\ref{lem:derivative-lemma}.

\paragraph{Proof of Proposition~\ref{lem:uniqueness}.}
Consider the function $\lambda(\a)$ defined in \eqref{eq:defn-lambda-alpha}. Suppose there exist two different $\a_0\neq \a_0'$ where $\lambda(\a_0) = \lambda(\a_0') = 0$; by Lemma~\ref{lem:derivative-lemma}, we know $\lambda'(\a_0)$ and $\lambda'(\a_0')$ are both positive and bounded away from $0$. 
Thus, in view of the continuity of $\lambda(\cdot)$, there must exist some $\lambda_0 > 0$ such that $\lambda(\a)=\lambda_0$ has at least two solutions. This, however, contradicts Lemma~\ref{lem:property-of-solution-bm11}(2).
	
\paragraph{Proof of Proposition~\ref{prop:amp-fixd-point-convergence}.}
We consider each claim separately. 
For the first claim, by virtue of the second statement in Lemma~\ref{lem:property-of-solution-bm11}, the mapping $\a\mapsto\lambda(\a)$ is continuous and differentiable on $(\a_{\min}, +\infty)$, and for any $\lambda \geq 0$, the solution of $\lambda(\a)=\lambda$ exists and is unique. As such, the inverse mapping $\lambda\mapsto\alphastar(\lambda)$ is well-defined on $ \lambda \geq 0$ and is continuous. Then we can see that $\alphastart \to\alphastar$. 
Moreover, recognizing that $\taustart = \taustar(\alphastart)$ and $\taustar = \taustar(\alphastar)$ and using the continuously differentiable mapping $\a\mapsto\taustar(\a)$ defined in Lemma~\ref{lem:property-of-solution-bm11}(1), we conclude that $\taustart \to \taustar$. 
Lastly, in light of inequality~\eqref{eq:lambda-bound}, 
when $\lambda_t \neq \lambda_{t+1}$ we have
\begin{align*}
	\frac{\alphastart - \alpha^\star_{t+1}}{\lambda_t - \lambda_{t+1}} \to \alphastar'(0) \geq \frac{1}{C_1}.
\end{align*}
We can thus conclude that $\alphastart$ is monotonously increasing when $t$ is large enough.

We now turn to the second claim. As ensured by Proposition~\ref{lem:uniqueness}, there exists one unique solution to $\lambda(\a_0)$; therefore, the expression~\eqref{eq:lambda-bound} translates to 
\begin{align*}
	0 < C_1 < \lambda'(\alphastar) < C_2,
\end{align*}
or equivalently, $C_2^{-1} < \alphastar'(0) < C_1^{-1}$.
We have thus finished the proof of the second claim. 

The third claim follows directly from the second claim and the bound~\eqref{eq:tau-alpha-bound}, and the proof of Proposition~\ref{lem:uniqueness} is completed. Finally, we make the remark that $L_\tau$ and $L_\a$ depend on $C_1, C_2$ and $C_3$ from Lemma~\ref{lem:derivative-lemma}; from Proposition~\ref{lem:uniqueness}, we know that $\a_0=\alphastar$ is the unique value of $\a_0$ satisfying Proposition~\ref{lem:uniqueness}. Therefore, the expressions~\eqref{eq:lambda-lower} \eqref{eq:lambda-upper}\eqref{eq:e2-bound} and \eqref{eq:e3-bound} with $\a_0=\alphastar$ and $\taustar(\a_0)=\taustar$ give us the explicit from of $L_\tau$ and $L_\a$ in terms of $\alphastar$ and $\taustar$.

\subsubsection{Proof of Lemma~\ref{lem:derivative-lemma}}\label{sec:proof-derivative-lemma}

We first make note of an inequality proved in \cite[Lemma A.5]{miolane2018distribution} as follows:
\begin{align}
\label{eq:form-c3}
\taustar'(\a_0) \leq (\a_0 + 1)\frac{\taustar^3(\a_0)}{\delta\sigma^2} =: C_3.
\end{align}
which validates the inequality~\eqref{eq:tau-alpha-bound}. 

It then suffices to establish the first inequality.
Towards this, let us first derive the explicit expressions for the derivative, and prove the upper and lower bounds. 
For any $\a > \a_{\min}(\delta)$, define $\Xi\defn \Theta/\tau^\star(\a)$, and the following quantities:
\begin{gather*}
	E_1 \defn \mathbb{E}\left[ \Phi(-\Xi-\a) + \Phi(\Xi-\a)\right] , \quad E_2 \defn \mathbb{E}\left[ \Xi\phi(-\Xi-\a) - \Xi\phi(\Xi-\a)\right],\\
	E_3 \defn \mathbb{E}\left[ \phi(- \Xi-\a) + \phi(\Xi-\a)\right], \quad E_4 \defn \mathbb{E}\left[\Xi^2 \left[\Phi(\a-\Xi) - \Phi(-\a-\Xi)\right]\right].
\end{gather*}
It is easily seen that $E_1, E_2, E_3$ and $E_4$ are continuous and differentiable with respect to $\a$. 

\paragraph{Explicit expression for the derivatives at $\a=\a_0$.}
Let us first derive ${\tau^\star}'(\a)$ with $\taustar(\a)$ defined in Lemma~\ref{lem:property-of-solution-bm11}(1). 
Recall the definition $\mathsf{F}(\tau^2, \zeta) =\sigma^2 + \frac{1}{\delta}\mathbb{E}\left\lbrace \left[ \eta(\Theta + \tau Z; \zeta) - \Theta\right]^2\right\rbrace$; 
direct calculations yield 
\begin{align*}
\frac{\partial \mathsf{F}}{\partial \a}(\tau^2, \tau\a) = & \frac{2{\tau^\star(\a)}^2}{\delta}\left\lbrace \a \mathbb{E} \left[ \Phi(\Xi-\a) + \Phi(-\Xi-\a) \right] - \mathbb{E}\left[\phi(\Xi-\a) + \phi(-\Xi-\a)\right]\right\rbrace 
= \frac{2{\tau^\star(\a)}^2}{\delta}\left[\a E_1 - E_3\right] ;\\
\frac{\partial \mathsf{F}}{\partial \tau}(\tau^2, \tau\a) = & \frac{2\tau^\star(\a)}{\delta} \left\lbrace (1+\a^2) \mathbb{E}\left[\Phi(\Xi-\a)+\Phi(-\Xi-\a)\right] - \mathbb{E}\left[(\Xi+\a)\phi(\Xi-\a) - (\Xi-\a)\phi(-\Xi-\a) \right]\right\rbrace \\
=&  \frac{2\tau^\star(\a)}{\delta}\left[(1+\a^2)E_1 + E_2 - \a E_3\right].
\end{align*}
Invoking the implicit function theorem, we can guarantee that 
\begin{align*}
{\tau^\star}'(\a) = \frac{\partial_{\a}\mathsf{F}(\tau^2, \a\tau)}{2\tau-\partial_{\tau}\mathsf{F}(\tau^2, \a\tau)}
= - \frac{\tau^\star(\a)\left[\a E_1 - E_3\right]}{(1+\a^2)E_1 + E_2 - \a E_3 - \delta}.
\end{align*}
For $\a_0$ with $\lambda (\a_0)=0$, we know $E_1=\delta$, and then 
\begin{align*}
{\tau^\star}'(\a_0)
= - \frac{\tau^\star(\a_0)\left[\a_0 \delta- E_3\right]}{\a_0^2\delta + E_2 - \a_0 E_3}.
\end{align*}
In terms of $\lambda'(\a)$, one has 
\begin{align*}
\frac{\mathrm{d}}{\mathrm{d}\a} \mathbb{E}\left[ \eta'(\Theta + \tau _\star(\a)Z; \tau^\star(\a)\a)\right] 
& = \mathbb{E}\left[ (-1+\Xi\frac{{\tau^\star}'(\a)}{\tau _\star(\a)})\phi(-\Xi-\a) + (-1-\Xi\frac{{\tau^\star}'(\a)}{\tau _\star(\a)})\phi(\Xi-\a)
\right] \\
& = - \frac{\a\left[\a E_1E_3 - E_3^2 + E_1 E_2 \right] + E_3(E_1-\delta )}{(1+\a^2)E_1 + E_2 - \a E_3 - \delta} .
\end{align*}
Finally, we are ready to calculate the derivative $\lambda'(\a)$ at $\a=\a_0$. By expression~\eqref{eq:nonzero-lambda}, and noticing that $\lambda(\a_0)=0$ and $E_1=\delta$ in this case, we calculate
\begin{equation}
\label{eqn:lambda-derivative}
\begin{aligned}
\lambda'(\a_0) & = \frac{\a_0 {\tau^\star}'(\a_0) + \tau^\star(\a_0)}{\a_0 \tau^\star(\a_0)} \lambda(\a_0) - \left. \frac{\a_0\tau^\star(\a_0) }{\delta }\frac{d}{d\a} \mathbb{E}\left[ \eta'(\Theta + \tau^\star(\a) Z; \tau^\star(\a)\a)\right] \right|_{\a=\a_0} \\
& = \frac{\a_0^2 \tau^\star(\a_0)}{\delta} \cdot \frac{\delta \a_0 E_3  - E_3^2 +\delta E_2 }{\delta\a_0^2 + E_2 - \a_0 E_3}.
\end{aligned}
\end{equation}

\paragraph{Bounding the derivatives.}
To establish the inequality~\eqref{eq:lambda-bound}, it is sufficient for us to control the following quantities 
\begin{align*}
	\delta \a_0 E_3  - E_3^2 +\delta E_2
	~\text{ and }~
	\delta\a_0^2 + E_2 - \a_0 E_3,
\end{align*}
respectively. 
Let us first express the function $\mathsf{F}(\tau^2, \tau\a)$ in an explicit fashion. 
For every fixed $\xi$, we have 
\begin{align*}
\mathbb{E}\left[ \eta(\xi + Z; \alpha) - \xi \right]^2 
= & (-\a-\xi)\phi(\a-\xi) + (\a^2+1)\Phi(-\a+\xi) + (-\a+\xi)\phi(\a+\xi) + (\a^2+1)\Phi(-\a-\xi) \\
& + \xi^2 \left[\Phi(\a-\xi) - \Phi(-\a-\xi)\right].
\end{align*}
Taking expectation with respect to $\xi = \Theta/\tau^\star(\a)$ on both sides, we arrive at 
\begin{align}
\label{eqn:tmp-mozart}
{\tau^\star(\a)}^2 = \mathsf{F}({\tau^\star(\a)}^2, \tau^\star(\a)\a) =\sigma^2 +\frac{{\tau^\star(\a)}^2 }{\delta}
\left[ (\a^2 + 1)E_1  + E_2 -\a E_3  + E_4 \right].
\end{align}
For $\a_0 > \a_{\min}(\delta) > 0$ with $\lambda(\a_0)=0$, one has $E_1=\delta$ and $E_4 \geq 0$. 
As a result, the equation~\eqref{eqn:tmp-mozart} leads to 
\begin{align}\label{eq:denom-lower-bound}
	\a_0^2\delta  + E_2 -\a E_3  \leq - \sigma^2\delta {{\tau^\star}^{-2}(\a_0)} < 0.
\end{align}

Next we control the quantity $\delta \a_0 E_ 3- E_3^2 + \delta E_2$ by looking at two cases separately. 
Observing that $\a_0^2\delta  + E_2 -\a E_3  \leq - \sigma^2\delta {{\tau^\star}^{-2}(\a_0)}$, $E_3\geq 0$ and $E_2 \leq 0$, we know
\begin{align*}
\delta \a_0 E_ 3- E_3^2 + \delta E_2 < E_3 \left[ - \frac{\sigma^2 \delta}{{\tau^\star}^{2}(\a_0) \a_0} - \frac{E_2}{\a_0}
\right] + \delta E_2 = - \frac{\sigma^2 \delta}{{\tau^\star}^{2}(\a_0) \a_0} E_3 +  \frac{\a_0 \delta - E_3}{\alpha_0} E_2.
\end{align*}
Then we have 
\begin{align*}
\delta \a_0 E_ 3- E_3^2 + \delta E_2 \leq \left\lbrace 
\begin{aligned}
& - \frac{\sigma^2 \delta}{{\tau^\star}^2(\a_0) \a_0} E_3, & & \alpha_0 \delta - E_3 > 0;\\
& \delta E_2, & & \alpha_0 \delta - E_3 \leq 0.
\end{aligned} \right. 
\end{align*}

Taking the above properties collectively with~\eqref{eqn:lambda-derivative}, we now move on to prove the conclusion in \eqref{eq:lambda-bound} for two cases respectively, namely,  
\begin{align}
\label{eq:lambda-lower}
\lambda'(\a_0) \left\lbrace 
\begin{aligned} 
&\geq \frac{\sigma^2\a_0}{\tau^\star(\a_0)} \frac{E_3}{|E_2| - \a_0\left[\delta\a_0-  E_3\right] } \geq \frac{\sigma^2\a_0}{\tau^\star(\a_0)} \frac{E_3}{|E_2|}, & \alpha_0 \delta - E_3 > 0;\\
&= \frac{\a_0^2 \tau^\star(\a_0)}{\delta}\frac{E_3\left[ \alpha_0 \delta - E_3 \right] +\delta E_2 }{\a_0 \left[ \alpha_0 \delta - E_3 \right] + E_2}\geq  \frac{\a_0^2 \tau^\star(\a_0)}{\delta} \min\left\lbrace \delta,  \frac{E_3}{\a_0} \right\rbrace = \min\left\lbrace \a_0^2 \tau^\star(\a_0), \frac{\a_0 \tau^\star(\a_0)E_3}{\delta} \right\rbrace , & \alpha_0 \delta - E_3 \leq 0.
\end{aligned}\right.
\end{align}

For the upper bound of $\lambda'(\a_0)$, we can directly verify that
\begin{align*}
\delta \a_0 E_3  - E_3^2 +\delta E_2 \geq \delta E_2 - 4,
\end{align*}
which combined with the expression~\eqref{eq:denom-lower-bound} leads to
\begin{align}\label{eq:lambda-upper}
\lambda'(\a_0) \leq \frac{\a_0^2 \tau^\star(\a_0)}{\delta} \frac{\delta|E_2|+4}{\sigma^2 \delta \taustar^{-2}(\a_0)}.
\end{align}

From our assumption $\mathbb{E}[\Theta^2] < \infty$, we can find $M$, such that $$\mathbb{E}[\Theta^2\textbf{1}\left\lbrace |\Theta|\leq M\right\rbrace ] \geq \mathbb{E}[\Theta^2]/2.$$ Then we know that $\mathbb{P}(|\Theta|\leq M) \geq \mathbb{E}[\Theta^2/M^2 \textbf{1}\left\lbrace |\Theta|\leq M\right\rbrace]\geq \mathbb{E}[\Theta^2]/(2M^2)$. Combining with the definition of $E_3$ yields
\begin{align}\label{eq:e3-bound}
E_3 \geq \mathbb{E}\left[ \phi(|\Xi| - \a_0) \right] \geq \mathbb{E}\left[ \phi\left(\frac{|\Theta|}{\taustar(\a_0)} + \a_0\right) \right]\geq \frac{\mathbb{E}[\Theta^2]}{2M^2}\phi\left( \frac{M}{\taustar(\a_0)}+\a_0\right) .
\end{align}
Also, it is easily seen that 
\begin{align}\label{eq:e2-bound}
|E_2|\leq\max_{x\in\mathbb{R}} \left\lbrace x\phi(x-\a_0) - x\phi(-x-\a_0) \right\rbrace,
\end{align}
where the right-hand side is positive and bounded away from $0$ whenever $\a_0$ is positive. Therefore, the right-hand side of the expression~\eqref{eq:lambda-lower} and the expression~\eqref{eq:lambda-upper} are both bounded away from $0$ and $\infty$ for any fixed $\a_0 \geq \a_{\min}$. Combining these two cases, we have proved the advertised inequality~\eqref{eq:lambda-bound}.

\subsection{Proof of Proposition~\ref{prop:tau-t-limit}}\label{sec:convergence-tau}

From the proof of \citet[Proposition 1.3]{bayati2011lasso}, we know that the function $\tau^2 \mapsto \mathsf{F}(\tau^2, \a\tau)$ is concave for any $\a>0$ and $\Theta$ not equal to $0$. Therefore, we obtain
\begin{align*}
\Big|\frac{\tau_{t+1}^2 - \taustart^2}{\tau_{t}^2 - \taustart^2 } \Big| \leq \frac{\taustart^2-\sigma^2}{\taustart^2-0} \leq 1 - \frac{\sigma^2}{{\tau^\star_{\max}}^2} =: \eta,
\end{align*}
as $\tau_{t+1}^2 = \mathsf{F}(\tau^2, \a\tau)$ and $\sigma^{2} = \mathsf{F}(0, 0).$
Consequently, it leads to 
\begin{align*}
\left| \tau_{t+1}^2 -{\tau^\star_{t+1}}^2\right| \leq \eta \left| \tau_{t}^2 -\taustart^2\right| + \left| \taustart^2 -{\tau^\star_{t+1}}^2\right|.
\end{align*}
Without loss of generality, assume $\left\lbrace \lambda_t\right\rbrace$ decays with $t$. Invoking the above relation recursively, we obtain 
\begin{align*}
	\frac{\left| \tau_{t+1}^2 -{\tau^\star_{t+1}}^2\right|}{\eta^{t+1}} & \leq \frac{\left| \tau_{1}^2 -{\tau^\star_{1}}^2\right|}{\eta} + \sum_{s=1}^t\frac{\left| {\tau^\star_{s}}^2 -{\tau^\star_{s+1}}^2\right|}{\eta^s} \\
	\small{\text{(since $\taustar(\lambda)$ is $L_\tau$-Lipschitz )}} & \leq \frac{\left| \tau_{1}^2 -{\tau^\star_{1}}^2\right|}{\eta} + 2 \tau_{\max}^\star L_{\tau} \sum_{s=1}^{[t/2]}\frac{\left| {\lambda_{s}} -{\lambda_{s+1}}\right|}{\eta^s}+ 2 \tau_{\max}^\star L_{\tau} \sum_{s=1+[t/2]}^{t}\frac{\left| {\lambda_{s}} -{\lambda_{s+1}}\right|}{\eta^s} \\
	& \leq \frac{\left| \tau_{1}^2 -{\tau^\star_{1}}^2\right|}{\eta} +2 \tau_{\max}^\star L_{\tau} \frac{1}{\eta^{[t/2]}}\sum_{s=1}^{[t/2]}  [\lambda_s - \lambda_{s+1}] + 2 \tau_{\max}^\star L_{\tau} \frac{[ \lambda_{[t/2]} - \lambda_{[t/2]+1}]}{1-\eta}\eta^{-t}\\
	& \leq \frac{\left| \tau_{1}^2 -{\tau^\star_{1}}^2\right|}{\eta} +2 \tau_{\max}^\star L_{\tau} \lambda_1  \frac{1}{\eta^{[t/2]}}+ 2 \tau_{\max}^\star L_{\tau} \frac{[ \lambda_{[t/2]} - \lambda_{[t/2]+1}]}{1-\eta}\eta^{-t}.
\end{align*}
Re-arranging the above expression, we are left with
\begin{align*}
\left| \tau_{t}^2 -{\tau^\star_{t}}^2\right| \leq \frac{\left| \tau_{1}^2 -{\tau^\star_{1}}^2\right|}{\eta} \eta^t + 2\tau_{\max}^\star L_\tau  \lambda_1  \eta^{[t/2]}+ \frac{2\tau_{\max}^\star L_\tau}{1-\eta} [ \lambda_{[t/2]} - \lambda_{[t/2]+1}] \goto 0,\quad t\goto \infty.
\end{align*}
This completes the proof of Proposition~\ref{prop:tau-t-limit}.


\section{Proofs about the AMP updates}
\label{sec:proof-AMP}

The goal of this section is to prove Theorem~\ref{thm:interpolation-risk-amp}. 
In Section~\ref{sec:key-lemma}, we state a key lemma (cf.~Lemma~\ref{thm:distance-limit}), which characterizes the convergence of the AMP updates (as $t\goto \infty$) to the minimum $\ell_1$-norm interpolator $\thetaint$. 
The proof of Theorem~\ref{thm:interpolation-risk-amp} is then built upon this lemma. 
Section~\ref{sec:proof-distance-lemma} is then devoted to the main proof of Lemma~\ref{thm:distance-limit} with auxiliary lemmas established in Section~\ref{sec:auxiliary-lemmas} and Section~\ref{sec:converge-var-cov}.

The main structure of the proof is similar to that of \cite{bayati2011lasso}; in the following text, we often refer to the paper as BM for simplification.
The major difference between the min $\ell_{1}$ scenario and a fixed $\lambda$ scenario (considered in \cite{bayati2011lasso} and other references) lies in the lack of restricted strong convexity around the solution point, which prevents us from translating the closeness in the loss function values to the proximity of the minimizers. 
We shall overcome this challenge by carefully investigating the AMP updates for decaying choices of the regularization parameter. 
Throughout this section, we make use of the properties for the state evolution parameters repeatedly (we refer the readers to Section~\ref{sec:proof-state-evolution} for more details). 

\subsection{Proof of Theorem~\ref{thm:interpolation-risk-amp}}
\label{sec:key-lemma}

The key enabler for obtaining Theorem~\ref{thm:interpolation-risk-amp} from Proposition~\ref{prop:amp-fixed-t} is the following result, which connects the minimum $\ell_1$-norm interpolator with the AMP iterations in expression~\eqref{eq:amp-formula-linear} (with the parameters selected according to \eqref{eqn:choice-zeta-t}). 

\begin{lem}
	\label{thm:distance-limit}
	The sequence produced by the AMP updates satisfies
	\begin{equation}\label{eq:distance-limit}
		\lim_{t\goto\infty}\lim_{p\goto\infty} \frac{1}{p}\|\t^t-\thetaint\|_2^2 ~\overset{\mathrm{a.s.}}{=} 0. 
	\end{equation}
\end{lem}

Let us first provide the proof of Theorem~\ref{thm:interpolation-risk-amp}
and defer the proof of Lemma~\ref{thm:distance-limit} to Section~\ref{sec:proof-distance-lemma}. 
To begin with, given any $t\geq 0$, in view of the pseudo-Lipschitz property of $\psi$ (cf.~\eqref{eqn:pseudo-lipschitz}), one has 
	\begin{align*}
	\left|\frac{1}{p}\sum_{i=1}^p \psi(\thetainti, \coefi)-\frac{1}{p}\sum_{i=1}^p \psi(\theta_i^{t+1},\coefi)\right| & \leq \frac{L}{p} \sum_{i=1}^p |\theta_i^{t+1}-\thetainti|\left(1 + \sqrt{(\thetainti)^2 + (\coefi)^2} + \sqrt{(\theta_i^{t+1})^2 + (\coefi)^2} \right)\\
	& \leq  \frac{L}{p} \|\t^{t+1}-\thetaint\|_2 \sqrt{\sum_{i=1}^p \left(1 + |\thetainti| + |\theta_i^{t+1}|+ 2|\coefi|\right)^2} \\
	& \leq L \frac{\|\t^{t+1}-\thetaint\|_2}{\sqrt{p}} \cdot \sqrt{4 + \frac{16\|\coef\|_2^2}{p} + \frac{4\|\t^{t+1}\|_2^2}{p} + \frac{4\|\thetaint\|_2^2}{p}}.
	\end{align*}
Regarding the right-hand side of the above relation, 
Lemma~\ref{thm:distance-limit} guarantees that $\lim_{p\goto\infty} \|\t^{t+1}-\thetaint\|_2/\sqrt{p} \goto 0$ almost surely as $t \goto \infty$; our assumptions about $\coef$ ensure that $\|\coef\|_2^2/p$ is bounded. By virtue of  Lemma~\ref{lem:l2-upper-bound}, the other two terms involving state evolution $\t^t$ and the $\ell_1$-minimization solution $\thetaint$ are also bounded. Putting these together, one can readily conclude that 
	\begin{align*}
	\lim_{p\goto\infty} \frac{1}{p} \sum_{i=1}^p \psi(\thetainti, \coefi) & = \lim_{t\goto\infty}\lim_{p\goto\infty} \frac{1}{p}\sum_{i=1}^p \psi(\theta_i^{t},\coefi) \quad (\text{Lemma~\ref{thm:distance-limit},  Lemma~\ref{lem:l2-upper-bound}})\\
	 & = \lim_{t\goto\infty} \mathbb{E}\left\lbrace \psi(\eta(\Theta  + \tau_tZ; \zeta_t), \Theta)\right\rbrace \quad \text{(Proposition~\ref{prop:amp-fixed-t})}\\
	& = \mathbb{E}\left\lbrace \psi(\eta(\Theta  + \taustar Z; \alphastar\taustar), \Theta)\right\rbrace. 
	\end{align*}
	Here, the last step makes use of Proposition~\ref{prop:tau-t-limit} and Proposition~\ref{prop:amp-fixd-point-convergence} which demonstrates that $\tau_t \goto \taustar$  and  $\alphastart\goto\alphastar$ as $t\goto\infty$. 
	The proof of Theorem~\ref{thm:interpolation-risk-amp} is thus complete.

\subsection{Proof of Lemma~\ref{thm:distance-limit}}
\label{sec:proof-distance-lemma}

This section is devoted to the proof of Lemma~\ref{thm:distance-limit}.
To this end, we first introduce a key result in Lemma~\ref{lem:lasso-structure} which characterizes the conditions under which, the $\ell_2$-norm of the perturbation $\|\bs{r}\|_2$ can be controlled whenever the difference between
$\mathcal{C}_{\lambda_t}(\t + \bs{r})$ and $\mathcal{C}_{\lambda_t}(\t)$ (cf.~\eqref{eq:definition-lasso-loss}) is small. 
The proof of this lemma can be found in Section~\ref{sec:lasso-structure-proof}.
With Lemma~\ref{lem:lasso-structure} in place, proving Lemma~\ref{thm:distance-limit} boils down to verifying each required condition.  
To accomplish this goal, 
we make use of a series of results in Lemmas~\ref{lem:l2-upper-bound}-\ref{lem:condition-4}, followed by the complete proof of Lemma~\ref{thm:distance-limit}.
The proofs of these auxiliary results are deferred to Section~\ref{sec:auxiliary-lemmas}. 

We define the Lasso problem associated with the regularization parameter $\lambda_t$ as follows
\begin{align}
	\label{eq:definition-lasso-loss}
	\widehat{\t}_t = \widehat{\t}(\lambda_t; \bs{X}, \bs{y}) = \argmin_{\t}\mathcal{C}_{\lambda_t}(\t), \quad \text{where} \quad \mathcal{C}_{\lambda_t}(\t) \defn \frac{1}{2}\|\bs{y} - \bs{X}\t\|_2^2 + \lambda_t \|\t\|_1.
\end{align}
As we shall make clear momentarily, each iterate of our AMP updates aims to take a step closer to the minimizer of 
$\mathcal{C}_{\lambda_t}(\t)$. More connections with these Lasso estimators shall be pointed out in the sequel.

\begin{lem}\label{lem:lasso-structure}
	There exists a function $\xi(\varepsilon, c_1, \dots, c_6)$ such that,
	if $\t, \bs{r}\in\mathbb{R}^p$ satisfy the following conditions:
	\begin{enumerate}
		\item $\|\bs{r}\|_2 \leq c_1 \sqrt{p}$; \label{condition-1-lasso}
		
		\item $\mathcal{C}_{\lambda_t}(\t + \bs{r}) \leq \mathcal{C}_{\lambda_t}(\t) + c_2\lambda_t p\varepsilon$;\label{condition-2-lasso}

		\item There exists $\mathrm{sg}(\mathcal{C}_{\lambda_t}, \t)\in \partial \mathcal{C}_{\lambda_t}(\t)$, such that $\|\mathrm{sg}(\mathcal{C}_{\lambda_t}, \t)\|_2 \leq \sqrt{p}\lambda_t \varepsilon$;
		\label{condition-3-lasso}
		
		\item Let $\bs{s} = (1/\lambda_t) [\bs{X}^\top(\bs{y}-\bs{X}\t)+sg(\mathcal{C}_{\lambda_t}, \t)] \in \partial \|\t\|_1$, and $S(c_3)=\left\lbrace i\in [p]: |s_i|\geq 1-c_3\right\rbrace $. Then for any $S'\subseteq [p]$, $|S'|\leq c_4p$, we have $\sigma_{\min}(\bs{X}_{S(c_3)\cup S'})\geq c_5$;
		\label{condition-4-lasso}

		\item $\sigma_{\max}(\bs{X}) \leq c_6$,
		\label{condition-5-lasso}
	\end{enumerate}
	then $\|\bs{r}\|_2 \leq \sqrt{p} \xi(\varepsilon, c_1, \dots, c_6)$, and for any $c_1, \dots, c_6 > 0$, one has $\xi(\varepsilon, c_1, \dots, c_6) \goto 0$ as $\varepsilon \goto 0$. 
\end{lem}

A few remarks are in order. First, we write $\xi(\varepsilon, c_1, \dots, c_6)$ to emphasize that the function does not depend on $\bs{X}$, $\mathcal{C}_{\lambda_t}$ or the constructions of $\t, \bs{r}$.
Secondly, Lemma~\ref{lem:lasso-structure} is a generalization of BM-Lemma 3.1. Since BM is concerned with the Lasso estimator with a single positive $\lambda$, the corresponding lemma only requires $\mathcal{C}_{\lambda}(\t + \bs{r}) \leq \mathcal{C}_{\lambda}(\t)$, and is employed for the Lasso loss function with $\t+\bs{r}$ being the Lasso solution. In this case, $\mathcal{C}_{\lambda}(\t + \bs{r}) \leq \mathcal{C}_{\lambda}(\t)$ holds true for every $\t$ by definition. 
In our setting, however, the minimum $\ell_1$-norm solution $\thetaint$ is not the minimizer of $\mathcal{C}_{\lambda_t}$ --- recognizing the fact that we aim to apply Lemma~\ref{lem:lasso-structure} with $\t +\bs{r}  = \thetaint$. BM-Lemma 3.1 therefore does not apply directly. 
Hence, it requires us to generalize their lemma in a way suitable to our setting.

\paragraph{Proof of Lemma~\ref{thm:distance-limit}.}
Now suppose that one can find constants $(c_1, \dots, c_6)$ and a sequence $\left\lbrace \varepsilon_t\right\rbrace $ satisfying $\lim_{t\goto\infty}\varepsilon_t=0$, such that the five conditions in Lemma~\ref{lem:lasso-structure} hold almost surely with the choice of $\varepsilon = \varepsilon_t$, $\t = \t^t$ and $\t + \bs{r} = \thetaint$. 
As a result, we know that for each $t$, it holds that 
\begin{align*}
\|\t^t-\thetaint\|_2 \leq \sqrt{p}\xi(\varepsilon_t, c_1, \dots, c_6) 
\end{align*}
almost surely. 
Further taking $t\goto \infty$ yields
\begin{align*}
\lim_{t\goto\infty} \frac{\|\t^t-\thetaint\|_2}{\sqrt{p}} \leq \lim_{t\goto\infty} \xi(\varepsilon_t, c_1, \dots, c_6) = 0,
\end{align*}
where the last equality follows from $\lim_{t\goto\infty}\varepsilon_t = 0$ and the conclusion in Lemma~\ref{lem:lasso-structure} that $\xi(\varepsilon, c_1, \dots, c_6)\goto 0$ as $\varepsilon \goto 0$.
Thus we complete the proof of Lemma~\ref{thm:distance-limit}. 
It remains to construct these $(c_1, \dots, c_6)$ 
and a converging sequence of $\varepsilon_t$, such that the five conditions in Lemma~\ref{lem:lasso-structure} hold almost surely as $p\goto\infty$.

To begin with, we make the observation that the last condition of Lemma~\ref{lem:lasso-structure} follows directly from the classical random matrix theory where $ \sigma_{\max}(\bs{X}) \goto 1 + \frac{1}{\sqrt{\delta}}$ almost surely (see, e.g., \cite{bai2010spectral}). 
Thus it suffices to verify the other four conditions.

\begin{itemize}

\item \emph{Condition~\ref{condition-1-lasso} of Lemma~\ref{lem:lasso-structure}.}
We introduce the following lemma, which develops upper bounds on the $\ell_2$-norm of both the $\ell_1$-interpolation solution and the AMP iterations. 
The proof of this result is deferred to Section~\ref{sec:proof-l2-upper-bound}. 
\begin{lem}
	\label{lem:l2-upper-bound}
	There exists a constant $C$ that only depends on $\Theta$, $\sigma$ and $\delta$, 
	such that, almost surely
	\begin{subequations}
		\begin{equation}\label{eq:l2-bound-1}
		\lim_{t\goto\infty} \lim_{p\goto\infty} \frac{1}{p}\|\t^t\|_2^2 < C;\\
		\end{equation}
		\begin{equation}\label{eq:l2-bound-2}
		\lim_{p\goto\infty} \frac{1}{p}\|\thetaint\|_2^2 < C.
		\end{equation}
	\end{subequations}
\end{lem}
By virtue of this lemma, there exists a constant $c_1>0$ such that $\ltwo{\bs{r}} = \ltwo{\thetaint -\t^t} \leq c_1 \sqrt{p}$, which validates 
Condition~\ref{condition-1-lasso} of Lemma~\ref{lem:lasso-structure}.

\item		
\emph{Condition~\ref{condition-3-lasso} of Lemma~\ref{lem:lasso-structure}.}
Similar to the constructions in BM (pg.~25), let us denote 
\begin{align}\label{eq:def-s}
s_i^t \defn 
\left\lbrace \begin{matrix}
\text{sign}(\theta_i^t), & \text{if } \theta_i^t\neq 0;\\
\frac{1}{\zeta_{t-1}} \left\lbrace [\bs{X}^\top\bs{z}^{t-1}]_i + \theta_i^{t-1}\right\rbrace , & \text{otherwise}, 
\end{matrix}\right. 
\quad \text{or}\quad 
\bs{s}^t \defn \frac{1}{\zeta_{t-1}}\left(\t^{t-1}+\bs{X}^\top \bs{z}^{t-1}-\t^t\right).
\end{align}
In view of the AMP updates~\eqref{eq:amp-formula-linear-1}, we can verify the equivalence of these two expressions above, and that vector 
\begin{align}\label{eq:sub-gradient}
\text{sg}(\mathcal{C}_{\lambda_t}, \t^t) \defn \lambda_t\bs{s}^t-\bs{X}^\top(\bs{y}-\bs{X}\t^t)
\end{align}
is a valid sub-gradient of $\mathcal{C}_{\lambda_t}$ at $\t^{t}$. 
With these notation in place, we introduce the following lemma which controls the $\ell_2$-norm of $\mathrm{sg}(\mathcal{C}_{\lambda_t}, \t^t)$.  
The proof of this result can be found in Section~\ref{sec:proof-sg-norm}. 

\begin{lem}\label{lem:sg-norm}
Under our choices of $\left\lbrace \lambda_t\right\rbrace $ and $\left\lbrace \zeta_t\right\rbrace $, the sub-gradient of $\mathcal{C}_{\lambda_t}$ at point $\t^t$ defined in \eqref{eq:sub-gradient} satisfy
	\begin{equation}
		\lim_{t\goto\infty} \lim_{p\goto\infty} \frac{1}{\sqrt{p}\lambda_t}\|\mathrm{sg}(\mathcal{C}_{\lambda_t}, \t^t)\|_2 ~\overset{\mathrm{a.s.}}{=}~ 0.
	\end{equation}
\end{lem}
In other words, there exists a sequence of $\{\varepsilon_{t}\}$ approaching zero such that: for each $t$, it holds almost surely that
$\|\mathrm{sg}(\mathcal{C}_{\lambda_t}, \t^t)\|_2 \leq \sqrt{p}\lambda_t \varepsilon_{t}$.

\item \emph{Condition~\ref{condition-4-lasso} in Lemma~\ref{lem:lasso-structure}.}
As defined above, the vector $\bs{s}^t$ in the expression~\eqref{eq:def-s} 
satisfies
\begin{align*}
	\bs{s}^t =\lambda_t^{-1}
\left[\bs{X}^\top(\bs{y}-\bs{X}\t^t)+\mathrm{sg}(\mathcal{C}_{\lambda_t}, \t^t)\right] \in \partial \|\t^t\|_1,
\end{align*}
where $\mathrm{sg}(\mathcal{C}_{\lambda_t}, \t^t) \in \partial \mathcal{C}_{\lambda_t}(\t^t)$ (cf.~\eqref{eq:sub-gradient}) is a valid sub-gradient of $\mathcal{C}_{\lambda_t}$. 
It turns out that the approximate support set of $\t^{t}$ does not vary too much across iterations of the AMP algorithm. 
This is a high-level reason why $\sigma_{\min}(\bs{X}_{S(c_3)\cup S'})$ can be bounded away from zero, despite the fact that the loss function is not strongly convex.   
This observation is rigorized in the lemma below.

\begin{lem}\label{lem:S-convergence}
Given every $\gamma\in (0, 1)$and $t\geq 1$, define the set
	\begin{align}
	S_t(\gamma) = \left\lbrace i\in [p]: |s_i^t|\geq 1-\gamma\right\rbrace ,
	\end{align}
	with $\bs{s}^t$ defined in \eqref{eq:def-s}. 
	Then for any $\xi > 0$ there exists $t_{\star} = t_{\star}(\xi, \gamma) < \infty$ 
	such that, for all $t_2\geq t_1 \geq t_{\star}$, one has 
	\begin{align*}
	|S_{t_2}(\gamma)/S_{t_1}(\gamma)| < p\xi,
	\end{align*}
	almost surely as $p\goto \infty$.
\end{lem}
The proof of Lemma~\ref{lem:S-convergence} is provided in Section~\ref{sec:converge-var-cov}. 
Recall that BM-Lemma C.5 establishes similar results for the AMP iterates corresponding to a fixed $\lambda.$ 
Here, we aim to approximate the Lasso solution for different parameter $\lambda_{t}$ at each step of the AMP iteration. Therefore it requires us to consider auxiliary state-evolution formulas for each $t$ separately in order to establish Lemma~\ref{lem:S-convergence}. 

Based on Lemma~\ref{lem:S-convergence}, one can derive the following result concerning the constrained singular value of $\bs{X}$. 
\begin{lem}\label{lem:condition-4}
	There exists constraints $\gamma_1 \in (0, 1)$, $\gamma_2, \gamma_3 > 0$ and $t_{\min} < \infty$ such that for any $t\geq t_{\min}$,
	\begin{equation}
	\min_{S'}\left\lbrace \sigma_{\min}(\bs{X}_{S_t(\gamma_1)\cup S'}): S'\subseteq [p], |S'|\leq \gamma_2 p \right\rbrace \geq \gamma_3,
	\end{equation}
	almost surely as $p\goto\infty$.
\end{lem} 

Based on the conclusion of Lemma~\ref{lem:l2-upper-bound}, the proof of Lemma~\ref{lem:condition-4} follows verbatim as of BM-Proposition 3.6, and is thus omitted here. Note that apart from Lemma~\ref{lem:l2-upper-bound}, the proof requires BM-Lemma 3.4, which holds for AMP iterations with general choices of $\left\lbrace \zeta_t\right\rbrace $, which can be directly adapted to accommodate our setting.

As a direct consequence of Lemma~\ref{lem:condition-4}, Condition~\ref{condition-4-lasso} in Lemma~\ref{lem:lasso-structure} follows immediately with the choice of $\bs{s}=\bs{s}^t$.

\item \emph{Condition~\ref{condition-2-lasso} in Lemma~\ref{lem:lasso-structure}.}
To verify this condition, we only need to find a vanishing sequence $\{\varepsilon_{t}\}$ such that 
$\mathcal{C}_{\lambda_t}(\t^t) - \mathcal{C}_{\lambda_t}(\thetaint) \geq - c_2\lambda_t p\varepsilon_t$. 
Towards this end, we shall control $\mathcal{C}_{\lambda_t}(\t^t) - \mathcal{C}_{\lambda_t}(\thetaint)$ as follows. 
First, recalling that $\widehat{\t}_t$ is the minimizer of $\mathcal{C}_{\lambda_t}(\cdot)$ (cf.~\eqref{eq:definition-lasso-loss}). Hence, 
\begin{align}
\label{eq:second-condition}
\notag 
\frac{ \mathcal{C}_{\lambda_t}(\t^t) - \mathcal{C}_{\lambda_t}(\thetaint)}{p}  \geq\frac{ \mathcal{C}_{\lambda_t}(\widehat{\t}_t) - \mathcal{C}_{\lambda_t}(\thetaint) }{p} &= \frac{1}{2p}\|\bs{y}-\bs{X}\widehat{\t}_t\|_2^2 + \frac{\lambda_t }{p} \left[ \|\widehat{\t}_t\|_1 - \|\thetaint\|_1\right] \\
&\geq \frac{\lambda_t}{p} \left[ \|\widehat{\t}_t\|_1 - \|\thetaint\|_1\right].
\end{align}
In the following, we aim to show that $ \|\widehat{\t}_t\|_1 - \|\thetaint\|_1 \geq - c_2 \varepsilon_{t}$ for some vanishing sequence $\{\varepsilon_t\}.$
From the definition of $\mathcal{C}_{\lambda_t}$ in the expression~\eqref{eq:definition-lasso-loss}, for every $\t$ we can express  
\begin{align*}
\mathcal{C}_{\lambda_t}(\t) = \mathcal{C}_{\lambda_{t+1}}(\t) + (\lambda_{t}-\lambda_{t+1})\|\t\|_1;
\end{align*}
both of the right-hand side terms $\mathcal{C}_{\lambda_{t+1}}(\t)$ and $(\lambda_{t}-\lambda_{t+1})\|\t\|_1$ are convex and non-negative functions of $\t$, and are minimized at $\widehat{\t}_{t+1}$ and $\bs{0}_p$, respectively. 
For any $\t \in\mathbb{R}^p$ with $\|\t\|_1 > \|\widehat{\t}_{t+1}\|_1 \geq 0$, it is easily seen that $\mathcal{C}_{\lambda_{t+1}}(\t)\geq\mathcal{C}_{\lambda_{t+1}}(\widehat{\t}_{t+1})$ and $
(\lambda_{t}-\lambda_{t+1})\|\t\|_1 > (\lambda_{t}-\lambda_{t+1})\|\widehat{\t}_{t+1}\|_1 $; it follows that $\mathcal{C}_{\lambda_t}(\t)> \mathcal{C}_{\lambda_t}(\widehat{\t}_{t+1}) \geq  \mathcal{C}_{\lambda_t}(\widehat{\t}_t)$. Thus we reach the conclusion that $\|\widehat{\t}_t\|_1 \leq \|\widehat{\t}_{t+1}\|_1$.
Additionally, by virtue of Lemma~\ref{lem:l2-upper-bound}, we obtain that
\begin{align*}
\lim_{p\goto\infty}\frac{\|\thetaint\|_1}{p} \leq \lim_{p\goto\infty}\frac{\|\thetaint\|_2}{\sqrt{p}} \leq \sqrt{C}
\end{align*}
is bounded almost surely. As a summery, the sequence $\|\widehat{\t}_t\|_1/p$ enjoys the following two properties 
\begin{itemize}
\item $\|\widehat{\t}_t\|_1/p \leq \|\widehat{\t}_{t+1}\|_1 /p\leq \dots \leq \|\thetaint\|_1/p$, and almost surely as $p \goto \infty$, $ \lim_{t\goto \infty} \widehat{\t}_t = \thetaint$;

\item $\|\thetaint\|_1/p$ is bounded almost surely.
\end{itemize}
In view of the monotone convergence theorem, we reach the conclusion that
\begin{align*}
\lim_{t\goto\infty} \frac{\|\thetaint\|_1 - \|\widehat{\t}_t\|_1}{p} = 0 \quad \Longrightarrow\quad \frac{\|\thetaint\|_1 - \|\widehat{\t}_t\|_1}{p}\leq \varepsilon_t,
\end{align*}
for a vanishing sequence $\left\lbrace \varepsilon_t\right\rbrace $. Combining with expression~\eqref{eq:second-condition}, we arrive at 
$
\mathcal{C}_{\lambda_t}(\t^t) - \mathcal{C}_{\lambda_t}(\thetaint)   \geq -c_2 \lambda_t p \varepsilon_t.
$
Thus we verify the condition~\ref{condition-2-lasso} in Lemma~\ref{lem:lasso-structure}.

\end{itemize}

Taking the above results collectively, we complete the proof of Lemma~\ref{thm:distance-limit}.

\subsection{Proof of supporting lemmas to Lemma~\ref{thm:distance-limit}
}
\label{sec:auxiliary-lemmas}

\subsubsection{Proof of Lemma~\ref{lem:lasso-structure}}\label{sec:lasso-structure-proof}

As experienced readers might have already noticed, the statement of Lemma~\ref{lem:lasso-structure} is very similar to that of BM-Lemma 3.1;  
the only difference lies in the Condition~\ref{condition-2-lasso} and Condition~\ref{condition-3-lasso}. 
A closer inspection at the proof of BM-Lemma 3.1 reveals that these two conditions are only used to establish BM-(3.4) and BM-(3.5).
Therefore, as long as we can show BM-(3.4) and BM-(3.5) for our setting, the proof of our lemma will be completed, with the rest part following verbatim from the proof of BM-Lemma 3.1.

Let us first adapt BM-(3.4) and BM-(3.5) to our setting.
Throughout, we shall use $\xi(\varepsilon)$ to denote a function of constants $c_1, \dots, c_6>0$ and of $\varepsilon$ such that $\xi(\varepsilon) \goto 0$ as $\varepsilon \goto 0$. Additionally, we shall use $S = \text{supp}(\t)\subseteq [p]$. 
Formally, it suffices for us to show that, one can find such $\xi(\varepsilon)$ where 
\begin{subequations}
	\begin{equation}\label{eq:convex-1}
	\frac{\|\bs{r}_{\bar{S}}\|_1- \inprod{\bs{s}_{\bar{S}}}{\bs{r}_{\bar{S}}}}{p}  \leq \xi(\varepsilon);
	\end{equation}
	\begin{equation}\label{eq:convex-2}
	\|\bs{X}\bs{r}\|_2^2 \leq p \xi(\varepsilon). 
	\end{equation}
\end{subequations}
Given these two inequalities, the proof of Lemma~\ref{lem:lasso-structure} follows directly from BM-Lemma 3.1. 
For the sake of brevity, we only establish the  aforementioned two inequalities, and refer readers to \cite{bayati2011lasso} for the rest of the proof.

\paragraph{Verifying the expressions~\eqref{eq:convex-1} and \eqref{eq:convex-2}.}
With $\bs{s}$ defined in condition~\ref{condition-4-lasso}, we obtain
\begin{align*}
c_2 \varepsilon \geq & \frac{\mathcal{C}_{\lambda_t}(\t+\bs{r}) - \mathcal{C}_{\lambda_t}(\t)}{p\lambda_t} \quad \text{(Condition 2)}\\
= & \frac{\|\t_S+\bs{r}_S\|_1-\|\t_S\|_1}{p} + \frac{\|\bs{r}_{\bar{S}}\|_1}{p}+\frac{\|\bs{y}-\bs{X}\t - \bs{Xr}\|_2^2 -\|\bs{y}-\bs{X}\t\|_2^2 }{2p\lambda_t}\\
\overset{(\mathrm{i})}{=} & \frac{\|\t_S+\bs{r}_S\|_1-\|\t_S\|_1- \inprod{\text{sign}(\t_S)}{\bs{r}_S}}{p}  + \frac{\|\bs{r}_{\bar{S}}\|-\inprod{\bs{s}_{\bar{S}}}{\bs{r}_{\bar{S}}}}{p} + \frac{\lambda_t \inprod{\bs{s}}{\bs{r}} - \inprod{\bs{y}-\bs{X}\t}{\bs{Xr}} + \frac{1}{2}\|\bs{Xr}\|_2^2}{p\lambda_t} \\
\overset{(\mathrm{ii})}{=}  &  \frac{\|\t_S+\bs{r}_S\|_1-\|\t_S\|_1- \inprod{\text{sign}(\t_S)}{\bs{r}_S}}{p}  + \frac{\|\bs{r}_{\bar{S}}\|-\inprod{\bs{s}_{\bar{S}}}{\bs{r}_{\bar{S}}}}{p} +  \frac{ \inprod{\text{sg}(\mathcal{C}_{\lambda_t}, \t)}{\bs{r}} }{p\lambda_t} + \frac{\|\bs{Xr}\|_2^2}{2p\lambda_t} . \quad \text{(Condition 4)}.
\end{align*}
Here, $(\mathrm{i})$ follows from the fact that $\bs{s}\in \partial \|\t\|_1$, and thus $\text{sign}(\t_S)=\bs{s}_S$.
The equality $(\mathrm{ii})$ comes from the definition~\eqref{eq:sub-gradient}. 

Invoking the Cauchy-Schwarz inequality and Condition~\ref{condition-3-lasso} yields $\left|\inprod{\text{sg}(\mathcal{C}_{\lambda_t}, \t)}{\bs{r}}/(p\lambda_t )\right| \leq c_1 \varepsilon$.
Substitution into the above inequality with a little algebra leads to
\begin{equation}
\frac{\|\t_S+\bs{r}_S\|_1-\|\t_S\|_1- \inprod{\text{sign}(\t_S)}{\bs{r}_S}}{p}  + \frac{\|\bs{r}_{\bar{S}}\|-\inprod{\bs{s}_{\bar{S}}}{\bs{r}_{\bar{S}}}}{p} + \frac{\|\bs{Xr}\|_2^2}{2p\lambda_t} \leq (c_1+c_2) \varepsilon.
\end{equation}
It can be easily seen that these three terms on the left-hand side above are all non-negative. Therefore, the inequalities~\eqref{eq:convex-1} and \eqref{eq:convex-2} follow directly.

\subsubsection{Proof of Lemma~\ref{lem:l2-upper-bound}}
\label{sec:proof-l2-upper-bound}

The proof of this lemma is adapted from BM-Lemma 3.2 with modifications tailored to the minimum $\ell_{1}$-norm solution. 
As mentioned previously, Proposition~\ref{prop:amp-fixed-t} holds for general choices of $\{\zeta_t\}$, and is therefore directly applicable to our setting. Hence, the AMP iterates satisfy 
\begin{align*}
\lim_{t\goto\infty}\lim_{p\goto\infty}\frac{1}{p}\|\t^t\|_2^2 = \lim_{t\goto\infty} \mathbb{E}\left[ \eta^2(\Theta + \tau_tZ; \zeta_t)\right] \leq \mathbb{E}[\Theta^2] + {\taustar}^2.
\end{align*}
In other words, asymptotically $\frac{1}{p}\|\t^t\|_2^2$ is bounded by some constant that depends only  on $\mathbb{E}[\Theta^2]$ and $\taustar$. 

Consequently, to prove Lemma~\ref{lem:l2-upper-bound}, it suffices to establish the relation~\eqref{eq:l2-bound-2}.
Let us first decompose the minimum $\ell_1$-norm interpolator $\thetaint$ into the projection onto the row space of $\bs{X}$, and the residual. Formally, we write
\begin{align*}
	\thetaint = \thetaint_{\perp} + \thetaint_{||},
\end{align*}
where $\thetaint_{||}=\bs{X}^\top(\bs{XX}^\top)^{-1}\bs{X}\thetaint$. 
It is straightforward to verify that $\bs{X}\thetaint_{||}=\bs{X}\thetaint = \bs{y}$. We can view $\thetaint_{\perp}$ as the projection of $\thetaint$ onto the orthogonal space of $\text{row}(\bs{X})$, which is a uniformly random $(p-n)$-dimensional subspace of $\mathbb{R}^p$. By Kashin Theorem (BM-Lemma F.1), there exists a universal constant $c_1>0$ depending on $\delta$ such that
$\|\thetaint_{\perp}\|_2^2 \leq c_1 \|\thetaint_{\perp}\|_1^2/p$, almost surely as $p\goto\infty$. In addition, regarding the limiting value for the eigenvalues of Wishart matrices, it is known that there exists a constant $c_2$ depending only on $\delta$ such that $\|\thetaint_{||}\|_2^2 \leq c_2 \|\bs{X}\thetaint_{||}\|_2^2$ almost surely as $p\goto\infty$ (see BM-Lemma F.2). Therefore, we arrive at
\begin{align}
\notag \lim_{p\goto\infty} \frac{1}{p}\|\thetaint\|_2^2 
& = \lim_{p\goto\infty}\left[  \frac{1}{p} \|\thetaint_{\perp}\|_2^2 +  \frac{1}{p}\|\thetaint_{||}\|_2^2\right] \\
\notag & \leq \lim_{p\goto\infty} 
\left[ c_1 \frac{\|\thetaint_{\perp}\|_1^2}{p^2} +  \frac{\|\thetaint_{||}\|_2^2}{p}\right] \\
\notag (\text{by Cauchy-Schwarz}) &  
\leq \lim_{p\goto\infty} \left[ 2c_1 \left(\frac{\|\thetaint\|_1^2}{p^2} 
+ \frac{\|\thetaint_{||}\|_2^2}{p}\right) +  \frac{\|\thetaint_{||}\|_2^2}{p}\right]\\
& \leq  \lim_{p\goto\infty} \left[ 2c_1 \frac{\|\thetaint\|_1^2}{p^2} + (2c_1+1) c_2 \frac{\|\bs{y}\|_2^2}{p}\right],
\label{eqn:l2-norm-intepolator}
\end{align}
where the last step uses $\bs{X}\thetaint_{||}=\bs{X}\thetaint = \bs{y}$. 

As a consequence, it suffices to upper bound the quantities $\|\thetaint\|_1/p$ and $\|\bs{y}\|_2^2/p$.  
First, by our model assumption, $y_i\overset{\mathrm{i.i.d.}}{\sim} \NORMAL(0, \|\coef\|_2^2/n + \sigma^2)$; 
it immediately follows that $\lim_{p\goto\infty}\|\bs{y}\|_2^2/p$ is bounded. 
In addition, we note that $\bs{a} = \bs{X}^\top(\bs{XX}^\top)^{-1}\bs{y}$ satisfies the condition $\bs{Xa}=\bs{y}$.
Since $\thetaint$ has the minimum $\ell_1$-norm over all linear interpolators, we can guarantee that 
\begin{align*}
	\|\thetaint\|_1 \leq \|\bs{X}^\top(\bs{XX}^\top)^{-1}\bs{y}\|_1.
\end{align*}
Finally, by virtue of BM-Lemma F.2, we obtain $\sigma_{\max}((\bs{XX}^\top)^{-1}) \leq c_3$ for some $c_3>0$ depending on $\delta$, almost surely as $p\goto\infty$. Collecting these components together yields
\begin{align*}
\left(\frac{\|\thetaint\|_1}{p}\right)^2 \leq \frac{1}{p}\|\bs{X}^\top(\bs{XX}^\top)^{-1}\bs{y}\|_2^2 \leq \frac{\bs{y}^\top (\bs{XX}^\top)^{-1}\bs{y}}{p} \leq 2 \frac{\|\coef\|_2^2}{p} + 2 c_3 \frac{\|\noise\|_2^2}{p}.
\end{align*}
Therefore, $\lim_{p \to \infty} (\frac{\|\thetaint\|_1}{p})^2 \leq 2 \mathbb{E}[\Theta^2] + 2c_3\sigma^2$ holds almost surely. 
Combining this with the expression~\eqref{eqn:l2-norm-intepolator}, we see that $\lim_{p\goto\infty} \frac{1}{p}\|\thetaint\|_2^2 $ is upper bounded by a universal constant that depends only  on $\delta$.

Putting these two parts together finishes the proof of Lemma~\ref{lem:l2-upper-bound}.

\subsubsection{Proof of Lemma~\ref{lem:sg-norm}}
\label{sec:proof-sg-norm}

To prove Lemma~\ref{lem:sg-norm}, we aim to upper bound the target quantity $\|\mathrm{sg}(\mathcal{C}_{\lambda_t}, \t^t)\|_2$ by three terms, and look at each term separately. 
For ease of exposition, let us define
\begin{align*}
\omega_t = \frac{1}{\delta} \left\langle\eta'(\bs{X}^\top \bs{z}^{t-1} + \t^{t-1}; \zeta_{t-1})\right\rangle .
\end{align*}
The AMP iterate~\eqref{eq:amp-formula-linear-2} then translates into $\bs{y} - \bs{X}\t^t = \bs{z}^t - \omega_t \bs{z}^{t-1}$. 
With this piece of notation in mind, we can rewrite the sub-gradient (cf.~\eqref{eq:sub-gradient}) as
\begin{align*}
\text{sg}(\mathcal{C}_{\lambda_t}, \t^t) & = \lambda_t \bs{s}^t - \bs{X}^\top (\bs{z}^t-\bs{z}^{t-1}) - (1-\omega_t)\bs{X}^\top \bs{z}^{t-1} \\
& = \frac{\lambda_t}{\zeta_{t-1}}\left[ \zeta_{t-1}\bs{s}^t -  \bs{X}^\top\bs{z}^{t-1}	\right] - \bs{X}^\top (\bs{z}^t-\bs{z}^{t-1}) + \frac{\left[ \lambda_t - \zeta_{t-1}(1-\omega_t) \right]}{\zeta_{t-1}}\bs{X}^\top \bs{z}^{t-1}\\
& = \frac{\lambda_t}{\zeta_{t-1}}(\t^{t-1}-\t^t) - \bs{X}^\top (\bs{z}^t-\bs{z}^{t-1}) + \frac{\left[ \lambda_t - \zeta_{t-1}(1-\omega_t) \right]}{\zeta_{t-1}}\bs{X}^\top \bs{z}^{t-1}.
\end{align*}
Applying the triangle inequality leads to 
\begin{align}\label{eq:sg-decomposition}
\frac{1}{\sqrt{p}\lambda_t} \|\text{sg}(\mathcal{C}_{\lambda_t}, \t^t)\|_2 \leq 
\frac{1}{\zeta_{t-1}}\frac{\|\t^{t-1}-\t^t\|_2}{\sqrt{p}} 
+ \sigma_{\max}({\bs{X}}) \frac{\|\bs{z}^t-\bs{z}^{t-1}\|_2}{\sqrt{p}\lambda_t} 
+\sigma_{\max}({\bs{X}}) \frac{\|\bs{z}^{t-1}\|_2}{\sqrt{p}} \frac{\left[ \lambda_t - \zeta_{t-1}(1-\omega_t) \right]}{\zeta_{t-1}\lambda_t }.
\end{align}
The proof of Lemma~\ref{lem:sg-norm} then boils down to analyzing the three terms on the right-hand side of \eqref{eq:sg-decomposition}.
We first invoke the following lemma describing the convergence of the AMP updates, which will be proved in Section~\ref{sec:convergence-amp-estimate}. 
\begin{lem}
\label{lem:convergence-x-z}
	The AMP iterates obey the following convergence guarantees  
	\begin{align*}
	\lim_{t\goto\infty}\lim_{p\goto\infty} \frac{1}{p\lambda_t^2}\|\t^t - \t^{t-1}\|_2^2 = 0, 
	\qquad \lim_{t\goto\infty}\lim_{p\goto\infty} \frac{1}{p\lambda_t^2}\|\bs{z}^t - \bs{z}^{t-1}\|_2^2 = 0.
	\end{align*}
\end{lem}
Note that, in the limit, Proposition~\ref{prop:amp-fixd-point-convergence} ensures that $\zeta_t \goto \zetastar$, and is thus  bounded away from $0$.
Combining Lemma~\ref{lem:convergence-x-z} and the choice of $\lambda_{t}$ ensures $\frac{1}{\zeta_{t-1}}\frac{\|\t^{t-1}-\t^t\|_2}{\sqrt{p}}$ converges to zero as $t\goto \infty$. 
In addition, since $\sigma_{\max}(\bs{X})$ is almost surely bounded in the limit, the second term on the right-hand side of the expression~\eqref{eq:sg-decomposition} also converges to $0$.

It remains to consider the third term on the right-hand side of the expression~\eqref{eq:sg-decomposition}. 
To this end, one can first characterize the large-system limit of $\|\bs{z}^t\|_2$ as follows
\begin{align}
\label{eqn:tmp-zt}
\lim_{t\goto\infty} \lim_{p\goto\infty} \frac{\|\bs{z}^{t}\|_2}{\sqrt{p}} = \lim_{t\goto\infty} \tau_t = \taustar, \quad \text{almost surely},
\end{align}
as shown in BM-Lemma 4.1.
In addition, the result in BM-Lemma F.3 demonstrates that $\forall t\geq 1$,
$$\lim_{p\goto \infty} \left[1-\omega_t\right] \overset{\text{a.s.}}{=} 1 - \frac{1}{\delta}\mathbb{E}\left[ \eta'(\Theta + \tau_t Z; \zeta_t)\right].$$
By construction, we know $\lambda_t = \taustart \alphastart \left( 1 - \frac{1}{\delta}\mathbb{E}\left[ \eta'(\Theta + \taustart Z;  \taustart \alphastart)\right] \right)$. Hence,
\begin{align*}
\left|\frac{1}{\alphastart\taustart}-\frac{1-\omega_t}{\lambda_t}\right|   &\overset{\mathrm{a.s.}}{\goto}  \frac{1}{\lambda_t} \left|\frac{\lambda_t}{\alphastart\taustart}-\left(1-\frac{1}{\delta}\mathbb{E}\left[ \eta'(\Theta + \tau_t Z; \tau_t \alphastart)\right]\right)\right|  \\
 & = \frac{1}{\delta\lambda_t} \left|\mathbb{E}\left[ \eta'(\Theta + \taustart Z; \alphastart\taustart)\right] -  \mathbb{E}\left[ \eta'(\Theta + \tau_t Z; \tau_t \alphastart)\right]\right| \\
 & \leq \frac{1}{\delta\lambda_t} \mathbb{E}\left[ \left| \Phi\left( \alphastart - \frac{\Theta}{\taustart}\right)  -\Phi\left( \alphastart - \frac{\Theta}{\tau_t}\right) \right|\right] + \frac{1}{\delta\lambda_t} \mathbb{E}\left[ \left| \Phi\left( \alphastart + \frac{\Theta}{\taustart}\right)  -\Phi\left( \alphastart + \frac{\Theta}{\tau_t}\right) \right|\right] \\
 & \leq \frac{2\mathbb{E}[|\Theta|]}{\delta\taustart\tau_t}\frac{|\tau_t-\taustart|}{\lambda_t},
\end{align*}
where the second inequality comes from the Lipschitz property of $\Phi(\cdot)$. 
If we take $t\goto\infty$, then the limiting values of $\taustart\goto\taustar$ and $\tau_t\goto\taustar$ and Lemma~\ref{lem:converge-cov} immediately indicate that
\begin{align*}
\frac{2\mathbb{E}[|\Theta|]}{\delta\taustart\tau_t} \overset{\mathrm{a.s.}}{\goto} \frac{2\mathbb{E}[|\Theta|]}{\delta\taustar^2};\quad \frac{\left|\tau_t-\taustart\right|}{\lambda_t} \overset{\mathrm{a.s.}}{\goto}  0.
\end{align*}
Combining this with the model construction that $\mathbb{E}[|\Theta|] < +\infty$, we obtain $$(1-\omega_t)/\lambda_t \overset{\mathrm{a.s.}}{\goto}  1/(\alphastar\taustar),$$ and it follows that
\begin{align*}
\lim_{t\goto\infty} \lim_{p\goto\infty} \frac{\left[ \lambda_t - \zeta_{t-1}(1-\omega_t) \right]}{\zeta_{t-1}\lambda_t } \overset{\mathrm{a.s.}}{=} \lim_{t\goto\infty} \left[ \frac{1}{\zeta_{t-1}} - \frac{1}{\alphastar\taustar}\right] =0.
\end{align*}
Taken collectively with the expression~\eqref{eqn:tmp-zt}, the third term of the
inequality~\eqref{eq:sg-decomposition} converges to zero. 
This concludes the proof.


\subsection{Proof of Lemma~\ref{lem:S-convergence} and Lemma~\ref{lem:convergence-x-z}}
\label{sec:converge-var-cov}

The goal of this section is to prove Lemma~\ref{lem:S-convergence} and Lemma~\ref{lem:convergence-x-z}, 
which requires us to characterize the changes between two AMP iterates $\t^s$ and $\t^t$ in different iterations $s\neq t$. 
Towards this end, in Section~\ref{sec:definition-var-cov}, we define the covariance between AMP iterates, and state two auxiliary lemmas about the convergence rate of these covariances (cf.~Lemma~\ref{lem:converge-cov} and Lemma~\ref{lem:converge-r-general}). In Section~\ref{sec:proof-S-convergence} and Section~\ref{sec:convergence-amp-estimate}, we invoke these two lemmas to derive Lemma~\ref{lem:S-convergence} and Lemma~\ref{lem:convergence-x-z} respectively, which is then followed by the proofs of these two lemmas.

\subsubsection{Auxiliary definitions and lemmas}
\label{sec:definition-var-cov}

By virtue of Proposition~\ref{prop:amp-fixed-t}, the state evolution sequence $\lbrace \tau_t^2\rbrace_{t=0}^\infty $ can be viewed as the large $n$ ``variance'' of the AMP recursions. 
We generalize this notion to consider the correlations between $\t^s$ and $\t^t$ when $s \neq t$. 
Formally, define a sequence of scalars $\left\lbrace \mathsf{R}_{s, t}\right\rbrace_{s, t\geq 0}$ recursively as in BM-(4.13):
\begin{align}
	\label{eq:def-r}
	\mathsf{R}_{s+1, t+1} \defn \sigma^2 +\frac{1}{\delta}\mathbb{E}\left\lbrace \left[ \eta(\Theta+Z_s; \zeta_s)-\Theta\right] \left[ \eta(\Theta+Z_t; \zeta_t)-\Theta\right]\right\rbrace. 
\end{align}
Here $(Z_s, Z_t) \in \real^{2}$ are jointly Gaussian, independent of $\Theta$, with zero mean and
\begin{align}
\label{eqn:Z-cov-structure}
	\mathbb{E}[Z_s^2] = \mathsf{R}_{s, s}, 
	\quad \mathbb{E}[Z_t^2] =\mathsf{R}_{t, t}, 
	\quad \mathbb{E}[Z_sZ_t] = \mathsf{R}_{s, t}. 
\end{align}
The boundary conditions are given by $\mathsf{R}_{0, 0} = \sigma^2 + \mathbb{E}[\Theta^2]/\delta$, and 
\begin{align}
\label{eq:def-r-boundary}
\mathsf{R}_{0, t+1} = \sigma^2 + \frac{1}{\delta} \mathbb{E}\left\lbrace \left[\eta(\Theta + Z_t; \zeta_t) - \Theta\right](-\Theta) \right\rbrace ,
\end{align}
where $Z_t\sim \NORMAL(0, \mathsf{R}_{t, t})$ is independent of $\Theta$. 
Comparing these with the definition of state evolution formula in expressions~\eqref{eqn:state-evolution-1} and \eqref{eqn:state-evolution-2},
we can immediately see that $\mathsf{R}_{t, t} =\tau_t^2$ for $t\geq 0$.
We record the following result from \cite{bayati2011lasso}. 

\begin{lem}[Theorem 4.2 in \cite{bayati2011lasso}]\label{lem:amp-fixed-t-2-step}
	Under the setting of Proposition~\ref{prop:amp-fixed-t}, given any pseudo-Lipschitz function $\psi: \real^3 \to \real$, it holds that 
	\begin{align*}
		\lim_{p\goto\infty}\frac{1}{p}\sum_{i=1}^p \psi(\theta_i^s + (\bs{X}^\top \bs{z}^s)_i, \theta_i^t + (\bs{X}^\top \bs{z}^t)_i,\theta^\star_i) \overset{\mathrm{a.s.}}{=} \mathbb{E}\left\lbrace \psi(\Theta + Z_s, \Theta + Z_t,\theta^\star_i)
	\right\rbrace ,
	\end{align*}
	where $(Z_s, Z_t)$ are jointly Gaussian, independent from $\Theta$, mean-zero and satisfy~\eqref{eqn:Z-cov-structure}.
\end{lem}

Before embarking on the proofs of Lemma~\ref{lem:S-convergence} and Lemma~\ref{lem:convergence-x-z}, 
let us first introduce two crucial lemmas concerning the convergence rate of $\mathsf{R}_{s, t}$ for both the cases of $s=t$ and $s\neq t$ (note that the $s=t$ case represents the convergence rate of $\tau_t^2$). Their proofs can be found in Section~\ref{sec:proof-convergence-cov} and Section~\ref{sec:proof-converge-r-general}, respectively.
\begin{lem}[Convergence rate of the variance and $1$-step covariance]
	\label{lem:converge-cov}
	Define $\mathsf{R}_{s, t}$ as in then expression~\eqref{eq:def-r}.
	We have 
	\begin{align}\label{eq:r-rate-1}
	\max\left\lbrace |\mathsf{R}_{t, t} - \taustart^2|, |\mathsf{R}_{t+1, t+1} - \taustart^2|, \left|\mathsf{R}_{t, t} - 2\mathsf{R}_{t, t+1} + \mathsf{R}_{t+1, t+1}\right|  \right\rbrace \leq c_0 \exp\left\lbrace - c\Lambda_t\right\rbrace + 4L_\tau\tau^\star_{\max} l_t,
	\end{align}
	where $l_t \defn \sum_{s=1}^t |\lambda_s-\lambda_{s+1}| \exp\left\lbrace - c [\Lambda_t - \Lambda_s] \right\rbrace$ as in Assumption~\ref{assump:lambda}.  
	Here $\taustart$ is the solution to the system of equations \eqref{eqn:fix-eqn}, $c = (2\alpha^\star_{\max}\tau^\star_{\max})^{-1}$, $L_\tau$ is defined in Proposition~\ref{prop:amp-fixd-point-convergence}, and $c_0>0$ is some constant that depends on $\Theta$, $\sigma$ and $\delta$. 
\end{lem}

With Lemma~\ref{lem:converge-cov} in place, for general $t_1$ and $t_2$, one can easily decompose $\mathsf{R}_{t_1, t_2}$ into multiple consecutive differences and obtain the respective convergence rates as follows. 
\begin{lem}[Convergence rate of general covariances]
	\label{lem:converge-r-general}
	For any $t_0\geq 0$ and $t_1, t_2 \geq t_0$, with the same $c_0$ and $c$ defined in Lemma~\ref{lem:converge-cov}, the covariance $\mathsf{R}_{t_1, t_2}$ satisfies 
	\begin{align*}
	\left|\mathsf{R}_{t_1, t_1}-2\mathsf{R}_{t_1, t_2}+\mathsf{R}_{t_2, t_2}\right|  \leq 4\left[\sum_{i=t_1}^{\infty} \sqrt{c_0 \exp\left\lbrace - c\Lambda_{i}\right\rbrace + 4L_\tau\tau^\star_{\max} l_{i}} \right]^2.
	\end{align*}
\end{lem}
Proposition~\ref{prop:tau-t-limit} ensures that $\mathsf{R}_{t_i, t_i} = \tau_{t_i}^2 \to \taustar^2$ as $t_{0} \to \infty.$
Under Assumption~\ref{assump:lambda} (so that $\sum_{t=1}^{\infty}\sqrt{\residualT} < +\infty$ and 
$\sum_{t=1}^{\infty}\exp\left\lbrace - c\Lambda_t\right\rbrace < \infty$), 
Lemma~\ref{lem:converge-r-general} immediately implies that $\left| \mathsf{R}_{t_1, t_2}-\taustar^2\right| \goto 0$ as $t_1, t_2 \geq t_0$ and $t_0 \goto \infty$. 
With the assistance of the aforementioned two lemmas, we are ready to proceed to the proofs of 
Lemma~\ref{lem:S-convergence} and Lemma~\ref{lem:convergence-x-z}.

\subsubsection{Proof of Lemma~\ref{lem:S-convergence}}
\label{sec:proof-S-convergence}
Now we are ready to prove Lemma~\ref{lem:S-convergence}.
It suffices to show that $\lim_{p\goto\infty} |S_{t_2}(\gamma)\setminus S_{t_1}(\gamma)|/p \leq \xi$ almost surely. 
As argued in the proof of BM-Lemma 3.5, it is guaranteed that 
	\begin{align*}
	\lim_{p\goto\infty} \frac{1}{p} \left| S_{t_2}(\gamma)\setminus S_{t_1}(\gamma)\right| = \lim_{p\goto\infty} \mathbb{P}\left\lbrace |\Theta+Z_{t_2-1}|\geq (1-\gamma)\zeta_{t_2-1}, |\Theta + Z_{t_1-1}|< (1-\gamma)\zeta_{t_1-1}
	\right\rbrace 	=: P_{t_1, t_2},
	\end{align*}
	where $(Z_{t_1}, Z_{t_2})$ are jointly Gaussian with $\mathbb{E}[Z_{t_1}^2]=\mathsf{R}_{t_1, t_1}$, $\mathbb{E}[Z_{t_2}^2]=\mathsf{R}_{t_2, t_2}$ and $\mathbb{E}[Z_{t_1}Z_{t_2}]=\mathsf{R}_{t_1, t_2}$. 
	If we denote $a \defn (1-\gamma)\alphastar\taustar$, 
	in view of Lemma~\ref{prop:tau-t-limit}, we know that $\forall \varepsilon > 0$ and large enough $t_{\star}$, it holds that $|(1-\gamma)\zeta_{t_i-1} - a|\leq \varepsilon$ for $i\in \left\lbrace 1, 2\right\rbrace $. 
	Then with the same argument as BM-Lemma 3.5, we reach
	\begin{align*}
	P_{t_1, t_2} &\leq \frac{1}{4\varepsilon^2}\left[\mathsf{R}_{t_1-1, t_1-1}-2\mathsf{R}_{t_1-1, t_2-1}+\mathsf{R}_{t_2-1, t_2-1}\right] + \frac{4\varepsilon}{\sqrt{2\pi \mathsf{R}_{t_1-1, t_1-1}}} \\
	&\leq \frac{\left[\sum_{i=t_\star}^{\infty} \sqrt{c_0 \exp\left\lbrace - c\Lambda_{i}\right\rbrace} +\sum_{i=t_\star}^{\infty} \sqrt{4L_\tau\tau^\star_{\max} l_i} \right]^2}{\varepsilon^2} + 
	\frac{2\varepsilon}{\sigma},
	\end{align*}
	as a consequence of Lemma~\ref{lem:converge-r-general} and $\tau_t \geq \sigma$, $\forall t \geq 1$. 

	Under Assumption~\ref{assump:lambda}, $\sum_{i=t_\star}^{\infty} \sqrt{l_i}  \to 0$ and $\sum_{i=t_\star}^{\infty} \exp\left\lbrace -c/2\Lambda_i\right\rbrace  \to 0$ as $t_\star$ increases. 
	Taking $\varepsilon = \left[\sum_{i=t_\star}^{\infty} \sqrt{c_0 \exp\left\lbrace - c\Lambda_{i}\right\rbrace} +\sum_{i=t_\star}^{\infty} \sqrt{4L_\tau \tau^\star_{\max} l_i} \right]^{2/3}$
	gives
	$$
	P_{t_1, t_2}\leq C'\left[\sum_{i=t_\star}^{\infty} \sqrt{c_0 \exp\left\lbrace - c\Lambda_{i}\right\rbrace} +\sum_{i=t_\star}^{\infty} \sqrt{4L_\tau \tau^\star_{\max} l_i} \right]^{2/3}
	.$$
	 The conclusion of Lemma~\ref{lem:S-convergence} thus follows immediately. 

\subsubsection{Proof of Lemma~\ref{lem:convergence-x-z}}
\label{sec:convergence-amp-estimate}

In view of the proof for BM-Lemma 4.3 --- a general result proved for any positive thresholding sequence $\left\lbrace \zeta_t\right\rbrace $), we obtain 
\begin{align*}
\lim_{p\goto\infty} \frac{1}{p\lambda_t^2}\|\bs{z}^{t}-\bs{z}^{t-1}\|_2^2 
~\overset{\mathrm{a.s.}}{=}~
\lim_{p\goto\infty} \frac{1}{p\lambda_t^2}\|\t^t-\t^{t-1}\|_2^2. 
\end{align*}
It is thus sufficient to prove Lemma~\ref{lem:convergence-x-z} for the $\t^t$ sequence only. 
As a consequence of the generalized state evolution formula (cf.~Lemma~\ref{lem:amp-fixed-t-2-step}), we can guarantee that 
\begin{align*}
\lim_{p\goto\infty} \frac{1}{p\lambda_t^2}\|\t^{t}-\t^{t-1}\|_2^2& = \frac{1}{\lambda_t^2} \mathbb{E}\left\lbrace \left[\eta(\Theta+Z_t; \zeta_t) - \eta(\Theta+Z_{t-1}; \zeta_{t-1}) \right]^2 \right\rbrace   \\
& \leq \frac{2(\zeta_t-\zeta_{t-1})^2}{\lambda_t^2} + \frac{2\mathbb{E}[(Z_t - Z_{t-1})^2]}{\lambda_t^2}.
\end{align*}

\begin{itemize}
\item The term $\mathbb{E}[(Z_t - Z_{t-1})^2]$ can be controlled via Lemma~\ref{lem:converge-cov}, where 
\begin{align}\label{eq:r-121}
\mathbb{E}[(Z_{t}-Z_{t-1})^2] = \mathsf{R}_{t, t} - 2 \mathsf{R}_{t-1, t} + \mathsf{R}_{t-1, t-1} \leq c_0 \exp\left\lbrace - c\Lambda_t\right\rbrace + 4L\tau^\star_{\max} l_{t}.
\end{align}
To control the right-hand side,  we first obtain from Assumption~\ref{assump:lambda} that $\sum_{s=t}^{\infty}\sqrt{l_s}$ is a converging sequence. In addition,  Assumption~\ref{assump:lambda} ensures that $\sum_{s=t}^{\infty} \lambda_s$ is a diverging sequence. Then we conclude that $l_{t-1}/\lambda_t^2 \goto 0$ as $t\goto \infty$. Additionally, by Assumption~\ref{assump:lambda}, we know that $\sum_{s=1}^{\infty}\exp\left\lbrace -\Lambda_t\right\rbrace $ has a finite limit, while $\sum_{i=1}^{\infty}\lambda_i$ diverges, thus $\exp\left\lbrace -c\Lambda_t\right\rbrace/\lambda_t^2 \goto 0$.
Taking these collectively ensures that
\begin{align*}
\frac{2\mathbb{E}[(Z_t - Z_{t-1})^2]}{\lambda_t^2} \goto 0.
\end{align*}

\item It remains to show that the term $2(\zeta_t-\zeta_{t-1})^2/\lambda_t^2$ vanishes as $t\goto \infty$. We obtain from Lemma~\ref{lem:derivative-lemma} that $0 \leq {\alphastar}'(0) \leq C$. It then immediately follows that
\begin{align*}
\left| \frac{\zeta_{t}-\zeta_{t-1}}{\lambda_t} \right| \leq \alphastart\left| \frac{\tau_t - \tau_{t-1}}{\lambda_t} \right| + \tau_{t-1} \left| \frac{\alphastart - \alpha^\star_{t-1}}{\lambda_t-\lambda_{t-1}}\right| \frac{\lambda_t-\lambda_{t-1}}{\lambda_t} \goto \alphastar\left| \frac{\tau_t - \tau_{t-1}}{\lambda_t} \right| + \taustar {\alphastar}'(0) \frac{\lambda_t-\lambda_{t-1}}{\lambda_t} \goto 0,
\end{align*}
where the limit value $|\tau_t - \tau_{t-1}|/\lambda_t \overset{\mathrm{a.s.}}{\to} 0$ follows from Lemma~\ref{lem:converge-cov}, and $(\lambda_t - \lambda_{t-1})/\lambda_t \to 0$ holds by virtue of  Assumption~\ref{assump:lambda}. 
\end{itemize}
Putting all this together completes the proof.

\subsubsection{Proof of Lemma~\ref{lem:converge-r-general}}
\label{sec:proof-converge-r-general}

The main idea of this proof is to decompose $\mathsf{R}_{t_1, t_2}-\taustar^2$ into terms of the form $\mathsf{R}_{t, t+1}-\taustar^2$ or $\mathsf{R}_{t, t}-\taustar^2$. 
Without loss of generality, we assume $t_2 > t_1$.
The proof of BM-Theorem 4.2 (which is a general result for any positive $\left\lbrace \zeta_t\right\rbrace $ sequence, so we can safely use the arguments there) ensures that: if we define 
\begin{align*}
\bs{h}^{t+1} = \coef - (\bs{X}^\top \bs{z}^t + \t^t),
\end{align*}
then the empirical distribution of $\left\lbrace \bs{h}^{i+1}\right\rbrace_{T\geq i \geq 0}$ converges  weakly to a sequence of Gaussian random variables $\left\lbrace Z_i\right\rbrace_{T\geq i\geq 0}$ 
as $p\goto\infty$. Here, $Z_0, Z_1, Z_2, \dots$ are Gaussian random variables,  defined on the same probability space with $\mathbb{E}[Z_t] = 0$ and $\mathbb{E}[Z_t Z_s] = \mathsf{R}_{t, s}$. 
In addition, one has 
\begin{align*}
	\mathsf{R}_{t_1, t_2} = \lim_{p\to\infty} \frac{1}{p}\inprod{\bs{h}^{t_1+1}}{\bs{h}^{t_2+1}}. 
\end{align*}
With these results in place, direct calculations yield 
\begin{align*}
\left|\mathsf{R}_{t_1, t_1}-2\mathsf{R}_{t_1, t_2}+\mathsf{R}_{t_2, t_2}\right| 
=  \mathbb{E} [(Z_{t_1}-Z_{t_2})^2] 
& = \sum_{i, j=t_1}^{t_2-1}\mathbb{E} [ (Z_{i+1}-Z_i)(Z_{j+1}-Z_j)]\\
& \leq \left[\sum_{i=t_1}^{t_2-1}\left\lbrace \mathbb{E} (Z_{i+1}-Z_i)^2\right\rbrace^{1/2}\right] ^2 \\
& \leq 4\left[\sum_{i=t_1}^{\infty} \sqrt{c_0 \exp\left\lbrace - c\Lambda_{i}\right\rbrace + 4L_\tau \tau^\star_{\max} l_{i}} \right]^2, 
\end{align*}
where the last inequality follows from inequality~\eqref{eq:r-121}. 
The proof of Lemma~\ref{lem:converge-r-general} is thus completed. 

\subsubsection{Proof of Lemma~\ref{lem:converge-cov}}
\label{sec:proof-convergence-cov}

Similar to the proof of BM-Lemma 5.7, we find it convenient to change coordinates and define 
\begin{align*}
	y_{t, 1} \defn \mathsf{R}_{t-1, t-1} = \tau_{t-1}^2,~~
	y_{t, 2} \defn \mathsf{R}_{t, t} = \tau_{t}^2,~~
	y_{t, 3} \defn \mathsf{R}_{t-1, t-1} - 2\mathsf{R}_{t-1, t} + \mathsf{R}_{t, t}.
\end{align*}
To capture the updating rule~\eqref{eq:def-r}, we introduce the following recursive formula via the mapping $\bs{y}_{t+1} = \mathsf{G}^t(\bs{y}_t)$: 
\begin{subequations}
	\label{eq:y-t}
	\begin{gather}
	\label{eq:y-t1}
	y_{t+1, 1} = \mathsf{G}^t_1(\bs{y}_t) \defn y_{t, 2}, \\
	\label{eq:y-t2}
	y_{t+1, 2} = \mathsf{G}^t_2(\bs{y}_t) \defn \sigma^2 + \frac{1}{\delta}\mathbb{E}\left\lbrace \left[ \eta(\Theta+Z_t; \alphastart \sqrt{y_{t, 2}})-\Theta\right]^2 \right\rbrace =\mathsf{F}(y_{t, 2}, \alphastart \sqrt{y_{t, 2}}), \\
	\label{eq:y-t3}
	y_{t+1, 3} = \mathsf{G}^t_3(\bs{y}_t) \defn \frac{1}{\delta}\mathbb{E}\left\lbrace \left[ \eta(\Theta+Z_t; \alphastart \sqrt{y_{t, 2}})-\eta(\Theta+Z_{t-1}; \alpha_{t-1}^\star \sqrt{y_{t, 1}})\right]^2 \right\rbrace,
	\end{gather}
\end{subequations}
where $(Z_t, Z_{t-1})$ are jointly mean-zero Gaussian random variables with  $\mathbb{E}[Z_t^2]=y_{t, 2}$, $\mathbb{E}[Z_{t-1}^2]=y_{t, 1}$ and $\mathbb{E}[(Z_t-Z_{t-1})^2] =y_{t, 3}$. 
In contrast to BM-Lemma 5.7, the mapping $\mathsf{G}^t$ is different for each step $t$ since at different $t$, $\alphastart$ is chosen as the unique solution to the equation set~\eqref{eq:fixed-point}, which varies across iterations. 
(As such, the mapping $\mathsf{G}^t$ is well-defined for $y_{t, 3} \leq 2(y_{t, 1} + y_{t, 2})$.) 
In addition, the recursive updates are initialized to $Z_0\sim \NORMAL(0, 1)$ and 
\begin{align*}
	y_{1, 1} &\defn \sigma^2 +\frac{1}{\delta}\mathbb{E}[\Theta^2];\\
	y_{1, 2} &\defn \sigma^2 + \frac{1}{\delta}\mathbb{E}\left\lbrace \left[\eta(\Theta+Z_0; \zeta_0)-\Theta \right] ^2\right\rbrace ;\\
	y_{1, 3} &\defn \frac{1}{\delta}\mathbb{E}\left\lbrace \eta(\Theta+Z_0;\zeta_0)^2\right\rbrace.
\end{align*}
We first make note of the following fact: if $y_{t, 1} = y_{t, 2}={\taustart}^2$, then one has $y_{t+1, 1} = y_{t+1, 2} = {\taustart}^2$ since $\alphastart$ satisfies 
the equation set~\eqref{eq:fixed-point} where 
\begin{align*}
	{\taustart}^2 = \sigma^2 + \frac{1}{\delta}\mathbb{E}\left\lbrace \left[ \eta(\Theta+\taustart Z_t; \alphastart \taustart)-\Theta\right]^2 \right\rbrace. 
\end{align*}
In other words, $\bs{y}^\star_{t}$ is a fixed point of the mapping $\mathsf{G}^t$, namely, $\bs{y}^\star_{t} = \mathsf{G}^t(\bs{y}^\star_t). $
In addition, one has $\lim_{t\goto \infty}\taustart = \taustar$ by Proposition~\ref{prop:tau-t-limit}.

We claim that the following three properties hold true. 
\begin{enumerate}
	\item As $t\goto\infty$, the update sequence satisfies $y_{t, 1}\goto {\taustart}^2$ and $y_{t, 2}\goto {\taustart}^2$, with $y_{t, 3} < y_{t, 1}+y_{t, 2}-\sigma^2$.

	\item Define another recursive updating rule $\widetilde{y}_{t+1, 3} \defn \mathsf{G}_3^t({\taustart}^2, {\taustart}^2, \widetilde{y}_{t, 3})$. 
	If we can find some $t_0\geq T_{\min}$ such that $\widetilde{y}_{t_0, 3}< 2 {\tau^\star_{t_0}}^2$, then the following two properties are satisfied 
	\begin{enumerate}
	\item for all $t \geq t_0$, it holds that $\widetilde{y}_{t, 3} < 2 \taustart^2$;

	\item it follows that 
	$\widetilde{y}_{t, 3} \goto 0$ as $t \goto \infty$.
	\end{enumerate}
	Here $T_{\min}$ is some constant that is pre-determined by $\Theta$ and $\sigma$.

	\item When $t\geq t_0$, the Jacobian $J_t\defn  J_{G^t}(\bs{y}^\star_t)$ of $G^t$ at $\bs{y}^\star_t\defn ({\taustart}^2, {\taustart}^2, 0)$ has spectral radius
	\begin{align}
	\label{eqn:J-t-radius}
	\sigma(J_t) \leq 1 - \frac{\lambda_t}{2\alpha^\star_{\max}\tau^\star_{\max}}.
	\end{align}
\end{enumerate}

Let us take these claims as given for the moment and proceed to the proof of Lemma~\ref{lem:converge-cov}. 
In light of the first claim, we obtain that for all large enough $t$, one has $y_{t, 3}\leq 2 \taustart^2-\sigma^2$.
Then the second claim further guarantees that $y_{t, 3}\goto 0$.
Taking the first two claims collectively implies that $\bs{y}_t \goto \bs{y}^\star_t = (\taustart^2, \taustart^2, 0)$ as $t\goto\infty$. 
In addition, by virtue of the third property, for appropriately large $t$, we obtain
\begin{align}
\label{eqn:dist-y-to-ystar}
\notag 
\|\bs{y}_{t+1} - \bs{y}^\star_{t+1}\|_2 
= \|\bs{y}_{t+1} - \bs{y}^\star_{t}\|_2 + \|\bs{y}^\star_{t+1} - \bs{y}^\star_t\|_2
& \stackrel{(\mathrm{i})}{\leq} \sigma(J_t)\|\bs{y}_t - \bs{y}^\star_t\|_2 + \|\bs{y}^\star_{t+1} - \bs{y}^\star_t\|_2 \\
& \stackrel{(\mathrm{ii})}{\leq} \exp\left\lbrace -c \lambda_t \right\rbrace \|\bs{y}_t - \bs{y}^\star_t\|_2 +4\tau^\star_{\max}  |\taustart-\tau^\star_{t+1}|, 
\end{align}
where we define $c = (2\alpha^\star_{\max}\tau^\star_{\max})^{-1}$. 
Here, the inequality $(\mathrm{i})$ uses the relations 
\begin{align*}
	\bs{y}_{t+1} = \mathsf{G}^t(\bs{y}_t) 
	~\text{ and }~
	\bs{y}^\star_{t} = \mathsf{G}^t(\bs{y}^\star_t),
\end{align*}
and $(\mathrm{ii})$ uses the property~\eqref{eqn:J-t-radius}. 
Recalling that we define $\Lambda_t \defn \sum_{i\leq t} \lambda_{i}$, we can apply the inequality~\eqref{eqn:dist-y-to-ystar} recursively to yield 
\begin{align*}
\frac{\|\bs{y}_{t+1} - \bs{y}^\star_{t+1}\|_2}{\exp\left\lbrace-c \Lambda_t \right\rbrace } 
& \ \leq \frac{\|\bs{y}_{t} - \bs{y}^\star_{t}\|_2}{\exp\left\lbrace-c \Lambda_{t-1} \right\rbrace} + 4\tau^\star_{\max} \frac{|\taustart-\tau^\star_{t+1}| }{\exp\left\lbrace-c \Lambda_t \right\rbrace} \\
& \ \leq \frac{\|\bs{y}_{t_0} - \bs{y}^\star_{t_0}\|_2}{\exp\left\lbrace-c \Lambda_{t_0-1} \right\rbrace} + 4\tau^\star_{\max}\sum_{s=t_0}^t \frac{|\taustart-\tau^\star_{s+1}| }{\exp\left\lbrace-c \Lambda_s \right\rbrace} \\
& \leq \  \frac{\|\bs{y}_{t_0} - \bs{y}^\star_{t_0}\|_2}{\exp\left\lbrace-c \Lambda_{t_0-1} \right\rbrace} + 4L_\tau \tau^\star_{\max}\sum_{s=1}^t \frac{|\lambda_s-\lambda_{s+1}| }{\exp\left\lbrace-c \Lambda_s \right\rbrace}.
\end{align*}
Then we can conclude that, with $c_0 \defn \|\bs{y}_{t_0} - \bs{y}^\star_{t_0}\|_2 \exp\left\lbrace c \Lambda_{t_0-1} \right\rbrace$, 
\begin{align*}
\|\bs{y}_{t+1} - \bs{y}^\star_{t+1}\|_2 \leq c_0 \exp\left\lbrace - c\Lambda_t\right\rbrace + 4L\tau^\star_{\max} \sum_{s=1}^t |\lambda_s-\lambda_{s+1}| \exp\left\lbrace - c [\Lambda_t - \Lambda_s] \right\rbrace.
\end{align*}
Finally, recalling the definition
\begin{align*}
y_{t, 1}  = \mathsf{R}_{t-1, t-1} = \tau_{t-1}^2,~~
y_{t, 2}  = \mathsf{R}_{t, t} = \tau_{t}^2,~~
y_{t, 3} = \mathsf{R}_{t-1, t-1} - 2\mathsf{R}_{t-1, t} + \mathsf{R}_{t, t}, 
\end{align*}
we can show that
\begin{align*}
\max\left\lbrace |\mathsf{R}_{t, t} - \taustart^2|, \left|\mathsf{R}_{t, t} - 2\mathsf{R}_{t, t+1} + \mathsf{R}_{t+1, t+1}\right|  \right\rbrace \leq c_0 \exp\left\lbrace - c\Lambda_t\right\rbrace + 4L\tau^\star_{\max} \sum_{s=1}^t |\lambda_s-\lambda_{s+1}| \exp\left\lbrace - c [\Lambda_t - \Lambda_s] \right\rbrace,
\end{align*}
which establishes the advertised result in Lemma~\ref{lem:converge-cov}. 
It remains to prove the three claims above.


\subsubsection{Proof of the three claims} 

In the sequel, we look at these three claims separately. 

\subsubsection*{Proof of Claim~1} 
The iteration updates~\eqref{eq:y-t1} and \eqref{eq:y-t2}, along with the initial condition, yield that $y_{t+1, 1} = y_{t, 2} = \tau_t$, $\forall t \geq 1$. Taking the results in Proposition~\ref{prop:tau-t-limit} collectively with $\lim_{t\goto\infty} \tau_t = \taustar$ and $\lim_{t\goto\infty} \taustart = \taustar$, we obtain  $$y_{t, 1} - {\taustart}^2 \to 0 \qquad  \text{and} \qquad y_{t, 2} - {\taustart}^2 \to 0.$$

The second part of the claim is proved by induction; for the proof to go through, let us verify the initial condition and the induction step respectively.

\begin{itemize}
	\item For the initial condition, we can write  
	\begin{align*}
	y_{1, 1} + y_{1, 2} - y_{1, 3} = 2\sigma^2 + \frac{1}{\delta}\mathbb{E}\left\lbrace 
	\Theta \left(\Theta - \eta(\Theta+Z_0;\zeta_0)
	\right)\right\rbrace.
	\end{align*}
	It is easy to check that the function $\theta \mapsto \theta - \eta(\theta+Z_0; \zeta_0)$ is monotonically increasing for any $Z_0$ and $\zeta_0$; therefore, the random variables $\Theta$ and $\Theta - \eta(\Theta+Z_0;\zeta_0)$ are positively correlated. In other words, 
	\begin{align*}
	\frac{1}{\delta}\mathbb{E}\left\lbrace 
	\Theta \left(\Theta - \eta(\Theta+Z_0;\zeta_0)
	\right)\right\rbrace \geq 0\quad  \Longrightarrow \quad y_{1, 1} + y_{1, 2} - y_{1, 3} \geq 2 \sigma^2.
	\end{align*}

	\item 
	Suppose $y_{t, 3} < y_{t, 1}+y_{t, 2}-\sigma^2$ holds for all steps up to $t$. 
	At step $t+1$, the iteration formula~\eqref{eq:y-t} directly leads to
	\begin{align*}
	y_{t+1, 1} + y_{t+1, 2} - y_{t+1, 3} = 2\sigma^2 + \frac{2}{\delta}\mathbb{E}\left[ \left( \eta(\Theta + Z_t; \alphastart\sqrt{y_{t, 2}})-\Theta \right) \left( \eta(\Theta + Z_{t-1}; \alpha_{t-1}^\star \sqrt{y_{t, 1}})-\Theta \right) \right],
	\end{align*}
	where $(Z_t, Z_{t-1})$ are jointly Gaussian with $\mathbb{E}[Z_t^2]=y_{t, 2}$, $\mathbb{E}[Z_{t-1}^2]=y_{t, 1}$ and $\mathbb{E}[(Z_t-Z_{t-1})^2] =y_{t, 3}$.
	By definition of $Z_t$ and $Z_{t-1}$, 
	one has $\mathbb{E}[Z_t Z_{t-1}] = (y_{t, 1}+y_{t, 2}-y_{t, 3})/2$ 
	which stays positive given the induction assumption. 

	Now since the mapping $x \mapsto \eta(x+\theta;\zeta)-\theta$ is monotone in $x$, we can readily conclude that 
	\begin{align*}
	\mathbb{E}\left[ \left(\eta(\Theta + Z_t; \alphastart\sqrt{y_{t, 2}})- \Theta \right)\left(\eta(\Theta + Z_{t-1}; \alpha_{t-1}^\star\sqrt{y_{t, 1}})-\Theta \right)\right]\geq 0, 
	\end{align*}
	which further leads to 
	$y_{t+1, 1} + y_{t+1, 2} - y_{t+1, 3} > 2\sigma^2$. 
	We have thus proved that $y_{t+1, 3}<y_{t+1, 1}+y_{t+1, 2}-\sigma^2 $ at $t+1$. 
\end{itemize}
Combining the initial condition with the induction argument, we conclude that $y_{t, 3}<y_{t, 1}+y_{t, 2}-\sigma^2$ for all $t$.

\subsubsection*{Proof of Claim~2(a)} 

To establish the second claim, with a slight abuse of notation, we define
\begin{align}
\label{eq:def-G3}
G_3(y_1, y_2, y_3; \alpha_1, \alpha_2)  \defn \frac{1}{\delta}\mathbb{E}\left\lbrace \left[ \eta(\Theta+Z_1; \alpha_1 \sqrt{y_1})-\eta(\Theta+Z_2; \alpha_2 \sqrt{y_2})\right]^2 \right\rbrace,
\end{align}
where $(Z_1, Z_2)$ are jointly Gaussian with zero mean and $\mathbb{E}[Z_1^2] = y_1$, $\mathbb{E}[Z_2^2]=y_2$ and $\mathbb{E}[(Z_1-Z_2)^2] = y_3$. 
We remark that the dependence on $y_{3}$ is only through the covariance between $Z_{1}$ and $Z_{2}$. 
From the definition of $\widetilde{y}_{t+1, 3}$, we immediately have
\begin{align*}
\widetilde{y}_{t+1, 3} = G_3(\taustart^2, \taustart^2, \widetilde{y}_{t, 3}, \alphastart,  \alpha^\star_{t-1}).
\end{align*}
It is straightforward to verify that $G_3$ is continuous in all the parameters under our assumption that $\mathbb{E}[\Theta^2]<\infty$. 
The continuity of $G_3$, together with the fact that ${\alphastart}^2$ converges,  
implies that $\widetilde{y}_{t+1, 3} - G_3({\taustart}^2, {\taustart}^2, \widetilde{y}_{t, 3}; \alphastart, \alphastart) \goto 0$, for any fixed $t \geq 0$.

With this notation in mind, we are ready to prove that $\widetilde{y}_{t+1, 3} < 2 (\tau_{t+1}^\star)^2$ for all sufficiently large $t$ via an inductive argument. 
First, suppose $\widetilde{y}_{i, 3} < 2 (\tau_{i}^\star)^2$ holds for all steps between $t_{0}$ up to $t$.
In view of relation ${\taustart}^2 = \sigma^2 + \frac{1}{\delta}\mathbb{E}\left\lbrace \left[ \eta(\Theta+Z_t; \alphastart \taustart)-\Theta\right]^2 \right\rbrace$, we arrive at 
\begin{align*}
2{\taustart}^2 - G_3({\taustart}^2, {\taustart}^2, \widetilde{y}_{t, 3}; \alphastart, \alphastart) = 2\sigma^2 + \frac{2}{\delta}\mathbb{E}\left\lbrace \left( \eta(\Theta + Z_1; \alphastart\taustart) -\Theta \right) \left(  \eta(\Theta + Z_2; \alphastart\taustart)-\Theta \right) \right\rbrace,
\end{align*}
where again $(Z_1, Z_2)$ are jointly Gaussian with zero mean and $\mathbb{E}[Z_1^2] = \taustart^2$, $\mathbb{E}[Z_2^2]=\taustart^2$ and $\mathbb{E}[Z_1 Z_2] = 2{\taustart}^2-\widetilde{y}_{t, 3}$.
Under the induction assumption $\widetilde{y}_{t, 3} \leq 2{\taustart}^2$, the two random variables $ \eta(\Theta + Z_1; \alphastart\taustart)-\Theta $ and $ \eta(\Theta + Z_2; \alphastart\taustart)-\Theta $ are positively or zero correlated. Then it is guaranteed that   
\begin{align}
\label{eq:g3-tau}
\frac{2}{\delta}\mathbb{E}\left\lbrace \left( \eta(\Theta + Z_1; \alphastart\taustart) -\Theta \right) \left(  \eta(\Theta + Z_2; \alphastart\taustart)-\Theta \right) \right\rbrace\geq 0 \quad \Longrightarrow \quad 2{\taustart}^2 - G_3({\taustart}^2, {\taustart}^2, \widetilde{y}_{t, 3}; \alphastart, \alphastart) \geq 2\sigma^2 .
\end{align}

Now to prove the upper bound for $\widetilde{y}_{t+1, 3}$, it is sufficient for us to control the difference between 
$G_3({\taustart}^2, {\taustart}^2, \widetilde{y}_{t, 3}; \alphastart, \alphastart)
$ and $ \widetilde{y}_{t+1, 3}$.
Towards this end, let us invoke the Lipschitz property for soft-thresholding function
\begin{align*}
 	\left|\eta(\Theta+Z_1; \alphastart \ts)-\eta(\Theta+Z_1; \a^\star_{t-1} \ts)\right|\leq \taustart |\alphastart-\alpha_{t-1}^\star|,
 \end{align*}
which leads to  
\begin{equation}\label{eq:g3-y}
\begin{aligned}
& |G_3({\taustart}^2, {\taustart}^2, \widetilde{y}_{t, 3}; \alphastart, \alphastart) - \widetilde{y}_{t+1, 3}| \\
= &  |G_3({\taustart}^2, {\taustart}^2, \widetilde{y}_{t, 3}; \alphastart, \alphastart) - G_3({\taustart}^2, {\taustart}^2, \widetilde{y}_{t, 3}; \alpha_{t-1}^\star, \alphastart) | \\
= &\  \frac{1}{\delta} \mathbb{E} \left| \left[\eta(\Theta+Z_1; \alphastart \ts)-\eta(\Theta+Z_1; \a^\star_{t-1} \ts)\right] \left[\eta(\Theta+Z_1; \alphastart \ts)+\eta(\Theta+Z_1; \a^\star_{t-1} \ts)-2\eta(\Theta+Z_2; \alphastart \ts)\right]\right| \\
\leq & \   \frac{\taustart |\alphastart-\alpha_{t-1}^\star|}{\delta} \mathbb{E} \left|  \eta(\Theta+Z_1; \alphastart \ts)+\eta(\Theta+Z_1; \a^\star_{t-1} \ts)-2\eta(\Theta+Z_2; \alphastart \ts)\right|\\
\stackrel{(\mathrm{i})}{\leq}  & \ 
\frac{\ts|\alphastart-\a_{t-1}^\star|}{\delta} \cdot 4(\mathbb{E}[|\Theta|]+\tau^\star_{\max})
\leq \frac{4(\mathbb{E}[|\Theta|]+\tau^\star_{\max})\tau^\star_{\max}}{\delta} |\alphastart-\a^\star_{t-1}|.
\end{aligned}
\end{equation}
Here, the inequality $(\mathrm{i})$ follows since 
$\mathbb{E}\left[ \left| \eta(\Theta+Z_i; \alphastart \tau^\star_s)\right| \right] \leq \mathbb{E}\left[ |\Theta + Z_1| \right] \leq \mathbb{E}[|\Theta|] + \tau^\star_{\max}$ with $i\in \left\lbrace 1,2\right\rbrace $ and $s\in \left\lbrace t, t-1\right\rbrace $.

Combining the inequalities~\eqref{eq:g3-tau} and \eqref{eq:g3-y} yields 
\begin{align*}
\widetilde{y}_{t+1, 3} \leq 2 {{\tau_{t+1}^\star}}^2 + \frac{4(\mathbb{E}[|\Theta|]+\tau^\star_{\max})\tau^\star_{\max}}{\delta} |\alphastart-\a_{t-1}^\star| + 2(\taustart^2-{\tau^\star_{t+1}}^2) - 2\sigma^2.
\end{align*}
As $t\goto\infty$, we know from Proposition~\ref{prop:tau-t-limit} that $ |\alphastart-\a_{t-1}^\star| \goto 0$, as well as $|\taustart^2-{\tau^\star_{t+1}}^2| \goto 0$; as such, we can always find some $T_{\min}$ determined by $\Theta$ and $\sigma^{2}$, such that the $\frac{4(\mathbb{E}[|\Theta|]+\tau^\star_{\max})\tau^\star_{\max}}{\delta} |\alphastart-\a_{t-1}^\star| + 2|\taustart^2-{\tau^\star_{t+1}}^2| - 2\sigma^2 < 0$. 
In this case, we conclude that $\widetilde{y}_{t+1, 3} < 2 {{\tau_{t+1}^\star}}^2$, and then by induction, $\widetilde{y}_{t, 3} < 2 {{\tau_{t+1}^\star}}^2$, $\forall t \geq t_0$.
This completes the proof of Claim 2(a).

\subsubsection*{Proof of Claim~2(b)} 

It turns out that the function $G_3({\taustart}^2, {\tau_{t}^\star}^2, \widetilde{y}_{t, 3}; \alphastart, \alphastart)$ is the same as the $\mathsf{G}_*$ function defined in BM-Pg 36. 
We record here  two key observations from BM-Pg 36 regarding this function. 
\begin{itemize}
	\item The derivative of function $G_{3}$ with respect to its third argument satisfies 
	\begin{align}\label{eq:g-derivative}
	\left. \frac{\partial}{\partial x}G_3({\taustart}^2, {\tau_{t}^\star}^2, x; \alphastart, \alphastart)\right|_{x=0} = \frac{1}{\delta}\mathbb{E}\left\lbrace \eta'(\Theta + \taustart Z; \alphastart \taustart )\right\rbrace = 1-\frac{\lambda_t}{\alphastart\taustart};
	\end{align}
	\item The mapping $x\mapsto \frac{\partial}{\partial x}G_3({\taustart}^2, {\tau_{t}^\star}^2, x; \alphastart, \alphastart)$ is decreasing in $[0, 2\taustart^2)$.
	\end{itemize}

	In addition, we claim that $G_3({\taustart}^2, {\tau_{t}^\star}^2, 0; \alphastart, \alphastart) = 0$. In order to see this, by construction in \eqref{eq:def-G3}, we have 
		\begin{align*}
		G_3({\taustart}^2, \taustart^2, 0; \alphastart, \alphastart) = \frac{1}{\delta}\mathbb{E}\left\lbrace \left[ \eta(\Theta+Z_1; \alphastart \taustart)-\eta(\Theta+Z_2; \alphastart \taustart)\right]^2 \right\rbrace,
		\end{align*}
	where $\mathbb{E}[Z_1^2] = \mathbb{E}[Z_2^2] = \taustart^2$ and $\mathbb{E}[(Z_1-Z_2)^2] = 0$.
    From this construction, we know $Z_1 = Z_2$, almost surely; it follows that $\eta(\Theta+Z_1; \alpha_1 \taustart) = \eta(\Theta+Z_1; \alpha_1 \taustart)$, almost surely. We therefore obtain $G_3({\taustart}^2, {\tau_{t}^\star}^2, 0; \alphastart, \alphastart) = 0$.

Taking these observations collectively 
 guarantees that: under the condition $\widetilde{y}_{t, 3} < 2 {\tau_{t}^\star}^2$, one has  
\begin{align*}
G_3({\taustart}^2, {\tau_{t}^\star}^2, \widetilde{y}_{t, 3}; \alphastart, \alphastart) \leq (1-\lambda_t/\alphastart\taustart) \widetilde{y}_{t, 3} \leq (1-(\alpha_{\max}^{\star} \tau_{\max}^{\star})^{-1} \lambda_t)\widetilde{y}_{t, 3}.
\end{align*}
In view of the inequality~\eqref{eq:g3-y}, we have 
$$|G_3({\taustart}^2, {\taustart}^2, \widetilde{y}_{t, 3}; \alphastart, \alphastart) - \widetilde{y}_{t+1, 3}| \leq C|\alphastart-\a_{t-1}^\star|,$$ 
where $C$ is determined by $\Theta$ and $\sigma^2$. Putting these pieces together leads to 
\begin{align}
	\label{eq:g3-y-2}
	\widetilde{y}_{t+1, 3} \leq (1-(\alpha_{\max}^{\star} \tau_{\max}^{\star})^{-1} \lambda_t)\widetilde{y}_{t, 3} + C|\a_{t+1}^\star-\alphastart|.
\end{align}
Invoking the above relation recursively,  we obtain, for $c_1 \defn (\alpha_{\max}^{\star} \tau_{\max}^{\star})^{-1}$, that
\begin{align}
	\widetilde{y}_{t+1, 3} 
	& \leq 
	\prod_{i=0}^{t-t_0} (1-c_1\lambda_{t-i}) \widetilde{y}_{t_0}
	+ C \sum_{i=0}^{t-t_0} \prod_{j=1}^i (1-c_1\lambda_{t+1-j})|\alpha^\star_{t+1-i} - \alpha^\star_{t-i}| \notag\\
	& \leq 
	e^{-c_1 \sum_{i=t_0}^{t}\lambda_i} \widetilde{y}_{t_0}
	+ 
	C \sum_{i=t_0}^{t} e^{-c_1\sum_{j=i+1}^t \lambda_j} |\alpha^\star_{i+1} - \alpha^\star_{i}| \notag\\
	& = 
	e^{-c_1 \sum_{i=t_0}^{t}\lambda_i} \widetilde{y}_{t_0}
	+ 
	C \sum_{i=t_0}^{t/2} e^{-c_1\sum_{j=i+1}^t \lambda_j} |\alpha^\star_{i+1} - \alpha^\star_{i}|
	+
	C \sum_{i=t/2}^{t} e^{-c_1\sum_{j=i+1}^t \lambda_j} |\alpha^\star_{i+1} - \alpha^\star_{i}|.
	\label{eq:condition-1357}
\end{align}
Observing that $\widetilde{y}_{t_0, 3}< 2 {\tau^\star_{t_0}}^2$ and the sequence 
$\sum_{i=t_0}^{t}\lambda_i$ diverges to infinity, we know that the first term on the right-hand side of \eqref{eq:condition-1357} converges to zero as $t$ increases. 
Moreover, as proved in Proposition~\ref{prop:amp-fixd-point-convergence}, the sequence $\left\lbrace \alphastart\right\rbrace $ is monotone when $t$ is sufficiently large, and we then see that $\sum_{t=t/2}^{+\infty} |\a_{t+1}^\star-\alphastart|$ converges to zero. 
In the meantime, $e^{-c_1\sum_{j=i+1}^t \lambda_j} \leq e^{-c_1 \lambda_t} \leq e^{-c_1 \max_t \lambda_t}$ is always controlled by some constant. 
As a result, the third term on the right-hand side of \eqref{eq:condition-1357} also vanishes. 
It remains to control the second term on the right-hand side of \eqref{eq:condition-1357}. 
To this end, we make the observation that $|\alpha^\star_{i+1} - \alpha^\star_{i}| \leq 2\alpha_{\max}^\star$ and 
\begin{align*}
	\sum_{i=t_0}^{t/2} e^{-c_1\sum_{j=i+1}^t \lambda_j}
	\leq 
	t  e^{-c_1\sum_{j=t/2}^t \lambda_j} =   e^{(\log t - c_1\sum_{j=t/2}^t \lambda_j)} \to 0,
\end{align*}
where the last inequality follows from Assumption~\ref{assump:lambda}. Thus, the second term on the right-hand side of \eqref{eq:condition-1357} also has a zero limit. 
Taking these collectively, we establish the advertised result in Claim~2(b).

\subsubsection*{Proof of Claim 3.} 
To begin with, by some direct algebra, we can express the Jacobian matrix of mapping $\mathsf{G}^t$ at $\bs{y}_\star = ({\tau_t^\star}^2, {\tau_t^\star}^2, 0)$ by 
\begin{align*}
	J_{\mathsf{G}^t}(\bs{y}_\star) = \begin{pmatrix}
	0 & 1 & 0 \\
	0 & \left. \frac{\mathrm{d}}{\mathrm{d} \tau^2} \mathsf{F}(\tau^2, \alphastart\tau)\right|_{{\tau_t^\star}^2} & 0 \\
	\cdots & \cdots & \left. \frac{\mathrm{d}}{\mathrm{d} y_3}\mathsf{G}_3^t({\tau_t^\star}^2, {\tau_t^\star}^2, y_3)\right|_{y_3 = 0}
\end{pmatrix}.
\end{align*}
It is easily seen that the maximal eigenvalues of the above matrix satisfies the following relation
\begin{align}
	\label{eq:G-radius}
	\sigma(J_{\mathsf{G}^t}(\bs{y}_\star)) = \max\left\lbrace \left. \frac{\mathrm{d}}{\mathrm{d} \tau^2} \mathsf{F}(\tau^2, \alphastart\tau)\right|_{{\tau_t^\star}^2} , \left. \frac{\mathrm{d}}{\mathrm{d} y_3}\mathsf{G}_3^t({\tau_t^\star}^2, {\tau_t^\star}^2, y_3)\right|_{y_3 = 0} \right\rbrace .
\end{align}
To control $\sigma(J_{\mathsf{G}^t}(\bs{y}_\star))$, it is sufficient for us to bound each term respectively. 
For the first term, from the proof of BM-Proposition 1.3, we know that the function $\tau^2 \mapsto \mathsf{F}(\tau^2, \a\tau)$ is concave, for any $\a>0$ and $\Theta$ that is not identically $0$; therefore, we obtain
\begin{align*}
\left. \frac{\mathrm{d}}{\mathrm{d} \tau^2} \mathsf{F}(\tau^2, \alphastart\tau)\right|_{{\tau_t^\star}^2} \leq \frac{\mathsf{F}(\taustart^2, \alphastart\taustart) - \mathsf{F}(0, 0)}{\taustart^2-0}=\frac{\taustart^2-\sigma^2}{\taustart^2} \leq 1 - \frac{\sigma^2}{{\tau^\star_{\max}}^2} \defn \eta.
\end{align*}
For the second term, with the function $G_3$ defined in expression~\eqref{eq:def-G3}, we write 
\begin{align*}
\left. \frac{\mathrm{d}}{\mathrm{d} y_3} \mathsf{G}_3^t({\tau_t^\star}^2, {\tau_t^\star}^2, y_3)\right|_{y_3 = 0} & = \left. \frac{\mathrm{d}}{\mathrm{d} x} G_3({\tau_t^\star}^2, {\tau_t^\star}^2, x; \alpha_{t-1}^\star, \alphastart)\right|_{x = 0}\\
& = \left. \frac{\mathrm{d}}{\mathrm{d} x} \mathbb{E}\left[ \eta(\Theta+ a(x) Z + b(x) W; \alphastart {\tau_t^\star})\eta(\Theta + {\tau_t^\star} Z; \alpha_{t-1}^\star {\tau_t^\star})\right]\right|_{x=0}.
\end{align*}
Here $Z$ and $W$ are independent standard Gaussian random variables and we denote
\begin{align}
\label{eqn:a-and-b}
	a(x) \defn {\tau_t^\star}-x/(2{\tau_t^\star})
	\qquad\text{and }~b(x)\defn \sqrt{x-x^2/(4{\tau_t^\star}^2)}. 
\end{align}
Further define
\begin{align}
\label{eq:def-g}
g(\Theta, Z, W; x) \defn \eta(\Theta + a(x) Z + b(x) W; \alphastart {\tau_t^\star})\Big[\eta(\Theta + {\tau_t^\star} Z; \alpha_{t-1}^\star \taustar)- \eta(\Theta + \taustar Z; \alphastart \tau^\star_t)\Big].
\end{align}
With this notation in place, we can write 
\begin{equation}
\label{eq:g3-bound-1}
\begin{aligned}
\left. \frac{\mathrm{d}}{\mathrm{d} y_3} \mathsf{G}_3^t(\taustart^2, \taustart^2, y_3)\right|_{y_3 = 0}  
& \ =  \left. \frac{\mathrm{d}}{\mathrm{d} y_3} G_3(\taustart^2, \taustart^2, x;
\alphastart, \alphastart )\right|_{x = 0} + \left.\frac{\mathrm{d}}{\mathrm{d} x}  \mathbb{E}\left\lbrace g(\Theta, Z, W; x)\right\rbrace \right|_{x=0}\\
& \ = 1 - \frac{\lambda_t}{\taustart\alphastart}   + \left.\frac{\mathrm{d}}{\mathrm{d} x} \mathbb{E}\left\lbrace g(\Theta, Z, W; x)\right\rbrace \right|_{x=0},
\end{aligned}
\end{equation}
where the last equality follows from the expression~\eqref{eq:g-derivative}.
To control the spectral radius of the Jacobian, it suffices to control the last term on the right-hand side above. 
We claim that it satisfies the following property: 
\begin{align}
\label{eq:diff-derivative}
\left.\frac{\mathrm{d}}{\mathrm{d} x} \mathbb{E}\left\lbrace g(\Theta, Z, W; x)\right\rbrace \right|_{x=0} \leq 2 |\alphastart-\a_{t-1}^\star|.
\end{align}
Taking this claim as given for the moment, we can translate the expression~\eqref{eq:g3-bound-1} into 
\begin{align*}
\left. \frac{\mathrm{d}}{\mathrm{d} y_3} \mathsf{G}_3^t(\taustart^2, \taustart^2, y_3)\right|_{y_3 = 0} \leq 1 - \frac{\lambda_t}{\alpha^\star_{\max}\tau^\star_{\max}} + 2|\alphastart-\a_{t-1}^\star|.
\end{align*}
Combining the results of Proposition~\ref{prop:amp-fixd-point-convergence} and Assumption~\ref{assump:lambda} gives
\begin{align*}
	|\alphastart-\a_{t-1}^\star| \leq L_{\a} |\lambda_t - \lambda_{t-1}| \leq  \frac{\lambda_t}{4\alpha^\star_{\max}\tau^\star_{\max}} 
\end{align*}
for sufficiently large $t$.
Thus we complete the proof of expression~\eqref{eqn:J-t-radius}. 
The only thing that is left is to establish the inequality~\eqref{eq:diff-derivative}, which shall be done as follows.

\paragraph{Proof of the inequality~\eqref{eq:diff-derivative}.}
Throughout this part, we denote $a=a(x)$ and $b=b(x)$ for simplicity if there is no confusion. 
Direct calculation of the derivative for the expression~\eqref{eq:def-g} gives 
\begin{align*}
	\left.\frac{\mathrm{d}}{\mathrm{d} x} \mathbb{E}\left\lbrace g(\Theta, Z, W; x)\right\rbrace \right|_{x=0} = (\mathrm{I}) + (\mathrm{II}),
\end{align*}
where 
\begin{subequations}
	\begin{align*}
	(\mathrm{I}) \defn \left. \mathbb{E}\left\lbrace - \frac{Z}{2\taustart} \left[ \eta(\Theta + \taustart Z; \alpha_{t-1}^\star \taustart)-\eta(\Theta + \taustart Z; \alphastart \taustart)\right] \eta'(\Theta + a Z + b W; \alphastart \taustart)\right\rbrace\right|_{x=0};
	\end{align*}
	\begin{align*}
	(\mathrm{II}) \defn \left. \mathbb{E}\left\lbrace - \frac{W}{2 b} \left[ \eta(\Theta + \taustart Z; \alpha_{t-1}^\star \taustart)-\eta(\Theta +\taustart Z; \alphastart \taustart)\right] \eta'(\Theta + a Z + b W; \alphastart \taustart)\right\rbrace\right|_{x=0}.
	\end{align*}
\end{subequations}
Next we shall control each term respectively. 
For the first term, invoking the Lipschitz property yields $|\eta'(x; 
\zeta)| \leq 1$ and$|\eta(x;\zeta_1)-\eta(x;\zeta_2)|\leq |\zeta_1-\zeta_2|$, which further give
\begin{align}
\label{eqn:part-I-brahms}
(\mathrm{I}) \leq \mathbb{E}\left[\left| \frac{Z}{2}\right| \cdot  |\a_{t-1}^\star-\alphastart|\right] \leq |\a_{t-1}^\star-\alphastart|.
\end{align}
It remains to study the second term (II), for which the main difficulty lies in the fact that $b(x) \goto 0$ as $x\goto 0$.
It is easy to see that the limit value of $(\mathrm{II})$ is equal to the limiting value of $(\mathrm{IIa})+(\mathrm{IIb})$, where
\begin{align*}
(\mathrm{IIa}) & \defn - \frac{1}{2b} \mathbb{E}\left\lbrace W \left[ \eta(\Theta + \taustart Z; \alpha_{t-1}^\star \taustart)-\eta(\Theta +\taustart Z; \alphastart \taustart)\right] \textbf{1}(\Theta + aZ + b W\geq \alphastart \taustart
)\right\rbrace;\\
(\mathrm{IIb}) & \defn - \frac{1}{2b}\mathbb{E}\left\lbrace W\left[ \eta(\Theta + \taustart Z; \alpha_{t-1}^\star \taustart)-\eta(\Theta +\taustart Z; \alphastart \taustart)\right] \textbf{1}(\Theta + a Z + b W\leq -\alphastart \taustart
)\right\rbrace.
\end{align*}
The analysis of the two parts are quite similar, so we only only discuss the first part. 
Denote 
\begin{align}
\label{eq:def-mu}
\mu(x; Z, \Theta, W) \defn \left[ \eta(\Theta + \taustart Z; \alpha_{t-1}^\star \taustart)-\eta(\Theta +\taustart Z; \alphastart \taustart)\right] \textbf{1}(\Theta + a(x)Z + b(x) W\geq \alphastart \taustart),
\end{align}
To establish expression~\eqref{eq:diff-derivative}, it suffices to show that
\begin{align}\label{eq:limit-mu}
	\lim_{x\goto 0} \left| \frac{1}{2b(x)}\mathbb{E}[W \mu(x; Z, \Theta, W)] \right| \leq \left|\alphastart-\alpha^\star_{t-1}\right| .
\end{align}

The remaining part of the current section is then devoted to the proof of~\eqref{eq:limit-mu}. 
Towards this, we make the key observation that 
$\mathbb{E}[\mu(x; Z, \Theta, W) \mid W = w]$ is a Lipschitz function in $w.$
In order to see this, first notice that the convoluted density of $p_{aZ}* p_\Theta$ is bounded over all $a\in [\taustart/2, \taustart]$.
Indeed, we can calculate the convolution density by
\begin{align}\label{eq:convolution-density}
(p_{a Z}* p_\Theta)(z) = \int_{-\infty}^{\infty} \frac{1}{a\sqrt{2\pi}}\exp\left\lbrace -\frac{(\Theta-z)^2}{2a^2 }\right\rbrace \mathrm{d} P_{\Theta} \leq \frac{1}{a\sqrt{2\pi}}\int_{-\infty}^{\infty}\mathrm{d} P_{\Theta} = \frac{1}{a\sqrt{2\pi}} \leq \frac{1}{\taustart}, \quad \forall z\in \mathbb{R}.
\end{align}
Next, for any $x\in (0, \taustart^2)$ and $w_1 < w_2$, direct calculations yield 
\begin{align*}
& \ \Big|\mathbb{E}[\mu(x; Z, \Theta, w_1)] - \mathbb{E}[\mu(x; Z, \Theta, w_2)] \Big| \\
= & \ 
\Big| \mathbb{E} \big[ (\eta(\Theta + \taustart Z; \alpha_{t-1}^\star \taustart)-\eta(\Theta +\taustart Z; \alphastart \taustart))\cdot( \textbf{1}(\Theta + a(x)Z + b(x) w_1\geq \alphastart \taustart) - \textbf{1}(\Theta + a(x)Z + b(x) w_2\geq \alphastart \taustart)\big] \Big| 
\\
\leq & \ \taustart |\alphastart-\alpha^\star_{t-1}| \cdot \mathbb{P}\left\lbrace \alphastart \taustart-b(x) w_2\leq \Theta + a(x) Z \leq \alphastart \taustart-b(x) w_1\right\rbrace  \\
 \leq & \  b(x) |\alphastart-\alpha^\star_{t-1}||w_1-w_2|,
\end{align*}
where the last inequality follows from the expression~\eqref{eq:convolution-density} and the fact that $a(x)\in[\taustart/2, \taustar]$ when $0<x<\taustart^2$. 

Now we proceed to control $\mathbb{E}[W \mu(x; Z, \Theta, W)].$
First we can write 
\begin{align*}
\left|\mathbb{E}\left[ W\mu(x; Z, \Theta, W) - W\mu(x; Z, \Theta, 0)\right]\right|  =  &\  \mathbb{E}\left\lbrace W\mathbb{E} \left[\mu(x; Z, \Theta, W) - \mu(x; Z, \Theta, 0)  \vert W\right] \right\rbrace \\
\leq & \  b(x) |\alphastart-\alpha^\star_{t-1}| \mathbb{E}[W^2 ] \\
=&  \  b(x) |\alphastart-\alpha^\star_{t-1}|. 
\end{align*}
Additionally, by symmetry, we know that
\begin{align*}
\mathbb{E}\left[W\mu(x; Z, \Theta, 0)\right] = 0,
\end{align*}
since $W\sim \NORMAL(0, 1)$ is independent from $\mu(x; Z, \Theta, 0)$. 
Putting everything together, we conclude that
\begin{align*}
	\left| \frac{1}{2b(x)}\mathbb{E}[W \mu(x; Z, \Theta, W)]\right| 
	 \leq \frac{1}{2}|\alphastart-\alpha^\star_{t-1}|,
\end{align*}
thus concluding the proof of the inequality~\eqref{eq:limit-mu}.
Similarly, one can control $(\mathrm{IIb})$.
Taking these collectively with \eqref{eqn:part-I-brahms}, 
we validated the inequality~\eqref{eq:diff-derivative}.



\section{Proofs about the risk curve}
\label{sec:proof-risk-curve}

\subsection{Proof of Lemma~\ref{lem:nu-prime} and Lemma~\ref{lem:nu-one-limit}}
\label{sec:proof-derivative-lemmas}

In this section, we present the proofs of Lemma~\ref{lem:nu-prime} and Lemma~\ref{lem:nu-one-limit} by analyzing each term inside the derivative~\eqref{eq:derivative-nu}. 
The proofs are divided into several parts: first in Section~\ref{sec:expression-of-derivative}, we derive the explicit expressions of $F_1(\nu, \delta, \alpha)$, $F_2(\nu, \delta, \alpha)$ and their partial derivatives, which serve as the basis for subsequent analyses. 
Their limiting values are computed in Section~\ref{sec:delta-goto-1} for the case when $\delta \goto 1^-$, which in turn establishes Lemma~\ref{lem:nu-one-limit}. 

When it comes to the case $\delta \goto 0^+$, we start by stating two crucial lemmas concerning the growth of $\alphastar$ and $\nustar$, and  some partial derivatives, as in Lemma~\ref{lem:zero-order-limit} and \ref{lem:one-order-limit}. 
Properties~\eqref{eq:denominator} and \eqref{eq:nu-zero-limit} are direct consequences of these two lemmas. For property~\eqref{eq:numerator}, it turns out that the first-order approximations of both the 
minuend and subtrahend cancel out.
Therefore, we need to resort to the second-order computation of these two terms, which is postponed to Section~\ref{sec:second-order-limit}.
Putting these together completes the proof of Lemma~\ref{lem:nu-prime}. 
Finally, the proofs of auxiliary Lemma~\ref{lem:zero-order-limit} and \ref{lem:one-order-limit} are deferred to the end of this section. 

For notational simplicity, throughout this section, we denote $F_{1}(\nu, \delta, \a)$ as $F_1$, and similarly for $F_2$ and all other partial derivative functions when the values of $(\nu, \delta, \a)$ are clear from the context.

\subsubsection{Expressions of $F_1$ and $F_2$}
\label{sec:expression-of-derivative}

Let us first express $F_1$ and $F_2$
as functions of the density and cumulative density functions of the standard Gaussian distribution (denoted as $\phi(\cdot)$ and $\Phi(\cdot)$ respectively). 
For $F_1$,  it is easily seen that 
\begin{equation}
\label{eq:f1}
F_1(\nu, \delta, \alpha)  = \epsilon \left[ \Phi(-\alpha+\sqrt{\delta}\nu) + \Phi(-\alpha-\sqrt{\delta}\nu)\right] + 2(1-\epsilon) \Phi(-\alpha)-\delta.
\end{equation}
When it comes to $F_2$, invoking the equality~\eqref{eq:auxillary-integral} in Lemma~\ref{lem:formula}, we arrive at the following decomposition
\begin{align}
\label{eqn:f2-decomposition}
 	F_2 = F_{21} + \epsilon \delta^{-1} F_{22} + (1-\epsilon)\delta^{-1} F_{23},
 \end{align} 
where
\begin{equation}\label{eq:f2}
\begin{aligned}
F_{21}(\nu, \delta,\alpha)   \defn & \frac{\nu^2}{M^2}-1 ;\\
F_{22}(\nu, \delta,\alpha) \defn &  \int_{-\infty}^{\infty} ( \eta(\sqrt{\delta}\nu+z,\alpha)-\sqrt{\delta}\nu) ^2\phi(z) \mathrm{d}z \\
=& \ (\b-\a)\phi(\a+\b) + (-\b-\a)\phi(\a-\b)  \\
&+ (\a^2+1-\delta\nu^2)\left[ \Phi(-\a-\b) + \Phi(-\a+\b) \right]  + \delta\nu^2;\\
F_{23}(\nu, \delta,\alpha) \defn & \int_{-\infty}^{\infty} \eta^2(z, \alpha) \phi(z)  \mathrm{d}z = 2\left[ -\a\phi(\a) + (\a^2+1)\Phi(-\a)\right] .                                
\end{aligned}
\end{equation}
Direct computation of the partial derivatives yield the following expressions.

\paragraph{Partial derivatives of $F_1$.} For $F_1$,
we have the following three partial derivatives with respect to $\a,\delta$ and $\nu$ respectively, where  
\begin{equation}\label{eq:f1-derivative}
\begin{aligned}
\nabla_{\alpha} F_1(\nu, \delta, \alpha) & = - \epsilon \left[ \phi(\alpha-\sqrt{\delta}\nu) + \phi(\alpha+\sqrt{\delta}\nu)\right] - 2(1-\epsilon)\phi(\alpha);\\
\nabla_{\delta} F_1(\nu, \delta, \alpha) & = \epsilon \frac{\nu}{2\sqrt{\delta}} \left[ \phi(\alpha-\sqrt{\delta}\nu) - \phi(\alpha+\sqrt{\delta}\nu)\right] - 1;\\
\nabla_{\nu} F_1(\nu, \delta, \alpha) & = \epsilon \sqrt{\delta}\left[ \phi(\alpha-\sqrt{\delta}\nu) - \phi(\alpha+\sqrt{\delta}\nu)\right].
\end{aligned}
\end{equation}
\paragraph{Partial derivatives of $F_2$.} 
With respect to $\a$, from the decomposition~\eqref{eqn:f2-decomposition}, one obtains 
\begin{align}
\label{eqn:sonata}
 	\nabla_{\alpha} F_{2} = \epsilon \delta^{-1}\nabla_{\alpha} F_{22} + (1-\epsilon)\delta^{-1} \nabla_{\alpha} F_{23},
 \end{align} 
where
\begin{align*}
\nabla_{\alpha} F_{22}(\nu, \delta, \alpha) = & -2\left[ \phi(\a+\b) + \phi(\a-\b)\right] + 2\a\left[ \Phi(-\a-\b) + \Phi(-\a+\b)\right] ;\\
\nabla_{\alpha} F_{23}(\nu, \delta, \alpha) = & - 4\left[ \phi(\alpha) - \a\Phi(-\a)\right].
\end{align*}
With respect to $\delta$, a little algebra leads to 
\begin{align}
\label{eqn:F2-der-tmp}
	\notag \nabla_{\delta}F_{2}(\nu, \delta, \alpha) 
	&= -\delta^{-2} \left[ \epsilon F_{22} + (1-\epsilon) F_{23} \right]  + \epsilon \delta^{-1} \nabla_{\delta} F_{22} \\
	&=-\epsilon \delta^{-2}\left[ F_{22} - \delta\nabla_{\delta} F_{22}  \right]  -(1-\epsilon)\delta^{-2} F_{23}. 
\end{align}
We then further evaluate the right-hand side of the above equation. Recognizing that $$\nabla_{\delta} F_{22}(\nu, \delta, \alpha) =  \nu^2\left[ \Phi(\a-\b)-\Phi(-\a-\b)\right],$$ we can guarantee that 
\begin{align*}
	F_{22} - \delta\nabla_{\delta} F_{22} = -(\a-\b)\phi(\a+\b) - (\a+\b)\phi(\a-\b) + (\a^2+1)\left[ \Phi(-\a-\b) + \Phi(-\a+\b) \right].
\end{align*}
Plugging this into the expression~\eqref{eqn:F2-der-tmp} yields the final expression of $\nabla_{\delta}F_{2}.$
Finally, with respect to $\nu$, it is straightforward to verify that 
\begin{align*}
	\nabla_{\nu} F_2(\nu, \delta, \alpha) = 2M^{-2}\nu + \epsilon\delta^{-1} \nabla_{\nu} F_{22},
\end{align*}
where 
\begin{align*}
	\nabla_{\nu} F_{22}(\nu, \delta, \alpha) = 2\nu\delta \left[ \Phi(\a-\b)-\Phi(-\a-\b)\right].
\end{align*}


\subsubsection{The limit of $\delta\goto 1^-$: proof of Lemma~\ref{lem:nu-one-limit}}\label{sec:delta-goto-1}

Equipped with these close-form expressions, we are ready to study the  
the behaviors of $\alphastar$ and $\nustar$ when $\delta \goto 1^-$, and check the limiting orders of all terms in the expression~\eqref{eq:derivative-nu}.

\paragraph{Limiting orders of $\alphastar$ and $\nustar$.} 
We first make the observation that for any given $\nu, \alpha > 0$, $F_1(\nu, \delta, \a)$ is strictly decreasing in $\a$ and increasing in $\nu$. 
Also we note that $\nustar < M$ as $M/\nustar = \taustar > 1.$
Then it is easy to check that 
\begin{align*}
	0 = F_1(\nustar, \delta, \alphastar) < F_1(M, \delta, \alphastar) < F_1(M, \delta, 0) = 1-\delta.
\end{align*}
Taking $\delta \goto 1^-$, one can deduce that $F_1(M, \delta, \alphastar) \goto 0$, which further indicates that $\alphastar \goto 0$. 
Now putting  $\alphastar \goto 0$ together with the expression~\eqref{eq:f2}, we reach 
\begin{align*}
	F_{22}(\nustar, \delta, \alphastar) \goto 1; \quad F_{23}(\nustar, \delta, \alphastar) \goto 1 \quad \Longrightarrow\quad  F_2(\nustar, \delta, \alphastar) \goto \frac{\nustar^2}{M^2}.
\end{align*}
Combining this with the fact that $F_2(\nustar, \delta, \alphastar) =0$ ensures 
$\nustar \goto 0$.

\paragraph{Limiting order of $\nustar'(\delta)$.}
With the limiting values of $\alphastar$ and $\nustar$ in place, we are ready to check the limiting orders of all the terms in expression~\eqref{eq:derivative-nu}. 
First, taking the expressions of $F_1$, $F_{2}$ and their derivatives in Section~\ref{sec:expression-of-derivative} collectively with some algebra, 
we can guarantee that 
\begin{gather*}
\nabla_{\a}F_1 \sim -2 \phi(0) ; \quad \nabla_{\delta} F_1 \sim -1; \quad \nabla_{\nu} F_1 \sim 2 \epsilon \delta \nustar \alphastar \phi(0);\\
\nabla_{\a}F_2 \sim  (-4+2\epsilon)\phi(0) ; \quad \nabla_{\delta} F_2 \sim - 1; \quad \nabla_{\nu} F_2 \sim 2M^{-2}\nustar,
\end{gather*}
when $\delta \goto 1^{-}$ and all the partial derivatives are evaluated at the point $(\nustar, \delta, \alphastar)$.
Substituting these relations into \eqref{eq:derivative-nu} yields 
\begin{gather*}
\nabla_{\delta}F_2\nabla_{\alpha}F_1 - \nabla_{\alpha}F_2\nabla_{\delta}F_1 \sim  -2 (1-\epsilon)\phi(0);\\
\nabla_{\nu}F_2\nabla_{\alpha}F_1 - \nabla_{\alpha}F_2\nabla_{\nu}F_1\sim -4M^{-2} \phi(0) \nustar.
\end{gather*}
When $\epsilon$ is bounded away from $1$ and when $M$ bounded away from $\infty$,  
one can easily see that $\nustar'(\delta) \goto - \infty$ as $\delta \goto 1^-$ (since $0 < \nustar \goto 0$). We thus conclude the proof of Lemma~\ref{lem:nu-one-limit}.

\subsubsection{The limit of $\delta\goto 0^+$: proof of Lemma~\ref{lem:nu-prime}}
\label{sec:second-order-limit}

Before embarking on the main proof, we make note of the following two lemmas concerned with the growth of $\alphastar$ and $\nustar$ and the derivatives of $F_{1}$ and $F_{2}$ when $\delta \goto 0^+$ which shall be used multiple times. 
Their proofs of these lemmas can be found in Section~\ref{sec:zero-order-limit} and \ref{sec:one-order-limit} respectively. 

\begin{lem}\label{lem:zero-order-limit}
	When $\delta \goto 0^+$, the following properties are satisfied  
	\begin{equation}\label{eq:limit-beta}
	\lim_{\delta \goto 0^+} \alphastar = +\infty; 
	\quad \lim_{\delta \goto 0^+} \sqrt{\delta}\alphastar = 0;   
	\quad \lim_{\delta \goto 0^+} \frac{\delta \alphastar}{\phi(\alphastar)} = 2;
	\quad \lim_{\delta \goto 0^+} \nustar = \nu_0.
	\end{equation}
\end{lem}
\begin{lem}\label{lem:one-order-limit}
	When $\delta\goto 0^+$, the limiting order of the partial derivatives of $F_1$ and $F_2$ are characterized as following:
	\begin{gather}
	\label{eq:limit-f1} \lim_{\delta\goto 0^+} \frac{\nabla_{\alpha}F_1}{\phi(\alphastar)} = -2; \quad \lim_{\delta\goto 0^+} \nabla_{\delta} F_1 = -1; \quad \lim_{\delta\goto 0^+} \frac{\nabla_{\nu}F_1}{\phi^2(\alphastar)} = 4\epsilon \nu_0;\\
	\label{eq:limit-f2} 
	\lim_{\delta\goto 0^+}\alphastar \nabla_{\alpha}F_2= -2; 
	\quad \lim_{\delta\goto 0^+} \alphastar\phi(\alphastar)\nabla_{\delta} F_2= -1; 
	\quad  \lim_{\delta\goto 0^+} \nabla_{\nu} F_2 = \frac{2}{\nu_0},
	\end{gather}
	where all the partial derivatives of $F_1$ and $F_2$ are evaluated at the point $(\nustar, \delta, \alphastar)$.
\end{lem}
First note that relation~\eqref{eq:nu-zero-limit} in Lemma~\ref{lem:nu-prime} directly comes from Lemma~\ref{lem:zero-order-limit}; Lemma~\ref{lem:one-order-limit} combined with a little algebra leads to relation \eqref{eq:denominator}. 
Therefore, to prove Lemma~\ref{lem:nu-prime}, it is only left for us to establish the limiting order of \eqref{eq:numerator}, which shall be done as follows.

\paragraph*{Proof of relation~\eqref{eq:numerator}.}
By virtue of Lemma~\ref{lem:one-order-limit}, one immediately notices that the leading terms in $ \nabla_{\delta}F_2\nabla_{\alpha}F_1$ and $\nabla_{\alpha}F_2\nabla_{\delta}F_1$ cancel out with each other. 
Therefore, to characterize the limiting order of $ \nabla_{\delta}F_2\nabla_{\alpha}F_1 - \nabla_{\alpha}F_2\nabla_{\delta}F_1$, it requires us to investigate the second-order terms.
Throughout this part, all the partial derivatives of $F_1$ and $F_2$ are calculated at the point $(\nustar, \delta, \alphastar)$ unless otherwise noted. 

\subsubsection*{Second-order terms of $\nabla_{\a}F_1$ and $\nabla_{\delta} F_1$.}
Let us consider $\nabla_{\a}F_1$ and $\nabla_{\delta} F_1$. 
We first make the observation that 
\begin{align}
\label{eqn:F1-alpha-second}
   \notag \frac{\nabla_{\a}F_1+2\phi(\alphastar)}{\phi(\alphastar)} &= -\epsilon\left[e^{-\delta\nustar^2/2}\left( e^{\alphastar\bstar} + e^{-\alphastar\bstar}\right)-2  \right] \\
    &\stackrel{(\mathrm{i})}{\sim} -\epsilon \alphastar^2 \delta\nustar^2 
    \stackrel{(\mathrm{ii})}{\sim} -2\epsilon\nu_0^2\alphastar\phi(\alphastar),
    \end{align}
where $(\mathrm{ii})$ follows as a direct consequence of \eqref{eq:limit-beta}.
Regarding the relation $(\mathrm{i})$, by virtue of the Taylor expansion, we obtain
\begin{align*}
e^{\alphastar\bstar} + e^{-\alphastar\bstar} = 2 + \delta \nustar^2  \alphastar^2 + o(\delta \nustar^2  \alphastar^2), \quad e^{-\delta\nustar^2/2} = 1-\delta\nustar^2/2 + o(\delta\nustar^2).
\end{align*}
As already shown in Lemma~\ref{lem:zero-order-limit}, $\alphastar \goto \infty$, and therefore, the $\delta\nustar^2/2$ term is of smaller order compared with $\delta \nustar^2  \alphastar^2$, which in turn justifies step $(\mathrm{i})$. 
Here, recall that these limits are taken with respect to $\delta\goto 0^+$.

Moreover, direct computation gives 
\begin{align}
\label{eqn:scherzo}
\frac{\phi(\alphastar-\sqrt{\delta}\nustar) - \phi(\alphastar+\sqrt{\delta}\nustar)}{\sqrt{\delta}\alphastar\nustar\phi(\alphastar)} = \exp\left\lbrace-\delta\nustar^2/2 \right\rbrace \frac{\exp\{\sqrt{\delta}\alphastar \nustar\}  - \exp\{-\sqrt{\delta}\alphastar \nustar\}}{\sqrt{\delta}\alphastar\nustar} \sim 2, \quad \delta \goto 0^+,
\end{align}
where the the last relation is given by L'Hôpital's rule, combined with the facts $\sqrt{\delta}\alphastar \goto 0$ and $\nustar \goto \nu_0$ as $\delta \goto 0^+$. 
Taking this collective with Lemma~\ref{lem:zero-order-limit},  
one arrives at
\begin{align}
\label{eqn:F1-delta-second}
    \nabla_{\delta}F_1 + 1 = \epsilon \frac{\nustar}{2\sqrt{\delta}} \left[ \phi(\alphastar-\sqrt{\delta}\nustar) - \phi(\alphastar+\sqrt{\delta}\nustar)\right] \sim \epsilon \nu_0^2\alphastar\phi(\alphastar).
\end{align}

\subsubsection*{Second-order term of $\nabla_{\delta}F_2$.}
Moving on to the quantity $\nabla_{\delta} F_2$, we can rearrange terms 
to derive the following decomposition 
    \begin{equation}\label{eq:partial-delta-f2}
    \frac{\nabla_{\delta}F_2 + (\alphastar\phi(\alphastar))^{-1}}{(\alphastar\phi(\alphastar))^{-1}}  
    = 
    \underbrace{\frac{\nabla_{\delta}F_2 + \delta^{-2}F_{23}}{(\alphastar\phi(\alphastar))^{-1}}}_{=: \Delta_1} - \underbrace{\frac{\delta^{-2}F_{23} - (\alphastar\phi(\alphastar))^{-1}}{(\alphastar\phi(\alphastar))^{-1}}}_{=: \Delta_2}.
    \end{equation}
With this decomposition in mind, we proceed to control the two terms above separately.
\begin{itemize}
\item {\bf Step 1: Bounding the term $\Delta_{1}$.} 
Armed with the expressions of $\nabla_{\delta}F_2$ and $F_{23}$ in closed-form, 
one can re-arrange terms and obtain   
\begin{align}
\label{eq:taylor-expansion}
\nabla_{\delta}F_2 + \delta^{-2}F_{23} = -\epsilon\delta^{-2}\left[ f(s) + f(-s) - 2f(0)\right] ,
\end{align}
where $f(s) \defn -(\alphastar+s)\phi(\alphastar-s) + (\alphastar^2+1)\Phi(-\alphastar+s)$, for $s = \bstar$.
To facilitate analysis of the expression~\eqref{eq:taylor-expansion}, we make note of the following two facts:
\begin{itemize}
	\item For every $k\geq 1$, the rescaled derivative $f^{(k)}(s) /\phi(\alphastar-s)$ is a polynomial of $s$ and $\alphastar$; 
	\item Since $\alphastar \goto \infty$ and $\nustar \goto \nu_0$ as $\delta \goto 0^+$, $s = \bstar$ is therefore negligible compared to any polynomial of $\alphastar$. 
\end{itemize}

Leveraging the aforementioned results, one can see that in the Taylor expansion of $f$ around $s=0$, the non-zero term with the lowest order of $s$ is the dominant term. 
By further calculating $f^{(2)}(0) = 0$ and $f^{(4)}(0) = 6\alphastar\phi(\alphastar)$, we see that
\begin{align*}
\nabla_{\delta}F_2 + \delta^{-2}F_{23} \sim -\epsilon\delta^{-2} \frac{2}{4!}6\alphastar\phi(\alphastar) (\bstar)^4 \sim -\frac{\epsilon\nu_0^4}{2} \alphastar\phi(\alphastar).
\end{align*}

\item {\bf Step 2: Bounding the term $\Delta_{2}$.}
Recalling our definition for function $F_{23}$ (cf.~\eqref{eq:f2}), we can express $\Delta_{2}$ as follows 
\begin{align}\label{eq:decomp-delta-2}
    	\frac{\delta^{-2}F_{23} -  (\alphastar\phi(\alphastar))^{-1}}{(\alphastar\phi(\alphastar))^{-1}} =(1+R_1)(1+R_2) - 1, 
\end{align}
where 
\begin{align*}
R_1 \defn \frac{\alphastar^3\left[-\alphastar\phi(\alphastar)+(\alphastar^2+1)\Phi(-\alphastar) \right] }{2\phi(\alphastar)} -1 ;\quad 
R_2 \defn \left(\frac{2\phi(\alphastar)}{\alphastar\delta} \right)^2 - 1.
\end{align*}

Let us consider each term separately. 
Firstly, directly invoking expression~\eqref{eq:auxillary-moment-3} from Lemma~\ref{lem:formula} suggests $R_1 = o(1)$, as $\delta \goto 0^+$. To further study the limiting order of $R_1$, we obtain 
\begin{align*}
\lim_{\alphastar \goto \infty}\alphastar^2 R_1 & =  \lim_{\alphastar \goto \infty} \frac{\alphastar^5\left[-\alphastar\phi(\alphastar)+(\alphastar^2+1)\Phi(-\alphastar) \right] - 2 \alphastar^2 \phi(\alphastar)}{2\phi(\alphastar)} \\
\small{\text{(L'Hôpital's rule)}} & =  \lim_{\alphastar \goto \infty} \frac{(2 \alphastar^2- 7\alphastar^4)\phi(\alphastar) + (7\alphastar^5+ 5 \alphastar^3)\Phi(-\alphastar)  }{-2 \phi(\alphastar)} + 2\\
& =  \frac{7}{2} \lim_{\alphastar \goto \infty} \frac{\alphastar^3\left[ \alphastar\phi(\alphastar) - (\alphastar^2+  1)\Phi(-\alphastar) \right] }{\phi(\alphastar)} - \lim_{\alphastar \goto \infty} \frac{\alphastar^2 \left[ \phi(\alphastar) - \alphastar\Phi(-\alphastar)\right] }{ \phi(\alphastar)}+ 2 \\
& =  -6,
\end{align*}
where the last step uses relations \eqref{eq:auxillary-moment-2}  and  \eqref{eq:auxillary-moment-3} from Lemma~\ref{lem:formula}. 
This establishes the limiting order $R_1 \sim -6\alphastar^{-2}$ as $\delta \goto 0^+$.

Turning to the term $R_2$, one can easily conclude from Lemma~\ref{lem:zero-order-limit} that $R_2=o(1)$ as $\delta\goto 0^+$. To further pin down the limiting order of $R_2$, we recall that 
\begin{align*}
\epsilon \left[ \Phi(-\alphastar+\sqrt{\delta}\nustar) + \Phi(-\alphastar-\sqrt{\delta}\nustar)\right] + 2(1-\epsilon) \Phi(-\alphastar)-\delta = 0
\end{align*}
as $F_1(\nustar, \delta, \alphastar)  = 0$. 
These allow us to decompose 
\begin{align}
\label{eqn:rondo}
\alphastar^2 \left[ \frac{\delta\alphastar}{2\phi(\alphastar)} - 1 \right] 
= \frac{\epsilon \alphastar^3 [\Phi(-\alphastar + \bstar) + \Phi(-\alphastar- \bstar) - 2 \Phi(-\alphastar) ]}{2\phi(\alphastar)}  
+ \frac{\alphastar^2 \left[\alphastar\Phi(-\alphastar)- \phi(\alphastar)\right]}{\phi(\alphastar)}. 
\end{align}
By virtue of the Taylor expansion, we can use similar reasoning as for the expression~\eqref{eq:taylor-expansion} to arrive at 
\begin{align*}
\Phi(-\alphastar + \bstar) + \Phi(-\alphastar- \bstar) - 2 \Phi(-\alphastar)
\sim  \alphastar \delta \nu_0^2\phi(\alphastar),
\end{align*}
as $\delta \goto 0^+$. 
Taking this together with the fact that $\delta\alphastar\sim2\phi(\alphastar)$ (cf.~\eqref{eq:limit-beta}) reveals that:
the first term in the decomposition~\eqref{eqn:rondo} scales as $\epsilon\nu_0^2\alphastar^2\phi(\alphastar)$.
In addition, from the equation~\eqref{eq:auxillary-moment-2}, 
the second term in the decomposition~\eqref{eqn:rondo} scales as $-1$ --- which is therefore the dominant term as $\alphastar \goto \infty. $ 
We can therefore conclude that 
\begin{align*}
\alphastar^2 \left[ \frac{\delta\alphastar}{2\phi(\alphastar)} - 1 \right] \sim -1 \quad \Longrightarrow \quad \frac{2\phi(\alphastar)}{\alphastar\delta} = 1 + \alphastar^{-2} + o(\alphastar^{-2}).
\end{align*}
As a consequence, we obtain $R_2\sim 2\alphastar^{-2}$. 

Substituting the limit scalings of $R_1$ and $R_2$ into the decomposition~\eqref{eq:decomp-delta-2} yields
\begin{align}
\label{eqn:mozart}
	\frac{\delta^{-2}F_{23} -  (\alphastar\phi(\alphastar))^{-1}}{(\alphastar\phi(\alphastar))^{-1}}
	\sim -4 \alphastar^{-2}.
\end{align}

\end{itemize}
Consequently, the above results on $\Delta_{1}$ and $\Delta_{2}$ taken collectively with the expression \eqref{eq:partial-delta-f2} lead to
\begin{align}
\label{eqn:F2-delta-second}
	\nabla_{\delta}F_2 = -(\alphastar\phi(\alphastar))^{-1}\left[ 1- 4\alphastar^{-2} + o(\alphastar^{-2})\right].
\end{align}

\subsubsection*{Second-order term of $\nabla_{\a}F_2$.}
We are only left to establish the order of $\nabla_{\a}F_2$. 
To this end, re-arranging  terms in the expression of $\nabla_{\a}F_2$ leads to
	\begin{align*}
		\frac{\nabla_{\a}F_2+2\a^{-1}}{\a^{-1}} 
		= \underbrace{\frac{\nabla_{\a}F_2 -\delta^{-1} \nabla_{\a}F_{23} }{\a^{-1}}}_{=:T_1} + 
		\underbrace{\frac{\delta^{-1} \nabla_{\a}F_{23}  + 2\a^{-1}}{\a^{-1}}}_{=: T_2}.
	\end{align*}
Therefore, it suffices to analyze the limiting order of the two terms $T_1$ and $T_2$ separately, which shall be done as follows. 
\begin{itemize}
\item {\bf Step 1: Bounding the term $T_1$.}
    We characterize the leading term of $T_1$ by examining its Taylor expansion with similar argument as in \eqref{eq:taylor-expansion}. Specifically, setting $s\defn \bstar$, we have
    \begin{align*}
     	\nabla_{\a}F_2 -\delta^{-1} \nabla_{\a}F_{23} = \epsilon\delta^{-1}[f_2(s) + f_2(-s)-2f_2(0)],
     \end{align*}
    where $f_2(s) := -2\phi(\alphastar+s)+2\alphastar\Phi(-\alphastar-s).$
    It is straightforward to verify that $f_2^{(2)}(0) = 2\phi(\alphastar)$, 
    and it follows that
    \begin{align*}
    \nabla_{\a}F_2 -\delta^{-1} \nabla_{\a}F_{23} \sim \epsilon\delta^{-1} 2\phi(\alphastar) (\bstar)^2 \sim 2\epsilon \nu_0^2 \phi(\alphastar).
    \end{align*}

\item {\bf Step 2: Bounding the term $T_2$.}
    To calculate the limiting order of $T_2$, we establish the decomposition
    \begin{align*}
    \frac{\delta^{-1} \nabla_{\a}F_{23} + 2\alphastar^{-1}}{\alphastar^{-1}}  = 2 \frac{2\phi(\alphastar)}{\delta\alphastar} \left[ -\frac{\alphastar\left[\alphastar\phi(\alphastar) -(\alphastar^2+1)\Phi(-\alphastar) \right] }{\phi(\alphastar)} + \frac{\phi(\alphastar)-\alphastar\Phi(-\alphastar)}{\phi(\alphastar)}\right] .
    \end{align*}
    As direct consequences of the expressions~\eqref{eq:auxillary-moment-2},  \eqref{eq:auxillary-moment-3} and \eqref{eq:limit-beta}, one can easily see that
    \begin{align*}
    \frac{\alphastar\left[\alphastar\phi(\alphastar) -(\alphastar^2+1)\Phi(-\alphastar) \right] }{\phi(\alphastar)} \sim -2 \alphastar^{-2}; \quad \frac{\phi(\alphastar)-\alphastar\Phi(-\alphastar)}{\phi(\alphastar)}\sim \alphastar^{-2}; \quad 
    \frac{2\phi(\alphastar)}{\delta\alphastar} \sim 1,
    \end{align*}
    as $\delta \goto 0^+$. Plugging the above relations into the decomposition of $T_2$ yields $T_2 \sim 6\alphastar^{-2}$.
 \end{itemize}  
In view of the results above on $T_1$ and $T_2$, we can conclude that
\begin{align}
\label{eqn:F2-alpha-second}
\frac{\nabla_{\a}F_{2} + 2\alphastar^{-1}}{\alphastar^{-1}} \sim 6\alphastar^{-2}.
\end{align}

\paragraph{Putting all this together.}
Thus far, we have established the limiting order of $\nabla_{\alpha}F_1$, $\nabla_{\delta}F_1$, $\nabla_{\delta}F_2$ and $\nabla_{\alpha}F_2$. Combining all the above pieces together, it is easy to justify that
\begin{gather*}
\nabla_{\delta}F_2\nabla_{\alpha}F_1  = 2\alphastar^{-1}\left[ 1- 4\alphastar^{-2} + o(\alphastar^{-2})\right];\\
\nabla_{\alpha}F_2\nabla_{\delta}F_1 = 2\alphastar^{-1}\left[ 1- 3\alphastar^{-2} + o(\alphastar^{-2})\right].
\end{gather*}
It immediately suggests that $\nabla_{\delta}F_2\nabla_{\alpha}F_1 - \nabla_{\alpha}F_2\nabla_{\delta}F_1 = - 2\alphastar^{-3} + o(\alphastar^{-3})
$, which proves the relation~\eqref{eq:numerator},
thus completing the proof of Lemma~\ref{lem:nu-prime}.


\subsection{Limiting orders of $\alphastar$ and $\nustar$ when $\delta \goto 0^+$: proof of Lemma~\ref{lem:zero-order-limit}}
\label{sec:zero-order-limit}

In this section, we present the proofs of the four relations in Lemma~\ref{lem:zero-order-limit} in the sequel.

\paragraph{First claim in \eqref{eq:limit-beta}.}
Recall that $F_1$ is defined as 
\begin{align*}
	F_1(\nu, \delta, \alpha) & \defn \epsilon\mathbb{P}\left( \big|\nu \sqrt{\delta}+ Z\big|> \alpha \right) + (1-\epsilon) \mathbb{P}\left( \left| Z\right|> \alpha \right)-\delta.
\end{align*}
Since the first two terms are non-negative, when $\delta\goto 0^+$, setting $F_1(\nustar, \delta, \alphastar) = 0$ leads to  
    \begin{align*}
	\mathbb{P}\left( \big|\bstar+ Z\big|> \alphastar \right) \goto 0; \quad \mathbb{P}\left( \left| Z\right|> \alphastar \right) \goto 0.
	\end{align*}
	From the second expression, it can be immediately concluded that $\lim_{\delta \goto 0^+} \alphastar = \infty$.
	
	\paragraph{Second claim in \eqref{eq:limit-beta}.}
	In order to study the limiting behavior of $\sqrt{\delta}\alphastar$, we first make the observations that $0 < \nustar \leq M$ as $M/\nustar = \taustar \geq 1$ (see \eqref{eq:fix-1}) and $\Phi$ is a monotonically increasing function such that
	\begin{align*}
	 	\Phi(-\alphastar-\bstar) < \Phi(-\alphastar) < \Phi(-\alphastar+\bstar).
	 \end{align*}
	Consequently, $F_1(\nustar,\delta,\alphastar)$ can be upper bounded as 
	\begin{align*}
	0 = F_1(\nustar,\delta,\alphastar) 
	&= \epsilon \left[ \Phi(-\alphastar+\sqrt{\delta}\nustar) + \Phi(-\alphastar-\sqrt{\delta}\nustar)\right] + 2(1-\epsilon) \Phi(-\alphastar)-\delta\\
	&\leq 2 \Phi(-\alphastar+\bstar) - \delta,
	\end{align*}
	which in turn suggests
	\begin{align*}
		\Phi(-\alphastar+\sqrt{\delta}\nustar) \geq \delta/2. 
	\end{align*}
	Apply Lemma~\ref{lem:formula} with a little algebra to yield 
	\begin{align*}
	2 \geq \frac{\delta}{\Phi(-\alphastar+\bstar)} \overset{(\mathrm{i})}{\sim} \frac{(\alphastar-\bstar) \delta}{\phi(-\alphastar+\bstar)}
	\overset{(\mathrm{ii})}{\sim} 
	\frac{1}{\sqrt{2\pi}}\alphastar \delta e^{\alphastar^2/2},
	\end{align*}
	where $(\mathrm{i})$ is a direct consequence of the expression~\eqref{eq:auxillary-moment-1}, and $(\mathrm{ii})$ follows from the observations that $\alphastar\goto\infty$ and $\bstar \leq M\sqrt{\delta} \goto 0$ as $\delta \goto 0^+$. 
	It therefore reveals that $\alphastar^2\delta = o(1)$, due to the fact that $\alphastar\goto\infty$. 
	Thus, we complete the proof of the second claim.
	
	\paragraph{Third claim in \eqref{eq:limit-beta}.}
	With the limiting values of $\alphastar$ and $\alphastar\sqrt{\delta}$ in place, we are now ready to characterize the limiting orders of $\Phi(-\alphastar - \bstar)$ and $\Phi(-\alphastar + \bstar)$. 
	To this end, we first recognize the following relation
	\begin{align}
	\label{eqn:sandwich-phi}
	\Phi(-\alphastar - \bstar) \sim \Phi(-\alphastar + \bstar) \sim \Phi(-\alphastar)  \sim \frac{\phi(\alphastar)}{\alphastar},
	\end{align}
	when $\delta \to 0^{+}$, 
	followed by the property $\alphastar\bstar \leq M \alphastar\sqrt{\delta} = o(1)$. 
	It immediately follows that
	\begin{align*}
	\lim_{\delta\goto 0^+} \frac{\delta\alphastar}{\phi(\alphastar)} 
	&= \lim_{\delta\goto 0^+} \frac{\left[ \epsilon \Phi(-\alphastar + \bstar) + \epsilon \Phi(-\alphastar - \bstar)+ 2(1-\epsilon) \Phi(-\alphastar)\right] \alphastar}{\phi(\alphastar)} \\
	&= \epsilon +\epsilon + 2(1-\epsilon) = 2,
	\end{align*}
	where the first equality comes from property $F_1(\nustar, \delta, \alphastar) = 0$. 
	We thus finish the proof of the third claim.
	
	\paragraph{Fourth claim in \eqref{eq:limit-beta}.}
	In order to understand the limit of $\nustar$ as $\delta\goto 0^+$, 
	we resort to the decomposition~\eqref{eqn:f2-decomposition} of $F_2$,  
	where
	\begin{align}
	\label{eqn:tmp-f2}
	 	F_2 = \frac{\nu^2}{M^2} - 1 + \epsilon \delta^{-1} F_{22} + (1-\epsilon)\delta^{-1} F_{23}.
	 \end{align}
	We claim that it satisfies 
	\begin{align}
	\label{eqn:nustar-limit-comp}
	0 = F_2(\nustar, \delta, \alphastar) = \frac{\nustar^2}{M^2} - 1 + \epsilon \nustar^2 + o(1), \quad \text{as } ~\delta \goto 0^+.
	\end{align}
	Taking the above result as given for the moment, one can conclude that $\lim_{\delta\goto 0^+} \nustar = \nu_0$ as desired, where $\nu_0$ is defined as $\nu_0 \defn M/\tau_0$ and  $\tau_0^2 = 1 + \epsilon M^2$ (cf.~the expression~\eqref{eq:zero-limit}).

	Now it remains to establish the crucial relation~\eqref{eqn:nustar-limit-comp}. 
	To this end, the idea is to characterize the limiting order of each term in the expression~\eqref{eqn:tmp-f2} in terms of $\nustar$, as $\delta \goto 0^+$.
	Let us start with the quantity $\delta^{-1}F_{23}$.  
	Firstly, invoking the equation~\eqref{eq:auxillary-moment-3} from Lemma~\ref{lem:formula}, we know that $F_{23}\sim 4\alphastar^{-3}\phi(\alphastar)$.   
	As a result, one can write 
	\begin{align*}
		\delta^{-1}F_{23}\sim 4\delta^{-1} \alphastar^{-3}\phi(\alphastar) \sim 2\alphastar^{-2},
	\end{align*}
	where the last relation uses the third claim in \eqref{eq:limit-beta} that we just proved, namely, 
	$\lim_{\delta \goto 0^+} \frac{\delta \alphastar}{\phi(\alphastar)} = 2$.
	Since $\alphastar\goto +\infty$ as $\delta \goto 0^+$, we can ensure that 
	$\delta^{-1}F_{23} = o(1)$ as $\delta\goto 0^+$.

	When it comes to the terms in $\delta^{-1}F_{22}$, we find it useful to  define the following function
	\begin{equation}\label{eq:defn-g}
	g(x) \defn \phi(x) - x\Phi(-x). 
	\end{equation}
	One shall then conclude from the expression~\eqref{eq:auxillary-moment-2} that $g(x)\sim x^{-2}\phi(x)$ when $x\goto \infty$. 
	With this piece of notation, we can rewrite $F_{22}$ as 
	\begin{align*}
	F_{22}(\nu, \delta, \alpha) = -(\a-\b)g(\a+\b) - (\a+\b)g(\a-\b) + \left[ \Phi(-\a-\b) + \Phi(-\a+\b) \right] + \delta \nu^2.
	\end{align*}
	Let us examine each term on the right-hand side above respectively. 
	For the terms involving $g$, again applying 
	$\lim_{\delta \goto 0^+} \frac{\delta \alphastar}{\phi(\alphastar)} = 2$ gives 
	\begin{align*}
	\frac{(\alphastar-\bstar)g(\alphastar+\bstar)}{\delta} \sim \frac{\alphastar^{-1}\phi(\alphastar)}{\delta} \sim \frac{1}{2}; \quad \frac{(\alphastar+\bstar)g(\alphastar-\bstar)}{\delta} \sim \frac{\alphastar^{-1}\phi(\alphastar)}{\delta} \sim \frac{1}{2}.
	\end{align*}
	In addition, for the terms involving the function $\Phi$, by the sandwich relation \eqref{eqn:sandwich-phi}, we arrive at 
	\begin{align*}
		\Phi(-\a-\b) + \Phi(-\a+\b) \sim \delta.
	\end{align*}
	Plugging the above relations into $F_{22}$ leads to $\delta^{-1}F_{22}(\nustar, \delta, \alphastar) = o(1) + \nustar^2$, as $\delta \goto 0^+$.
	
	Combining the conclusions about $\delta^{-1} F_{22}$ and $\delta^{-1} F_{23}$, we successfully establish the equation~\eqref{eqn:nustar-limit-comp}, thus finishing the proof of the fourth claim.


\subsection{Limiting orders of the partial derivatives: proof of Lemma~\ref{lem:one-order-limit}}\label{sec:one-order-limit}

We are now positioned to study the limiting orders of the partial derivatives of $F_1$ and $F_2$ stated in Section~\ref{sec:expression-of-derivative}. 
This will be accomplished by taking advantage of Lemma~\ref{lem:zero-order-limit}.

	\paragraph{Properties concerning $F_1$ in \eqref{eq:limit-f1}.}
	Recognizing the fact that $\sqrt{\delta}\alphastar \goto 0$ and $\nustar \goto \nu_0$ as $\delta \goto 0^+$ (as in Lemma~\ref{lem:zero-order-limit}), we obtain that $\phi(\alphastar-\sqrt{\delta}\nustar) \sim\phi(\alphastar+\sqrt{\delta}\nustar) \sim \phi(\alphastar)$. 
	Combining this with  the
	explicit expressions of derivatives of $F_{1}$ (cf.~\eqref{eq:f1-derivative}),
	we immediately obtain $\nabla_{\alpha} F_1 \sim -2\phi(\alphastar)$.  

	Further, taking Lemma~\ref{lem:zero-order-limit}
	collectively with the relation~\eqref{eqn:scherzo} where 
	\begin{align*}
	\frac{\phi(\alphastar-\sqrt{\delta}\nustar) - \phi(\alphastar+\sqrt{\delta}\nustar)}{\sqrt{\delta}\alphastar\phi(\alphastar)\nustar} \sim 2, \quad \delta \goto 0^+,
	\end{align*}
	we can directly see that 
	\begin{gather*}
	\nabla_{\nu} F_1  \sim 2\epsilon \nustar \delta \alphastar \phi(\alphastar) \sim 4\epsilon \nu_0\phi^2(\alphastar) \\ 
	\frac{\epsilon \nustar}{2\sqrt{\delta}} \left[ \phi(\alphastar-\sqrt{\delta}\nustar) - \phi(\alphastar+\sqrt{\delta}\nustar)\right] \sim \epsilon \nu_0^2\alphastar\phi(\alphastar) = o(1) 
	\end{gather*}
	as $\delta \goto 0^+$. From the second relation, one can conclude that $\nabla_{\delta} F_1\sim -1$, and thus complete the proof of expression~\eqref{eq:limit-f1}.
	
	\paragraph{Properties concerning $F_2$ in \eqref{eq:limit-f2}.} 
	Let us turn to the analysis of the partial derivatives related to $F_2$. In what follows, we shall check the limiting orders of $\nabla_{\a} F_2$, $\nabla_{\delta} F_2$ and $\nabla_{\nu} F_2$ respectively, with the assistance of Lemma~\ref{lem:formula} and Lemma~\ref{lem:zero-order-limit}.

	\begin{itemize}
    \item \textbf{Limiting order of $\nabla_{\a} F_2$.} 
    It is useful to recall the decomposition as in expression~\eqref{eqn:sonata}, where  
    \begin{align*}
 	\nabla_{\alpha} F_{2} = \epsilon \delta^{-1}\nabla_{\alpha} F_{22} + (1-\epsilon)\delta^{-1} \nabla_{\alpha} F_{23}.
 	\end{align*} 
 	Now we are only left to analyze each term on the right-hand side of the above relation separately. 
	Firstly, the relation~\eqref{eq:auxillary-moment-2} directly yields $\nabla_{\alpha}F_{23} \sim -4\alphastar^{-2}\phi(\alphastar)$; further recognizing that $\delta\sim 2\alphastar^{-1}\phi(\alphastar)$ (cf.~\eqref{eq:auxillary-moment-1}), we have $\delta^{-1}\nabla_{\alpha}F_{23}\sim -2\alphastar^{-1}$. 
	
	To analyze the quantity $\delta^{-1}\nabla_{\alpha}F_{22}$, we again invoke the definition of function $g$ in \eqref{eq:defn-g} to obtain the decomposition
	\begin{align*} 
	\nabla_{\alpha} F_{22}(\nu, \delta, \a) = -2 g(\a+\b) - 2g(\a-\b) + 2\b\left[ \Phi(-\a-\b)- \Phi(-\a+\b)\right] .
	\end{align*}
	Taking this together with the facts that $g(x)\sim x^{-2}\phi(x)$ when $x\goto \infty$, 
	and $\phi(\alphastar+\bstar)\sim \phi(\alphastar-\bstar) \sim \phi(\alphastar)$, leads to 
	\begin{align*}
		g(\alphastar+\bstar)\sim g(\alphastar-\bstar) \sim \alphastar^{-2}\phi(\alphastar)
		\qquad
		\text{ as }~\delta \goto 0^+.
	\end{align*} 
	Finally, we claim that the terms involving $\Phi$ are negligible. 
	This can be shown by combining the results in Lemma~\ref{lem:formula} with the equality \eqref{eqn:scherzo}. 
	Specifically, one has  
	\begin{align*} 	
	2\bstar \left[\Phi(-\alphastar+\bstar) - \Phi(-\alphastar-\bstar)\right]\sim 4 \phi(\alphastar) = o(1),
	\end{align*}
	where the last step uses $\phi(\alphastar) = o(1)$ as $\delta \goto 0^+$.
	Putting these pieces together gives  
	\begin{align*} 
	\delta^{-1} \nabla_{\alpha} F_{22}\sim -4 \delta^{-1}\alphastar^{-2}\phi(\alphastar) \sim -2\alphastar^{-1}.
	\end{align*}
	As a result, one immediately realizes that $\nabla_{\alpha}F_{2} \sim -2\alphastar^{-1}$.

    \item \textbf{Limiting order of $\nabla_{\delta} F_2$.} 
    First recall that $\nabla_{\delta} F_2 = -\epsilon \delta^{-2}\left[ F_{22} - \delta\nabla_{\delta} F_{22}  \right]  -(1-\epsilon)\delta^{-2} F_{23}.$
	As a direct consequence of the relation \eqref{eq:auxillary-moment-3}, 
	one has $F_{23}\sim 4\alphastar^{-3}\phi(\alphastar)$. 
    When it comes to the first term, re-arranging terms in the expression of $F_{22}$ leads to 
	\begin{align*} 
    F_{22} - \delta\nabla_{\delta} F_{22} = & -(\alphastar-\bstar)\phi(\alphastar+\bstar) - (\alphastar+\bstar)\phi(\alphastar-\bstar) \\
    & + (\alphastar^2+1)\left[ \Phi(-\alphastar-\bstar) + \Phi(-\alphastar+\bstar) \right].
    \end{align*}
	Applying \eqref{eq:auxillary-moment-3} again at $\alphastar+\bstar$ and $\alphastar-\bstar$ gives  
    \begin{gather*}
    (\alphastar-\bstar)\phi(\alphastar+\bstar) + (\alphastar^2+1)\Phi(-\alphastar-\bstar) \sim -2 \alphastar^{-3}\phi(\alphastar),\\
    (\alphastar+\bstar)\phi(\alphastar-\bstar) + (\alphastar^2+1)\Phi(-\alphastar+\bstar) \sim -2 \alphastar^{-3}\phi(\alphastar), 
    \end{gather*}
    which together with  $\delta\sim 2\alphastar^{-1}\phi(\alphastar)$ (cf.~\eqref{eq:auxillary-moment-1}) directly validate $\lim_{\delta\goto 0^+} \alphastar\phi(\alphastar)\nabla_{\delta} F_2= -1.$

    \item \textbf{Limiting order of $\nabla_{\nu} F_2$.} 
    It is helpful to recall the expression $\nabla_{\nu} F_2 = 2M^{-2}\nustar+ \epsilon\delta^{-1} \nabla_{\nu} F_{22}$. With Lemma~\ref{lem:zero-order-limit} in place, one has  
	\begin{align*} 
    \delta^{-1} \nabla_{\nu} F_{22} = 2\nustar \left[ \Phi(\alphastar-\bstar)-\Phi(-\alphastar-\bstar)\right] \sim 2\nu_0.
    \end{align*}
    As a result,  we have $\nabla_{\nu} F_2\sim  2(M^{-2}+ \epsilon)\nu_0 = 2/\nu_0$, where the last equality follows from the definition of $\nu_0$ as $\nu_0 \defn M/\tau_0$ with 
     $\tau_0^2 = 1 + \epsilon M^{2}$ according to the expression~\eqref{eq:zero-limit}.
    \end{itemize}

Thus, we complete the proof of Lemma~\ref{lem:one-order-limit}.

\subsection{Proof of Lemma~\ref{lem:epsilon}}
\label{sec:pf-lemma-epsilon}
		
We divide this proof into two parts. 
In the first part, we characterize the limiting values of $\alphastar$ and $\nustar/M$ as $\epsilon \goto 0$ to establish \eqref{eq:epsilon-limit-2}; 
in the second part, we proceed by  calculating the limiting orders of each quantity in the expression \eqref{eq:derivative-nu} for $\nustar'(\delta_0)/M$.

\paragraph{Step 1: limit of $\alphastar$ and $\nustar/M$ as $\epsilon \goto 0$.}
\begin{itemize}
\item To establish the first statement of the expression~\eqref{eq:epsilon-limit-2}, we make the observation that: having $F_1(\nustar, \delta_0, \alphastar)=0$ yields 
\begin{align*}
\epsilon\left[ \Phi(-\alphastar+\sqrt{\delta_0}\nustar) + \Phi(-\alphastar-\sqrt{\delta_0}\nustar)-2\Phi(-\alphastar)\right] = \delta_0 -2\Phi(-\alphastar),
\end{align*}
by recalling the expression of $F_1$ in \eqref{eq:f1}. 
It then immediately follows that
\begin{align*}
\left|  \delta_0 -2\Phi(-\alphastar)\right| \leq 4\epsilon.
\end{align*}

When $\epsilon \goto 0$, we know that $2\Phi(-\alphastar) \goto \delta_0$, or equivalently, $\alphastar \goto -\Phi^{-1}(\delta_0/2)=:\a_0$.

\item In the hope of proving the second statement of the expression~\eqref{eq:epsilon-limit-2}, we find it helpful to recall the decomposition $F_2 = F_{21} + \epsilon \delta^{-1} F_{22} + (1-\epsilon)\delta^{-1} F_{23}.$
From $F_2(\nustar, \delta_0,\alphastar) = 0$, the following relation holds true
\begin{align}
\label{eq:nu-limit-epsilon}
\left| \frac{\nustar^2}{M^2}-1 + \delta_0^{-1} F_{23} \right| \leq \epsilon \delta_0^{-1} \left| F_{22} - F_{23}\right| .
\end{align}
Below we shall demonstrate the fact that $|F_{22} - F_{23}| $ is upper bounded by some constant that only depends on $\delta_0$, which means when $\epsilon \goto 0$, the right-hand side of the inequality~\eqref{eq:nu-limit-epsilon} vanishes to zero. In other words, one can conclude 
\begin{align*}
\frac{\nustar}{M} \goto \sqrt{1 -\delta_0^{-1}F_{23}} \goto \sqrt{1 - 2\delta_0^{-1}[-\a_0\phi(\a_0) + (\a_0^2+1)\Phi(-\a_0)]},
\end{align*}
where the last step follows since
$F_{23}(\nustar, \delta,\alphastar) = 2[ -\alphastar\phi(\alphastar) + (\alphastar^2+1)\Phi(-\alphastar)].$

Therefore, it boils down to controlling $|F_{22} - F_{23}|$. 
From now on, let us consider the scenario when $\alphastar < 2\alpha_0$. 
Recognizing that $\alphastar \goto \alpha_0 > 0$ as $\epsilon \goto 0$,
one sees that $\alphastar < 2\alpha_0$ holds as long as $\epsilon$ is sufficiently small. By virtual of the expression~\eqref{eq:f2}, we know that $F_{22}(0, \delta_0, \alphastar) = F_{23}(\nustar, \delta_0, \alphastar)$, and
\begin{align*}
\nabla_{\nu} F_{22}(\nu, \delta_0, \alphastar) = 2\nu\delta_0 \left[ \Phi(\alphastar-\sqrt{\delta_0}\nu)-\Phi(-\alphastar-\sqrt{\delta_0}\nu)\right],
\end{align*}
which combined with direct calculation gives 
\begin{align}
	0 < \nabla_{\nu} F_{22}(\nu, \delta_0, \alphastar)  < 2 \nu\delta_0 \Phi(2\alpha_0 - \sqrt{\delta_0}\nu).
\end{align}
With the help of the above relation, we can further obtain 
\begin{align*}
0 \leq F_{22} - F_{23} = \int_{0}^{\nustar} \nabla_{\nu} F_{22}(\nu, \delta_0, \alphastar) \mathrm{d} \nu
&\leq \int_{0}^{\nustar} 2 \nu\delta_0 \Phi(2\alpha_0 - \sqrt{\delta_0}\nu) \mathrm{d}\nu \\
&= 2 \int_{0}^{\nustar\sqrt{\delta_0}} \nu\Phi(2\alpha_0-\nu) \mathrm{d}\nu \leq 2\int_{0}^{+\infty} \nu\Phi(2 \alpha_0-\nu) \mathrm{d}\nu \leq C_{\delta_0},
\end{align*}
where $C_{\delta_0}$ is a constant that only depends on $\delta_0$. 
Therefore, we complete the proof of the second statement in \eqref{eq:epsilon-limit-2}.

\end{itemize}

\paragraph{Step 2: analysis of quantities in \eqref{eq:derivative-nu}.}
Equipped with the limiting values $\nustar$ and $\alphastar$, we are ready to analyze those terms that appear in the expression~\eqref{eq:derivative-nu}. 
Akin to the previous part, we also assume without loss of generality that $\alphastar < 2 \alpha_0$ throughout this part. 

\begin{itemize}
\item  \textbf{Limiting order of the numerator.}
Let us first consider the numerator $\nabla_{\delta}F_2\nabla_{\alpha}F_1 - \nabla_{\alpha}F_2\nabla_{\delta}F_1$, 
where the partial derivatives are evaluated at $(\nustar, \delta, \alphastar).$
With $|\phi(x)|\leq 1$ and $\nustar \leq M $ in mind, by virtue of the equations~\eqref{eq:f1-derivative}, we can easily verify that 
\begin{align*}
\left| \nabla_{\delta} F_1(\nu, \delta_0, \alpha) + 1\right| \leq  \frac{M}{\sqrt{\delta_0}} \epsilon = \sqrt{\frac{\text{SNR} \cdot \epsilon}{\delta_0}}; \quad \left| \nabla_{\alpha} F_1(\nu, \delta_0, \alpha) + 2 \p\right| \leq 4\epsilon
\end{align*}
with SNR defined in equation~\eqref{eq:defn-SNR}, 
which reveals that $\nabla_{\delta} F_1 \sim - 1$ and $\nabla_{\alpha} F_1 \sim - 2 \phi(\alphastar)$ as $\epsilon \to 0.$

We now turn to the $F_2$-related quantities. 
Invoking their explicit expressions as derived in Section~\ref{sec:expression-of-derivative} yields 
\begin{subequations}
\begin{align}
\label{eqn:beyond}
\left| \nabla_{\a} F_2 - \delta_0^{-1}\nabla_{\a}F_{23}\right| &= \epsilon \delta_0^{-1} \left| \nabla_{\a} F_{22} - \nabla_{\a} F_{23}\right| ;\\
\label{eqn:yyds}
\left| \nabla_{\delta} F_2 + \delta_0^{-2}F_{23}\right| 
&= \epsilon \delta_0^{-2} \left| F_{23}-F_{22} + \delta_0 \nabla_{\delta} F_{22} \right| .
\end{align}
\end{subequations}
We claim that the right-hand sides of both of the above equations vanish as $\epsilon \to 0.$
Taking these as given for the moment, we have $\nabla_{\a} F_2 \sim \delta_0^{-1}\nabla_{\a}F_{23}$ and $\nabla_{\delta} F_2 \sim - \delta_0^{-2}F_{23}$
which further ensure that 
\begin{align*}
\nabla_{\delta}F_2\nabla_{\alpha}F_1 - \nabla_{\alpha}F_2\nabla_{\delta}F_1
& ~\goto~ 2\delta_0^{-2}\phi(\alpha_0) F_{23} + \delta_0^{-1} \nabla_{\a} F_{23} \\
& ~\goto~ 2\delta_0^{-2}{\phi(\alpha_0)}\left[ -2\a_0\phi(\a_0) + (\a_0^2+1)\delta_0\right] - 2 \delta_0^{-1} \left[ 2 \phi(\a_0) - \a_0\delta_0\right],
\end{align*}
as claimed in the expression~\eqref{eqn:nu-ratio-numerator}.
Here, the last step uses $\a_0\defn -\Phi^{-1}(\delta_0/2)$. 
For any $0<\delta_0 < 1$, the limiting value is a {negative} number, due to the basic relation 
\begin{align*}
	\Phi(-\a_0) \in 
	\left[\phi(\a_0)\bigg(\frac{1}{\a_0}-\frac{1}{\a_0^3}\bigg),
	~ \phi(\a_0)\bigg(\frac{1}{\a_0}-\frac{1}{\a_0^3} + \frac{1}{\a_0^5}\bigg)\right].
\end{align*}

\paragraph*{Analysis of the expressions~\eqref{eqn:beyond} and \eqref{eqn:yyds}.} Combining $|\phi(x)|\leq 1$, $|\Phi(x)|\leq 1$  with the expressions of $\nabla_{\a} F_{22}$ and $\nabla_{\a} F_{23}$ in Section~\ref{sec:expression-of-derivative}, one can easily see that
\begin{align*}
	\left| \nabla_{\a} F_{22}\right| \leq 4(1+\alphastar) \leq 8(1+\a_0)
	\qquad \text{and }~
	\left| \nabla_{\a} F_{23}\right| \leq 4(1+\alphastar) \leq 8(1+\a_0), 
\end{align*}
given $\alphastar < 2\alpha_0$. It is thus clear that \eqref{eqn:beyond} is negligible when $\epsilon \to 0$. 

As for the relation~\eqref{eqn:yyds}, note that we have shown  in the previous part that $|F_{23}-F_{22}| \leq C_{\delta_0}$, where $C_{\delta_0}$ only depends on $\delta_0$. Additionally, it is clear that 
\begin{align*}
0 \leq \nabla_{\delta}F_{22} 
& =  \nustar^2\left[ \Phi(\alphastar-\nustar)-\Phi(-\alphastar-\nustar)\right]\\
&\leq \nustar^2 \Phi(2\alpha_0 - \sqrt{\delta_0} \nustar)  \leq \max_{\nu\geq 0} \left\lbrace \nu^2 \Phi(2\a_0 - \sqrt{\delta_0} \nu)\right\rbrace  
=: C_{\delta_0}'.
\end{align*}
Thus we conclude that $|\nabla_{\delta}F_{22}|\leq C_{\delta_0}'$.
Putting the above arguments together leads to the fact that $\left| F_{23}-F_{22} + \delta_0 \nabla_{\delta} F_{22} \right|$ is bounded by a universal constant determined by $\delta_0$ and is thus negligible when $\epsilon \to 0$. 

\item \textbf{Limiting order of the denominator.}
It remains to study the denominator $M(\nabla_{\nu}F_2\nabla_{\alpha}F_1 - \nabla_{\alpha}F_2\nabla_{\nu}F_1)$. 
To begin with, from the previous analysis, it is seen that $\nabla_{\alpha}F_2$ scales as $- 2 \delta_0^{-1} \left[ 2 \phi(\a_0) - \a_0\delta_0\right]$ and $\nabla_{\alpha}F_1 $ as $ -2\phi(\a_0)$ when taking $\epsilon \goto 0$. In what follows, we shall analyze $M\nabla_{\nu}F_2$ and $M\nabla_{\nu}F_1$ separately. 
For these two quantities, the following inequalities hold true 
\begin{align*}
	\left| M\nabla_{\nu}F_1\right| &= \sqrt{\text{SNR}\cdot\delta_0\epsilon} \left| \phi(\alphastar-\sqrt{\delta_0}\nustar) - \phi(\alphastar+\sqrt{\delta_0}\nustar)\right| \leq 2\sqrt{\text{SNR}\cdot \delta_0\epsilon}
\end{align*}
and 
\begin{align*}
\left| M\nabla_{\nu}F_2 - 2\frac{\nustar}{M}\right| 
 = \sqrt{\text{SNR}\cdot \delta_0^{-2}\epsilon}\left|\nabla_{\nu} F_{22}\right| 
 &= 2 \sqrt{\text{SNR}\cdot \epsilon} \nustar  \left[ \Phi(\alphastar-\sqrt{\delta_0}\nustar)-\Phi(-\alphastar-\sqrt{\delta_0}\nustar)\right] \\
 &\leq 2 \sqrt{\text{SNR}\cdot \epsilon} \nustar  \Phi(2\a_0-\sqrt{\delta_0}\nustar) \leq 2 \sqrt{\text{SNR} \cdot \epsilon} C_{\delta_0}'',
\end{align*}
with $C_{\delta_0}'' \defn \{\max_{\nu\geq 0}\nu  \Phi(2\a_0-\sqrt{\delta_0}\nu)\}$. 
Taken these collectively, as $\epsilon \goto 0$, we achieve
\begin{align*}
	M(\nabla_{\nu}F_2\nabla_{\alpha}F_1 - \nabla_{\alpha}F_2\nabla_{\nu}F_1) &\goto -4\phi(\a_0) \frac{\nustar}{M} 
	\goto 
	-4\phi(\a_0) \sqrt{1 - 2\delta_0^{-1}[-\a_0\phi(\a_0) + (\a_0^2+1)\Phi(-\a_0)]},
\end{align*}
where the last step relies on the relation~\eqref{eq:epsilon-limit-2}. 
\end{itemize}

\paragraph{Summary.}
Substituting the above parts on the numerator and the denominator into the expression~\eqref{eq:derivative-nu}, we conclude that when $\epsilon \goto 0$, one has 
\begin{align*}
\lim_{\epsilon \goto 0} \frac{\nustar'(\delta_0)}{M} < 0,
\end{align*}
which further indicates that the solution of $\nustar/M$ decreases with $\delta$ near $\forall \delta_0\in (0, 1)$, as long as $\epsilon$ is below a certain threshold.


\section{Auxiliary lemmas and details}
\label{sec:auxiliary}

\subsection{An example satisfying Assumption~\ref{assump:lambda}}
\label{sec:verify-assumption}

\begin{exas} 
Let $\{\lambda_t\}$ be a piece-wise constant sequence with
\begin{align}
	\lambda_t = \mu_k, \qquad \text{for } S_{k-1}+1\leq t\leq  S_k,
\end{align}
where the length of each piece $s_k \defn S_k - S_{k-1}$ and $S_{0}$ is set to be $0.$ 
Further choose $\mu_k = 1/ \max\left\lbrace \log k, 1\right\rbrace $ and $\Lambda_{S_k} = \sum_{i=1}^{S_k}\lambda_i = k^3$, $k\geq 1$.
\end{exas}

\begin{proof}
Let us now verify that this sequence satisfies Assumption~\ref{assump:lambda}. 
It is straightforward to validate the other inequalities, so we only present the proof for the second relation of~\eqref{eqn:assmption-lambda-summable2}.
In this case, when $S_{k} \leq t \leq S_{k+1}-1$, direct calculations yield 
\begin{align*}
l_t = \sum_{s=1}^t |\lambda_s-\lambda_{s+1}| \exp\left\lbrace - c [\Lambda_t - \Lambda_s] \right\rbrace & = \sum_{j=1}^k  |\mu_j-\mu_{j+1}| \exp\left\lbrace - c[\Lambda_t - \Lambda_{S_j}]\right\rbrace \\
& \leq \sum_{j=1}^k |\mu_j-\mu_{j+1}|  
\exp\left\lbrace - c k^3 + c j^3 \right\rbrace 
\exp\left\lbrace - c(t-S_k)\mu_k \right\rbrace .
\end{align*}
Further, it is easy to verify that $|\mu_j - \mu_{j+1}|\sim 1/(j\log^2 j)$, and as a result, 
\begin{align*}
\sqrt{l_t} \lesssim 
\exp\left\lbrace - \frac{c\mu_k(t-S_k)}{2}\right\rbrace \exp\left\lbrace -\frac{c k^3}{2}\right\rbrace \sqrt{\sum_{j=1}^k \frac{1}{j \log^2 j}\exp\left\lbrace c j^3\right\rbrace }  \overset{\text(*)}{\lesssim} \exp\left\lbrace - \frac{c\mu_k(t-S_k)}{2}\right\rbrace  \frac{1}{k^{3/2}\log k},
\end{align*}
where $\text(*)$ follows from the fact that 
\begin{align*}
	\int_1^k \frac{1}{x \log^2 x}\exp( c x^3 ) \mathrm{d}x \lesssim \frac{1}{k^3 \log^2 k}\exp( c k^3 ).
\end{align*}
Finally, we arrive at 
\begin{align*}
\sum_{t=1}^{+\infty} \sqrt{l_t} & \leq \sum_{k=1}^{+\infty} \left[ \sum_{t=S_k}^{S_{k+1}-1} \exp\left\lbrace - \frac{c\mu_k(t-S_k)}{2}\right\rbrace\right]   \frac{1}{k^{3/2}\log k} 
 \lesssim  \sum_{k=1}^{+\infty}  \frac{1}{1 - \exp(-c/2 \mu_k )}\frac{1}{k^{3/2}\log k} \lesssim \sum_{k=1}^{+\infty}  \frac{1}{k^{3/2}} < \infty.
\end{align*}
\end{proof}


\subsection{Auxiliary lemma for Gaussian distributions}

We collect some useful expressions about the standard Gaussian distribution, which shall be used multiple times in the proof of Section~\ref{sec:proof-risk-curve}. 
\begin{lem}\label{lem:formula}
The density function and cumulative density function $\phi(\cdot)$ and $\Phi(\cdot)$ of the standard Gaussian distribution obey the following relations: 
	\begin{subequations}
		\begin{align}
		\label{eq:auxillary-integral} & \int_{b}^{\infty}(z-a)^2\phi(z)\mathrm{d}z = (b-2a)\phi(b) + (a^2+1)[1-\Phi(b)];\\
		\label{eq:auxillary-moment-1}& \lim_{\a \goto +\infty}\frac{\a \Phi(-\a)}{\phi(\a)} = 1;\\
		\label{eq:auxillary-moment-2}& \lim_{\a \goto +\infty}\frac{\a^2\left[ \phi(\a)-\a\Phi(-\a)\right] }{\phi(\a)} = 1;\\
		\label{eq:auxillary-moment-3} & \lim_{\a \goto +\infty} \frac{\a^3\left[\a\phi(\a) - (\a^2+1)\Phi(-\a)\right] }{\phi(\a)} = -2.
		\end{align}
	\end{subequations}
\end{lem}
\begin{proof}[Proof of Lemma~\ref{lem:formula}.]
To verify the first expression, direct calculations yield 
	\begin{align*}
	\int_{b}^{\infty}(z-a)^2\phi(z)\mathrm{d}z & = \int_{b}^{\infty} z^2 \phi(z) \mathrm{d}z - 2a \int_{b}^{\infty}z\phi(z) \mathrm{d}z + a^2  \int_{b}^{\infty}\phi(z)\mathrm{d}z \\
	& = - \int_{b}^{\infty} z \mathrm{d}\phi(z) + 2a \int_{b}^{\infty} \mathrm{d}\phi(z) + a^2  \int_{b}^{\infty}\mathrm{d}\Phi(z) \\
	& = - \left[ - b\phi(b) - \int_b^{\infty} \phi(z) \mathrm{d} z\right] - 2a \phi(b) + a^2\left[1 - \Phi(b)\right]\\
	& = (b-2a)\phi(b) + (a^2+1)[1-\Phi(b)].
	\end{align*}
The last three equations can be verified similarly by use of L'Hôpital's rule; as an illustration, we provide the proof of the last equation here. 
By taking the derivatives of both the numerator and the denominator and using a little algebra, we obtain 
\begin{align*}
	\lim_{\a \goto +\infty} \frac{\a^3\left[\a\phi(\a) - (\a^2+1)\Phi(-\a)\right] }{\phi(\a)}  & = \lim_{\alpha \goto +\infty}\frac{4\a^3\phi(\a) - \a^5\phi(\a) + (\a^5+\a^3)\phi(\a) - (5\a^4+3\a^2)\Phi(-\a)}{-\a \phi(\a)} \\
	& =  \lim_{\alpha \goto +\infty}\frac{- 5\a^2\phi(\a) + (5\a^3+3\a)\Phi(-\a)}{\phi(\a)} \\
	& = 3 \lim_{\alpha \goto +\infty}\frac{\a\Phi(-\a)}{\phi(\a)} - 5 \lim_{\alpha \goto +\infty}\frac{\a^2\left[ \phi(\a) - \a \Phi(-\a)\right] }{\phi(\a)} = -2,
	\end{align*}
	which validates the relation~\eqref{eq:auxillary-moment-3}.
	Thus we complete the proof.
\end{proof}

\end{document}